\documentclass{amsart}

\usepackage{amsmath,amssymb}

\usepackage{epstopdf}
\usepackage{times}
\usepackage{amssymb}
\usepackage{mathrsfs}
\usepackage{graphicx}
\usepackage{rotating,subfigure}
\usepackage[flushleft]{threeparttable}
\usepackage{multirow}

\newcommand{\head}[1]{\textnormal{\textbf{#1}}}

\usepackage{star}

\usepackage[colorlinks,
linkcolor=blue,
anchorcolor=blue,
citecolor=blue
]{hyperref}

\newcommand \Sb{\mathbf{S}}

\newcommand*{\supp}{\mathrm{supp}}

\begin{document}

\title[Spectral properties of kernel-based sensor fusion]{On the spectral property of kernel-based sensor fusion algorithms of high dimensional data}     

\author{Xiucai Ding}
\address{Department of Mathematics, Duke University, Durham, NC, USA}

		\author{Hau-Tieng Wu}
\address{Department of Mathematics and Department of Statistical Science, Duke University, Durham, NC, USA; Mathematics Division, National Center for Theoretical Sciences, Taipei, Taiwan}
\email{hauwu@math.duke.edu}
\maketitle

\begin{abstract}
We apply local laws of random matrices and free probability theory to study the spectral properties of two kernel-based sensor fusion algorithms, nonparametric canonical correlation analysis (NCCA) and alternating diffusion (AD), for two simultaneously recorded high dimensional datasets under the null hypothesis.
The matrix of interest is the product of the kernel matrices associated with the databsets, which may not be diagonalizable in general. 
We prove that in the regime where dimensions of both random vectors are comparable to the sample size, if NCCA and AD are conducted using a smooth kernel function, then the first few nontrivial eigenvalues will converge to real deterministic values provided the datasets are independent Gaussian random vectors. {Toward the claimed result, we also provide a convergence rate of eigenvalues of a kernel affinity matrix.}
\end{abstract}

\section{Introduction}\label{sec:intro}
Recently, many scientific fields have advanced due to revolutionary sensor technologies, which provide more delicate information about the system of interest. However, a complicated system might not be fully captured by a single sensor, and a common practice is applying various sensors to understand the system from different angles. This leads to heterogeneous datasets with various sensors involved. 
A long lasting challenge is adequately quantifying the system of interest by assembling available information from a given heterogeneous dataset. This problem is usually understood as the {\em sensor fusion} problem  \cite{lahat2015multimodal}. 
The process toward sensor fusion is challenged on multiple fronts, and a lot of research efforts have been invested in the past decades. For example, the classical canonical correlation analysis (CCA) \cite{hotelling1936relations}, which quantifies the system by estimating the highly correlated {\em linear} affine subspaces from two datasets. Since then, many variations have been proposed, including utilizing principle component analysis to correct CCA \cite{Hwang2013}, and its extension to multiple datasets \cite{Horst1961}, to name but a few. 

In recent years, due to the growth of data volume and complexity, researchers pay more attention to the possible {\em nonlinear} structure of the system, and many studies have been conducted toward this goal. 
Specifically, the CCA idea is generalized by taking the nonlinearity structure into account and many algorithms have been proposed to recover the information commonly captured by different sensors. For example, 
in \cite{michaeli2016nonparametric},  the {\em nonparametric CCA} (NCCA) quantifies the information jointly captured by two sensors, {which is related to} the mutual information between two databases. A similar work, named nonlinear CCA, based on the reproducing kernel Hilbert space is proposed in \cite{Hardoon2004}.
In \cite{froyland2015dynamic,lederman2015learning,talmon2016latent,froyland2017dynamic,marshall2017time}, the spectral structure of the product of the transition matrices of different sensors is considered. Among those algorithms, the product of transition matrices of different sensors proposed in \cite{lederman2015learning,talmon2016latent} is called the {\em alternating diffusion} (AD), which is very similar to  NCCA {and has been applied to study the sleep dynamics via two electroencephalogram signals \cite{liu2020diffuse}}.
In \cite{shnitzer2018recovering}, the product structure is generalized by symmetrizing and anti-symmetrizing transition matrices of different sensors to enhance the common and different information captured by different sensors. The symmetrization and anti-symmetrization of transition matrices {lead to algorithms called} {\em symmetrized AD and anti-symmetrized AD} respectively.
We refer readers with interest to \cite{shnitzer2018recovering} for a more complete discussion of other sensor fusion algorithms.

On the theoretical side, there has been a rich literature discussing the statistical behavior of CCA \cite{anderson3rd}, including its behavior in the high dimensional setup \cite{bao2017}. Compared with the linear setup, there are fewer studies analyzing algorithms with the nonlinear setup. Under the assumption that the common information shared by two sensors can be parametrized by a metric space, it is shown in \cite{lederman2015learning} that AD faithfully recovers the metric of the metric space. Under the assumption that the common information shared by two sensors can be parametrized by a manifold, it is proved in \cite{talmon2016latent,shnitzer2018recovering} that AD and the symmetrized AD recover the Laplace-Beltrami operator of the manifold up to some deviations, while the anti-symmetrized AD asymptotically behaves like a first order differential operator. A similar setup and proof are also available in \cite{froyland2015dynamic}. In \cite{marshall2017time}, the multiplication of transition matrices from a sequence of manifold is analyzed.

Note that the basic assumption beyond sensor fusion algorithms is that the sensors do capture useful information of the system. However, to the best of our knowledge, this assumption is rarely confirmed. This fundamental question motivated this work -- if the datasets are high dimensional and only contain noise, what is the spectral behavior of kernel-based sensor fusion algorithms? Specifically, we are interested in the spectral behavior of kernel-based sensor fusion algorithms in the {\em null case}. While there are many algorithms available, we focus on two closely related algorithms, NCCA and AD. Before jumping into technical details, below we summarize NCCA and AD.

\subsection{Generic formula for NCCA and AD}\label{sec:alog}
 
In practice, we are provided two point clouds 
\begin{equation}
\mathcal{X}:=\{\xb_i\}_{i=1}^n\subset \mathbb{R}^{p_1},\,\,\mathcal{Y}:=\{\yb_i\}_{i=1}^n\subset \mathbb{R}^{p_2}\,, 
\end{equation}
where $n,p_1,p_2\in \mathbb{N}$ and for $i=1,\ldots,n$, $\xb_i$ and $\yb_i$ are observed simultaneously. 
The first step is constructing affinity matrices $\mathbf{W}_x,\,\mathbf{W}_y\in \mathbb{R}^{n\times n}$ from the point clouds by
\begin{equation}\label{Definition:Wx Wy matrix}
\mathbf{W}_x(i,j) := \exp\left( -\frac{\norm{\xb_i-\xb_j}^2_{\mathbb{R}^{p_1}}}{2\sigma_x^2}\right), \ \mathbf{W}_y(i,j) := \exp\left( -\frac{\norm{\yb_i-\yb_j}^2_{\mathbb{R}^{p_2}}}{2\sigma_y^2}\right)\,,
\end{equation} 
where $\sigma_x,\sigma_y>0$ are chosen by the user and $i,j=1,2,\ldots,n$.
The second step is normalizing $\mathbf{W}_x$ and $\mathbf{W}_y$ 
so that they become row stochastic; that is, constructing two row stochastic matrices $\mathbf{A}_x,\,\mathbf{A}_y\in \mathbb{R}^{n\times n}$ via 
\begin{equation}\label{Definition Ax Ay}
\mathbf{A}_x =\mathbf{D}_x^{-1}\mathbf{W}_x,\,\,\mathbf{A}_y=\mathbf{D}_y^{-1}\mathbf{W}_y\,,
\end{equation}
where $\mathbf{D}_x,\mathbf{D}_y\in \mathbb{R}^{n\times n}$ are diagonal matrices with the diagonal entries defined by 
\begin{equation}\label{Definition Dx Dy}
\mathbf{D}_x(i,i)=\sum_{j=1}^n\mathbf{W}_x(i,j),\,\,\mathbf{D}_y(i,i)=\sum_{j=1}^n\mathbf{W}_y(i,j) \,,
\end{equation} 
for $i=1,\ldots,n$. Usually we call $\mathbf{D}_x$ and $\mathbf{D}_y$ {\em degree matrices}.
Different algorithms utilize the above ingredients differently. In NCCA \cite{michaeli2016nonparametric}, construct a non-stochastic matrix 
\begin{equation}\label{Definition Sxy}
\mathbf{S}_{xy}:=\mathbf{A}_x\mathbf{A}^\top_y\in \mathbb{R}^{n\times n};
\end{equation} 
in AD \cite{lederman2015learning}, construct a new row stochastic matrix 
\begin{equation}\label{Definition Axy}
\mathbf{A}_{xy}:=\mathbf{A}_x\mathbf{A}_y\in \mathbb{R}^{n\times n}.
\end{equation}  
We call $\mathbf{S}_{xy}$ the {\em NCCA matrix} and $\mathbf{A}_{xy}$ the {\em AD matrix}. Finally, compute the eigenvalues and associated eigenvectors of $\mathbf{S}_{xy}$ and $\mathbf{A}_{xy}$. Or in some cases, the singular values and vectors can be considered. 
The key ingredient in these algorithms is the spectral structure of $\mathbf{S}_{xy}$ and $\mathbf{A}_{xy}$, which is the main focus of the current paper. 

{
\begin{remark}
We remark that in the algorithm of NCCA \cite{michaeli2016nonparametric}, it studies the singular values of $\Sb_{xy}.$ We will see later that $\Sb_{xy}$ can be well approximated by a positive definite matrix (c.f. see our explanation in Section \ref{sec:proofstrategy}). In this sense, it makes no difference to study the eigenvalues of $\Sb_{xy}.$  {Moreover, in practice we can choose more general kernels for NCCA and AD. See \cite{lederman2015learning,michaeli2016nonparametric} for details. When the kernel is not Gaussian, the relationship between NCCA and AD is no longer as transparent as that shown in \eqref{Definition Sxy} and \eqref{Definition Axy}. Indeed, in general NCCA may not be written in the product form.}
\end{remark}
}

\subsection{Mathematical model for the algorithm}\label{Section:Mathematical Model}

We are interested in the null model. Assume $\mathcal{X}$ and $\mathcal{Y}$ are independently and identically sampled from Gaussian random vectors; that is, $\xb_i\sim \mathcal{N}(\mathbf{0}_{p_1},p_1^{-1}\Ib_{p_1})$ and $\yb_i\sim \mathcal{N}(\mathbf{0}_{p_2},p_2^{-1}\Ib_{p_2})$,
where $\mathbf{0}_p \in \mathbb{R}^p$ is the vector with all entries 0, and $\Ib_p$ is the $p \times p$ identity matrix. 
Furthermore, we are interested in the high dimensional regime; i.e, we assume that $p_1=p_1(n)$, $p_2=p_2(n)$ and
 \begin{equation} \label{eq_boundc1c2}
 \gamma \leq \frac{n}{p_i}\leq \gamma^{-1}, \ i=1,2,
 \end{equation}
 for a fixed constant $0<\gamma \leq 1$ when $n$ is sufficiently large. Denote $c_1=c_1(n):=\frac{n}{p_1}$ and $c_2=c_2(n):=\frac{n}{p_2}$. In the present paper, we only need the assumption (\ref{eq_boundc1c2}) and do not require $c_1$ and $c_2$ to converge. This makes our result more adaptive to the datasets since many classical results depend on the assumption of convergence of $c_1$ and $c_2$ \cite{bai2010spectral}. Such relaxation is due to the recent developments of local laws in random matrix theory; see, for instance, \cite{Knowles2017}.

For the NCCA and AD algorithms, we consider more general kernel functions than the Gaussian kernel in \eqref{Definition:Wx Wy matrix}. Take a decreasing and twice continuously differentiable  function $f$ satisfying 
\begin{equation}\label{eq_fassum}
f(0)+2 f'(2)-f(2)>0\mbox{ and }f(2)>0\,,
\end{equation}
and construct $n \times n$ affinity matrices $\mathbf{W}_x,\mathbf{W}_y\in\mathbb{R}^{n\times n}$ with $f$ via
\begin{equation}\label{eq_defnmxmy}
\mathbf{W}_x(i,j):=f (\norm{\mathbf{x}_i-\mathbf{x}_j}_2^2 ) , \ \mathbf{W}_y(i,j):=f (\norm{\mathbf{y}_i-\mathbf{y}_j}_2^2 ) \,.
\end{equation}
{Note that the kernel in \eqref{Definition:Wx Wy matrix} is a special case. While it is possible to consider different bandwidths for $\mathcal{X}$ and $\mathcal{Y}$, to simplify the discussion, we focus on \eqref{eq_defnmxmy}.}
The definitions of degree matrices and row stochastic matrices follow those in \eqref{Definition Dx Dy} and \eqref{Definition Ax Ay} respectively.
Since the proof strategy and techniques are the same, we focus our discussion on the NCCA matrix $\mathbf{S}_{xy}$. The same discussion holds for the AD matrix and will be summarized in the end of the paper. 
{
\begin{remark}
Note that in (\ref{eq_fassum}), we need to impose concrete assumptions at the points $0$ and $2$ following \cite{elkaroui2010}. It will be seen later that our proof relies on the expansion of $f$ for pairwise distance 
around the concentration points. On one hand, for the diagonal entries, it is clear that $\| \xb_i-\xb_i \|_2^2=0, \ i=1,2,\cdots, n;$ on the other hand, for the off-diagonal entries, we have 
\begin{equation*}
\mathbb{E}\left(\| \xb_i-\xb_j \|_2^2 \right)=2, \ 1 \leq i \neq j \leq n. 
\end{equation*}    
\end{remark}
}

\subsection{Our contribution and relevant existing results}
To study the spectral property of $\mathbf{S}_{xy}$, the first difficulty we encounter is its non-diagonalizability due to the normalization procedure \eqref{Definition Ax Ay} and the product structure \eqref{Definition Sxy}.   With the smoothness assumption on $f$ and independence assumption on $\mathcal{X}$ and $\mathcal{Y}$, we can find a diagonalizable approximation 
$f^{-2}(2) \mathbf{W}_x \mathbf{W}_y $ for $n^2 \mathbf{S}_{xy}$. It thus suffices to study the spectrum of a product of two affinity matrices, which is the main focus of the current paper.\footnote{Note that by Bochner's theorem, $\mathbf{W}_x$ and $\mathbf{W}_y$ are positive definite when $f$ is positive definite; for example, when $f$ is Gaussian. In this case, $\mathbf{W}_x \mathbf{W}_y$ is similar to a symmetric matrix $\mathbf{W}_x^{1/2} \mathbf{W}_y \mathbf{W}_x^{1/2}$ and hence the spectrum exists and is real. However, the positive definiteness assumption does not hold for general kernels.} To this end, we need spectral properties of $\mathbf{W}_x$ and $\mathbf{W}_y$.

Take $\mathbf{W}_x$ as an example. The spectrum of $\mathbf{W}_x$ has been well studied in \cite{elkaroui2010} when $f$ is smooth, in \cite{cheng2013} for more general kernels and in \cite{el2015graph} when the kernel includes ``connection'' information {modeled by group actions}. In \cite{elkaroui2010}, it has been shown that when $f$ is smooth, the spectrum of kernel distance matrix is close to an isotropic shift of sample  Gram matrix $\Xb^\top \Xb$ with $\Xb\in \mathbb{R}^{p\times n}$, where the $i$-th column is $\xb_i$. It is well-known that the spectrum of the sample Gram matrix satisfies the celebrated Marchenko-Pastur (MP) law \cite{mp} and hence the spectrum of $\mathbf{W}_x$ follows a shifted  
MP law.  As a result, when $f$ is smooth, it is shown in \cite{elkaroui2010} that the spectrum of $\mathbf{W}_x$ can be asymptotically described by the finite rank additive perturbation of a sample Gram matrix.  
More recently, such results have been stated and extended using the free probability theory in  \cite{Fan2018}.  
However, none of the above results provide a convergent rate for the eigenvalues of $\mathbf{W}_x$. The first contribution of this paper is providing a convergent rate {for the eigenvalues of $\mathbf{W}_x$ under the null model} in the high dimensional regime using the technique of \emph{local laws} \cite{erdos2017dynamical} to study the \emph{rigidity} of eigenvalues for random matrices; see, for instance, \cite{bao2015, ding2018, ERDOS2012, Knowles2017}. Specifically, after showing that the \emph{Stieltjes} transforms of $\mathbf{W}_{x}$ and its deterministic counterpart are close, we translate this closeness to the closeness of eigenvalues of $\mathbf{W}_x$ and quantiles of some measure.  

Our second contribution is studying the eigenvalues of  $\mathbf{W}_x \mathbf{W}_y$ under the null model in the high dimensional regime. We first freeze one of the affinity matrices and replace it by its deterministic counterpart that happens with high probability. Similar ideas have been used in the study of  $F$-type matrices in \cite{han2016}. Since $\mathcal{X}$ and $\mathcal{Y}$ are independent, we show that with high probability, it suffices to study the model $\Ab\Ub \Bb \Ub^\top$, where $\Ab$ and $\Bb$ are deterministic symmetric matrices and $\Ub$ is an orthogonal matrix following the Haar distribution. The local law of this model has been recently studied in \cite{JHC} and its parallel additive model, $\Ab+\Ub \Bb \Ub^\top$, has been extensively studied in \cite{Bao2017cmp, bao20171, BAO2016}. By modifying their results, we prove the local laws and the rigidity of the eigenvalues for $\mathbf{W}_x \mathbf{W}_y$.      

\subsection{Notation and Organization}
To avoid repetition and make it convenient for readers, we summarize the notations and symbols in Table \ref{table_symbols}.
Throughout the paper, we use bold capital letters for matrices and bold lower-case letters for vectors. {When tracing the constant in the proof is not a concern,} we use $C$ as a generic constant, which may change from line to line{; otherwise we will make clear the difference from line to line}. We will systematically use two asymptotical notations. For two positive quantities $f_1$ and $f_2$, we use $f_1\lesssim f_2$ to indicate that there exists a constant $C$ such that $f_1\leq Cf_2$. If $f_1$ and $f_2$ are comparable; that is, $f_1\lesssim f_2$ and $f_2\lesssim f_1$, we use $f_1\sim f_2$. Another one is the stochastic domination, which we will introduce later in Definition \ref{defn_stochasdomi}.

The paper is organized as follows. In Section \ref{sec:preliminary}, we summarize relevant background knowledge. In Section \ref{sec:main}, we state our main results of this paper.
 In Section  \ref{sec:mainproof}, we prove the main results of this paper. We conclude the paper in Section \ref{sec:discussion}. A collection of preliminary lemmata  with a brief introduction of free probability theory, and all technical proofs are presented in Appendices \ref{sec:app:pre}--\ref{Section validation remark eq_firstthree}.

\begin{table*}[hbt!]
\centering
\begin{tabular}{@{}l*2{>{}l}%
  l<{}l@{}}
\toprule[1.5pt]
  & \multicolumn{1}{c}{\head{Symbol}}
  & \multicolumn{1}{c}{\head{Meaning}}\\
  \cmidrule(lr){2-2}\cmidrule(lr){3-4}
  \multirow{4}{*}{Setup} 
	& $\mathcal{X},\mathcal{Y}$ & point clouds $\mathcal{X}=\{\xb_i\}_{i=1}^n$ and $\mathcal{Y}=\{\yb_i\}_{i=1}^n$ & \\
 	& $p_1,p_2$ & dimensions of $\xb_i$ and $\yb_i$ respectively & \\
	& $c_1, c_2$ & aspect ratios, $c_1=n/p_1, c_2=n/p_2$ & \\
	& $\Xb,\Yb$ & data matrices associated with $\mathcal{X}$ and $\mathcal{Y}$ &   \\
  \cmidrule(lr){2-2}\cmidrule(lr){3-4}
  \multirow{4}{*}{Algorithm} 
	& $f$ & kernel function &   \\
	& $\mathbf{W}_x, \mathbf{W}_y$ & affinity matrices for $\mathcal{X}$ and $\mathcal{Y}$ respectively & \\
	& $\mathbf{D}_x, \mathbf{D}_y$ & degree matrices for $\mathbf{W}_x$ and $\mathbf{W}_y$ respectively & \\
	& $\mathbf{A}_x, \mathbf{A}_y$ & normalized affinity matrices& \\ 
  \cmidrule(lr){2-2}\cmidrule(lr){3-4}
    \multirow{4}{*}{General} 
	& $\|\Ab\|$ & the operator norm of $\Ab$& \\
	& $\circ$ & Hadamard product & \\
	& $\mathsf{S}(\alpha,\tau)$ & spectral parameter set with parameters $\alpha>0$ and $\tau>0$ defined in \eqref{eq_sprime} &\\
	& $\gamma_\mu(j)$ & the $j$-th typical location of the probability measure $\mu$ (Definition \ref{Definition typical locations}) &\\
  \cmidrule(lr){2-2}\cmidrule(lr){3-4}
    \multirow{4}{*}{Transforms} 
    &$m_\Ab(z)$ & Stieltjes transform of the ESD (equation (\ref{eq_defn_esd})) of the matrix $\Ab$ &\\
   & $m_\mu(z)$ & Stieltjes transform of the probability measure $\mu$ & \\
   & $S_\mu(z)$ & $S$-transform of the probability measure $\mu$ & \\
   & $R_\mu(z)$ & $R$-transform of the probability measure $\mu$ & \\
      & $M_\mu(z)$ & $M$-transform of the probability measure $\mu$ & \\
      \cmidrule(lr){2-2}\cmidrule(lr){3-4}
  \multirow{2}{*}{Free probability} 
  		& $\mu_{\alpha} \boxtimes \mu_{\beta} $ & Free multiplicative convolution  of $\mu_{\alpha}$ and $\mu_{\beta}$ & \bfseries \\
        		& $\mu_{\alpha} \boxplus \mu_{\beta} $ & Free additive convolution  of $\mu_{\alpha}$ and $\mu_{\beta}$ & \bfseries \\
  \cmidrule(lr){2-2}\cmidrule(lr){3-4}
 \multirow{5}{*}{Measures} 
  	&  $\mu_\Ab$  &  ESD of the matrix $\Ab$& \\
	&  $\mu_{c,\sigma^2}$  & MP distribution with parameters $c$ and $\sigma^2$ defined in \eqref{eq_mp}& \\
	& $\mu_{c,\Sigma}$ &  Generalized MP distribution with parameters $c$ and $\Sigma$ defined in \eqref{eq_defnf} & \\
  	&  $\nu_x,\nu_y$  &  Limiting spectral distributions of $\breve{\mathbf{K}}_x$ and $\breve{\mathbf{K}}_y$ respectively defined in \eqref{Definition special measure for kernel matrix} \\
  	&  $\nu_{xy}$  & $\nu_x \boxtimes \nu_y$  \\
   \cmidrule(lr){2-2}\cmidrule(lr){3-4}
\multirow{1}{*}{Constants}
         & $\varsigma$ & $f(0)+2 f'(2)-f(2)$ & \slshape &\\
   \cmidrule(lr){2-2}\cmidrule(lr){3-4}
\multirow{6}{*}{Matrices}
 	&  $\boldsymbol{1}_n$ & a $n$-dim vector with all entries being $1$& \\
 	& $\mathbf{I}_{n}$ & the $n\times n$ identity matrix & \\
        & $\psi$ & $\psi=(\psi(1),\ldots, \psi(n))^\top$ and $\psi(i):=\norm{\xb_i}_2^2-1$ & \slshape &\\
         & $\mathrm{Sh}_0$ & $f(2) \mathbf{1} \mathbf{1}^\top$ & \slshape &\\
         & $\mathrm{Sh}_1$ & $f'(2)[\mathbf{1} \psi^\top+\psi \mathbf{1}^\top ]$ & \slshape &\\
         & $\mathrm{Sh}_2$ & $\frac{f''(2)}{2}\left[ \mathbf{1}( \psi \circ \psi)^\top+(\psi \circ \psi) \mathbf{1}^\top+2 \psi \psi^\top \right]$ & \slshape &\\
\bottomrule[1.5pt]
\end{tabular}
\caption{Summary of commonly used symbols}\label{table_symbols}
\end{table*}

\section{Background}\label{sec:preliminary}

For any $n \times n$ matrix $\Hb$, its {\em empirical spectral distribution (ESD)} is defined as
\begin{equation}\label{eq_defn_esd}
\mu_{\Hb}=\frac{1}{n} \sum_{i=1}^n \delta_{\lambda_i}\,,
\end{equation} 
where $\lambda_i$, $i=1,2,\ldots,n$, are the eigenvalues of $\Hb$ and $\delta_{\lambda_i}$ is a Dirac delta measure supported on $\lambda_i\in \mathbb{C}$. Clearly, when $\Hb$ is Hermitian, $\mu_{\Hb}$ is a measure supported on $\mathbb{R}$. 
Denote $\mathbb{C}^+$ as the upper complex half plane.  The {\em Green function} of $\Hb$ is defined as
\begin{equation}
G_{\Hb}(z)=(\Hb-z)^{-1}, \ z \in \mathbb{C}^+\,. 
\end{equation}
The local convergence can be best formulated using the Green function.

\begin{definition}[Stieltjes transform]
The {\em Stieltjes transform} of a probability measure $\mu$ supported on $\mathbb{R}$ is defined as 
\begin{equation}
m_{\mu}(z)=\int \frac{1}{\lambda-z} \mu(d\lambda), \ z \in \mathbb{C}^{+}\,.
\end{equation}
\end{definition}
It is clear that for the ESD of $\Hb$, its Stieljes transform satisfies
\begin{equation}\label{eq_stieltjes}
m_{\mu_\Hb}(z)=\int \frac{1}{\lambda-z} \mu_{\Hb}(d\lambda)=\frac{1}{n}\text{tr}[G_{\Hb}(z)]\,.
\end{equation}
To simplify notation, when there is no danger of confusion, we denote $m_{\Hb}:=m_{\mu_\Hb}$.
Note that we have a trivial but useful bound 
\begin{equation}\label{eq_stitrivial}
|m_{\Hb}(z)| \leq \frac{1}{n} \sum_{i=1}^n \frac{1}{|\lambda_i-z|} \leq \frac{1}{\operatorname{Im}z}\,.
\end{equation}

\begin{definition}[$S$-transform and $R$-transform]
Take a Borel probability measure $\mu$ supported on $\mathbb{R}$. The \emph{$S$-transform} of $\mu$ is defined by
\begin{equation}
S_{\mu}(z):=-\frac{z+1}{z} \eta_{\mu}^{-1}(z+1)\,,
\end{equation}
where $z \in \mathbb{C}^+$ and
\begin{equation}
\eta_{\mu}(z)=\frac{m_{\mu}(-1/z)}{z}\,.
\end{equation}
The {\em $R$-transform} of $\mu$ is defined by (see \cite[Section 2.5.1]{ben2011} or \cite[Section 2]{RD})
\begin{equation}
R_{\mu}(z)=m_{\mu}(-z)-\frac{1}{z}\,,
\end{equation}
where $z \in \mathbb{C}^+$.
\end{definition}

It can be checked that there exists an one-to-one correspondence between the $S$-transform and the Stieltjes transform \cite{CDM}.
The $R$-transform is useful in characterizing the law of addition of independent random matrices whereas the $S$-transform is useful in characterizing the law of product of independent random matrices \cite[Section 2.5.2]{ben2011}.

\subsection{``Macroscopic'' asymptotics of eigenvalues}
For $\sigma>0$, let $\Zb =\sigma \Xb\in \mathbb{R}^{p_1 \times n}$  be a rectangular matrix whose columns are $\sigma \mathbf{x}_i$, $i=1,2,\ldots,n$. 
It is well-known that the ESD of $\Zb^\top \Zb$ has a deterministic limiting measure, denoted as $\mu_{c_1,\sigma^2}$, that follows the Marchenko-Pastur (MP) law \cite{mp} and satisfies
\begin{equation}\label{eq_mp}
\mu_{c_1,\sigma^2}(I)=
(1-c_1)_{+} \chi_{I}(0)+\zeta_{c_1,\sigma^2}(I)\,,
\end{equation}
for any measurable set $I \subset \mathbb{R}$, where $\chi_I$ is the indicator function and $(a)_+:=0$ when $a\leq 0$ and $(a)_+:=a$ when $a>0$,
\begin{equation}
d \zeta_{c_1,\sigma^2}(x)=\frac{1}{2\pi \sigma^2} \frac{\sqrt{(\lambda_+-x)(x-\lambda_-)}}{c_1x} dx\,, 
\end{equation}
$\lambda_{+}=\sigma^2(1+ \sqrt{c_1})^2$ and $\lambda_{-}=\sigma^2 (1-  \sqrt{c_1})^2$. Moreover, the Stieltjes transform of $\mu_{c_1,\sigma^2}$  satisfies the following self-consistent equation (for instance, see {\cite[Section 2]{alex2014})}:
\begin{equation}\label{eq_selfconsistent}
m_{\mu_{c_1,\sigma^2}}(z)+\frac{1}{z-(1-c_1)+zc_1 m_{\mu_{c_1,\sigma^2}}(z)}=0\,.
\end{equation}
{
\begin{remark}
In the standard literature of Random Matrix Theory, for instance \cite{bai2010spectral}, (\ref{eq_selfconsistent}) requires that $c_1$ converges to some fixed constant value. However, with the recent development using dynamic approach \cite{erdos2017dynamical}, we only need the boundedness of $c_1$, for instance, see \cite[Section 2]{alex2014}.  
\end{remark}
}

The above results have been extended to the case when the random vector has a general covariance structure; that is, the covariance matrix $\Sigma_{p_1}$ is not a scaled identity matrix. Denote the eigenvalues of $\Sigma_{p_1}$ by $\{\lambda_1, \ldots, \lambda_{p_1}\}$. Then, under a regularity condition of $\{\lambda_1, \ldots, \lambda_{p_1}\}$ (see Assumption \ref{assu_regularity} in  Appendix \ref{sec:app:pre} for details), the limiting spectral distribution (LSD) of $(\Sigma_{p_1}^{1/2} \Xb)^\top \Sigma_{p_1}^{1/2}\Xb$, denoted as $\mu_{c_1,\texttt{gMP}}$, follows the {\em generalized MP law}, which is determined by the unique solution of $g(m_{\Sigma_{p_1}}(z))=z$ \cite[Section 3]{bai2010spectral}  with $g$ defined as
 \begin{equation}\label{eq_defnf}
g(z)=-\frac{1}{z}+c_1 \sum_{i=1}^{p_1} \frac{p_1^{-1}}{z+\lambda_i^{-1}}\,.
\end{equation}  
We mention that the spectrum of $(\Sigma_{p_1}^{1/2} \Xb)^\top  \Sigma_{p_1}^{1/2}\Xb$ can be understood using the free probability theory. We only discuss the essential concepts for the statements of our main results and leave details to Appendix \ref{sec:appen:freeprob}.
Let $\Ab_n$ and $\Bb_n$ be independent $n \times n$ symmetric (or Hermitian) positive-definite random matrices
that are invariant, in law, by conjugation by any orthogonal (or unitary) matrix. Assume $\mu_{\Ab_n} \rightarrow \mu_A, \ \mu_{\Bb_n} \rightarrow \mu_B$ when $n \rightarrow \infty$, where $\mu_A$ and $\mu_B$ are two Borel probability measures. 
Denote $\mu_{\Ab_n} \boxtimes \mu_{\Bb_n}:=\mu_{\Ab_n \times \Bb_n}$. 
If $\lim_{n\to \infty}\mu_{\Ab_n\times \Bb_n}$ exists, we denote  
\begin{equation}
\mu_A \boxtimes \mu_B:=\lim_{n\to \infty}\mu_{\Ab_n\times \Bb_n}=\lim_{n\to \infty}\mu_{\Ab_n} \boxtimes \mu_{\Bb_n}\,.
\end{equation}
The left-hand side of the above equation is known as the {\em free multiplicative convolution} of $\mu_A$ and $\mu_B$, and it can be characterized using the $S$-transform via \cite[Chapter 22.5.2]{handbook}
\begin{equation} \label{eq_multiplicative}
S_{\mu_A \boxtimes \mu_B}(z)=S_{\mu_A}(z) S_{\mu_B}(z).
\end{equation}
As a concrete example, assuming that $\lim_{n\to\infty}\mu_{\Sigma_{p_1}}$ exists, then the LSD of $(\Sigma_{p_1}^{1/2} \Xb)^\top \Sigma_{p_1}^{1/2}\Xb $ can also be written as 
$(\lim_{n\to\infty}\mu_{\Sigma_{p_1}}) \boxtimes \mu_{c_1,1}$.

In parallel, if $\lim_{n\to \infty}\mu_{A_n+B_n}$ exists, we denote
\begin{equation*}
\mu_A \boxplus \mu_B:=\lim_{n\to \infty}\mu_{A_n +B_n}\,,
\end{equation*}
which is known as the {\em free additive convolution} for the summation of  $\mu_A$ and $\mu_B$.
The probability measure $\mu_A \boxplus \mu_B$ can be characterized in terms of the $R$-transform as \cite[Chapter 22.5.1]{handbook}
\begin{equation*}
R_{\mu_A \boxplus \mu_B}(z)=R_{\mu_A}(z)+R_{\mu_B}(z)\,. 
\end{equation*}

\subsection{``Microscopic'' asymptotics of eigenvalues}
The aforementioned spectral analysis provides only the ``macroscopic'' asymptotic of eigenvalues in the sense of quantifying the overall distribution of eigenvalues.  To obtain  information for each individual eigenvalue, we need to study the convergence all the way down to the scale of eigenvalue spacing of each individual eigenvalue.
To this end, we adopt the notation of stochastic domination, which is introduced in \cite[Definition 2.1]{alex2014} and is commonly used in the literature related to this work. It provides a way of making precise statements of the form ``$\mathsf{X}$ is bounded with high probability by $\mathsf{Y}$ up to small powers of $n$''.
 \begin{definition} [Stochastic domination]\label{defn_stochasdomi} Let
 \begin{equation*}
 \mathsf{X}=\big\{\mathsf{X}^{(n)}(u):  n \in \mathbb{N}, \ u \in \mathsf{U}^{(n)}\big\}, \   \mathsf{Y}=\big\{\mathsf{Y}^{(n)}(u):  n \in \mathbb{N}, \ u \in \mathsf{U}^{(n)}\big\},
 \end{equation*}
be two families of nonnegative random variables, where $\mathsf{U}^{(n)}$ is a possibly $n$-dependent parameter set. We say that $\mathsf{X}$ is {\em stochastically dominated} by $\mathsf{Y}$, uniformly in the parameter $u$, if for all small $\epsilon>0$ and large $ D>0$, we have 
\begin{equation*}
\sup_{u \in \mathsf{U}^{(n)}} \mathbb{P} \Big( \mathsf{X}^{(n)}(u)>n^{\epsilon}\mathsf{Y}^{(n)}(u) \Big) \leq n^{- D},
\end{equation*}   
for a sufficiently large $n \geq  n_0(\epsilon, D)$. We use the notation $\mathsf{X}=O_{\prec}(\mathsf{Y})$ if $\mathsf{X}$ is stochastically dominated by $\mathsf{Y}$, uniformly in $u$. 
 \end{definition}

For any (sufficiently small) constant { $0<\tau<1$} and $0 \leq \alpha \leq 1$, denote the {\em spectral parameter set} with parameter $\tau$ and $\alpha$ as
\begin{equation}\label{eq_sprime}
\mathsf{S}(\alpha,\tau):=\{ z=E+\mathrm{i} \eta : \tau \leq E \leq \tau^{-1}, \ n^{-\alpha+\tau} \leq \eta \leq \tau^{-1}\}.
\end{equation}
Throughout this paper, the stochastic domination will always be uniform in
all parameters, including matrix indices and the spectral parameter $z$ in the spectral parameter set.
It can be shown that under suitable regularity conditions on the covariance matrix $\Sigma_{p_1}$ (c.f. Assumption \ref{assu_regularity}), we have the \emph{averaged local law} \cite{Knowles2017}. {Specifically, for $z=E+\mathrm{i} \eta \in \mathsf{S}(1,\tau),$ we have}
\begin{equation}\label{eq_m1n}
m_{(\Sigma_{p_1}^{1/2} \Xb)^\top\Sigma_{p_1}^{1/2}\Xb}(z)=m_{\mu_{c_1,\texttt{gMP}}}(z)+O_{\prec}\Big(\frac{1}{n \eta}\Big)\, .
\end{equation}

For the multiplication of random matrices with Haar unitary (orthogonal) conjugation, i.e, 
\begin{equation}\label{model H=AUBUt}
\Hb=\Ab \Ub \Bb \Ub^\top,
\end{equation}
where $\Ub$ is a Haar distributed random unitary (orthogonal) matrix, the local law has also been studied in \cite{JHC,DJ}, and for {$z=E+\mathrm{i}\eta \in \mathsf{S}(1,\tau)$,} we have 
\begin{equation}
m_{\Hb}(z)=m_{\mu_\Ab \boxtimes \mu_\Bb}(z)+O_{\prec}\Big(\frac{1}{n \eta}\Big). 
\end{equation}
In parallel, although we will not need it, we mention that for the addition of random matrices with Haar unitary (orthogonal) conjugation, i.e, 
$\tilde\Hb=\Ab+\Ub \Bb\Ub^\top$, where $\Ub$ is a Haar distributed random unitary (orthogonal) matrix,
under mild conditions, the local law has been established \cite{Bao2017cmp, bao20171}. Specifically, for {$z=E+\mathrm{i} \eta \in \mathsf{S}(1,\tau)$},  $m_{\tilde\Hb}(z)=m_{\mu_\Ab \boxplus \mu_\Bb}(z)+O_{\prec}\Big(\frac{1}{n \eta}\Big)$.

\section{Main results}\label{sec:main}

We now state our main results under the mathematical model described in Section \ref{Section:Mathematical Model} with the assumptions 
\eqref{eq_boundc1c2} and \eqref{eq_fassum} held true. We prepare some notations and quantities.
Denote 
\begin{equation}\label{eq_mu}
\varsigma:=f(0)+2 f'(2)-f(2). 
\end{equation}
For $c\in \mathbb{R}$, denote $T_c$ to be the shifting operator that shifts a probability measure:
\begin{equation*}
T_c\mu(I):=\mu(I-c)\,,
\end{equation*}
where $I\subset\mathbb{R}$ is a measurable subset and $\mu$ is a probability measure defined on $\mathbb{R}$.
\begin{definition}[Typical locations]\label{Definition typical locations}
For a given probability measure $\mu$ and $n\in \mathbb{N}$, define the $j$-th \emph{typical location} of  $\mu$ as $\gamma_\mu(j)$; that is,  
\begin{equation}\label{eq_typical}
\int_{\gamma_\mu(j)}^{\infty} \mu(dx)=\frac{j-1/2}{n}\,,
\end{equation}
where $j=1,\ldots,n$. $\gamma_\mu(j)$ is also called the $j/n$-quantile of $\mu$.
\end{definition}

\subsection{Convergence rate of eigenvalues of a kernel affinity matrix}\label{Subsection:eigenvalue convergence rate of kernel matrix}

As discussed in Section \ref{sec:intro}, the affinity matrix $\mathbf{W}_x$ can be approximated by a shifted sample Gram matrix with finite rank perturbations. We first remove such perturbations. Denote
\begin{equation}\label{eq_defntildem}
\tilde{\mathbf{W}}_x:=\mathbf{W}_x-\mathrm{Sh}_0-\mathrm{Sh}_1-\mathrm{Sh}_2,
\end{equation}
where the shifts $\mathrm{Sh}_i\in \mathbb{R}^{n\times n}$, $i=0,1,2$,  are finite rank matrices defined as 
\begin{align}
& \mathrm{Sh}_0:= (f(2)+\frac{4}{p_1}) \mathbf{1} \mathbf{1}^\top,\nonumber \\
& \mathrm{Sh}_1:=f'(2)[\mathbf{1} \psi^\top+\psi \mathbf{1}^\top ],  \label{eq_sh0} \\
& \mathrm{Sh}_2:= \frac{f''(2)}{2}\left[ \mathbf{1}( \psi \circ \psi)^\top+(\psi \circ \psi) \mathbf{1}^\top+2 \psi \psi^\top \right], \nonumber
\end{align}
where $\circ$ is the Hadamard product, $\psi=(\psi(1),\ldots, \psi(n))^\top$, and $\psi(i):=\norm{\xb_i}_2^2-1$. 
Denote 
\begin{align}\label{Definition special measure for kernel matrix}
\nu_x:=T_{\varsigma}\mu_{c_1,-2f'(2)},\,\,\,
\nu_y:=T_{\varsigma}\mu_{c_2,-2f'(2)},\,\,\,
\nu_{xy}:=\nu_x \boxtimes \nu_y.
\end{align}
\begin{theorem}\label{lem_rigiditydistance}  
For some constants $C, C_1>0$ and any small fixed constants $\epsilon>0$ and $\vartheta>\epsilon$, when $n$ is sufficiently large, with probability at least $1-O(n^{-1/2})$, we have  
\begin{equation*}
\Big |\lambda_i(\tilde{\mathbf{W}}_x)-\gamma_{\nu_x}(i) \Big | \leq 
\begin{cases}
C n^{-1/9+2\vartheta} & 1\leq i \leq C_1 n^{5/6+3\vartheta/2} \,; \\
C n^{1/12+\epsilon } i^{-1/3} &  C_1 n^{5/6+3\vartheta/2} \leq i \leq n/2 \, .
\end{cases}
\end{equation*}
Similar results hold true for $i \geq n/2$. {Specifically, with probability at least $1-O(n^{-1/2})$, we have  
\begin{equation} \label{eq_anotherhalf}
\Big |\lambda_i(\tilde{\mathbf{W}}_x)-\gamma_{\nu_x}(i) \Big | \leq 
\begin{cases}
C n^{1/12+\epsilon } i^{-1/3} &   n/2+1 \leq i \leq n- C_1 n^{5/6+3\vartheta/2}  \,; \\
C n^{-1/9+2\vartheta} & n- C_1 n^{5/6+3\vartheta/2}+1 \leq i \leq n  \, .
\end{cases}
\end{equation}

}
\end{theorem}
We then have the eigenvalue rigidity of a kernel affinity matrix.
\begin{corollary}[Eigenvalue rigidity of kernel affinity matrix] \label{coro_lesspertub} 
For some constants $C, C_1>0$ and small fixed constants $\epsilon>0$ and $\vartheta>\epsilon$, when $n$ is sufficiently large, with probability at least $1-O(n^{-1/2})$, we have 
\begin{equation*}
\Big |\lambda_{i+1}(\mathbf{W}_x)-\gamma_{\nu_x}(i) \Big | \leq 
\begin{cases}
C n^{-1/9+2 \vartheta} & 1\leq i \leq C_1 n^{5/6+3\vartheta /2} \,; \\
C n^{1/12+\epsilon} i^{-1/3} &  C_1 n^{5/6+3\vartheta/2} \leq i \leq  n/2\,. 
\end{cases}
\end{equation*}
Similar results hold true for $i \geq n/2$ {in the form of (\ref{eq_anotherhalf}).}
\end{corollary}

{
\begin{remark} \label{remk_rates}
Since our proof relies on an entry-wisely expansion, we mention that the above rates of the eigenvalue rigidity may not be optimal. To identity the optimal rate, we provide the following lower bound for the convergent rate. With probability at least $1-O(n^{-1/2}),$ for some constant $C>0,$ we have 
\begin{equation}\label{eq_rates}
 \left| \lambda_{i+1}(\Wb_x)-\gamma_{\nu_x}(i) \right| \geq \frac{C \log n}{n}, \ 1 \leq i \leq n-1. 
\end{equation}
See Section \ref{Section validation remark eq_firstthree} for the proof.
We take an edge eigenvalue, i.e., $\lambda_{i+1}(\Wb_x)$ for $1 \leq i \leq C_1 n^{5/6+3 \vartheta/2},$ as an example for further discussion. By Theorem \ref{lem_rigiditydistance} and \eqref{eq_rates}, the optimal rate for $\left| \lambda_{i+1}(\Wb_x)-\gamma_{\nu_x}(i) \right| $ should lie in the interval $[C' n^{-1}\log n,C n^{-1/9+2 \vartheta}]$ with probability at least $1-O(n^{-1/2})$. Inspired by the optimal rates of some other models in Random Matrix Theory literature (c.f., see Lemma \ref{lem_rigidity}), we hypothesize that the optimal rate should be of order $n^{-2/3}.$ While finding the optimal convergence rate is out of the scope of this work, we perform some Monte-Carlo simulations to empirically study the optimal rates when $i=1.$ In our simulations below, we take $f(x)=\exp(-\frac{x}{2})$.  Consequently, we compute that $\gamma_{\nu_x}(1)=e^{-1}(1+\sqrt{n/p})^2+1-2e^{-1}.$ In Figure \ref{fig_conjecture}, we record our results of $n^{2/3}(\lambda_2(\Wb_x)-\gamma_{\nu_x}(1))$ in the y-axis for a variety of $p$ in the x-axis. In this simulation study, for simplicity, we consider $c=0.5,1,2$, where we recall that $c=p/n$. Empirically, we conclude from Figure \ref{fig_conjecture} that $n^{2/3}(\lambda_2(\Wb_x)-\gamma_{\nu_x}(1))$ may depend on $p/n,$ but it is independent of $p$. 
This empirical evidence supports our hypothesis that the optimal rates of Theorem \ref{lem_rigiditydistance} and Corollary \ref{coro_lesspertub} should be similar to those in Lemma \ref{lem_rigidity}. We will pursue this direction in our future work. 
\begin{figure}[ht]
\centering
\includegraphics[width=0.4\textwidth]{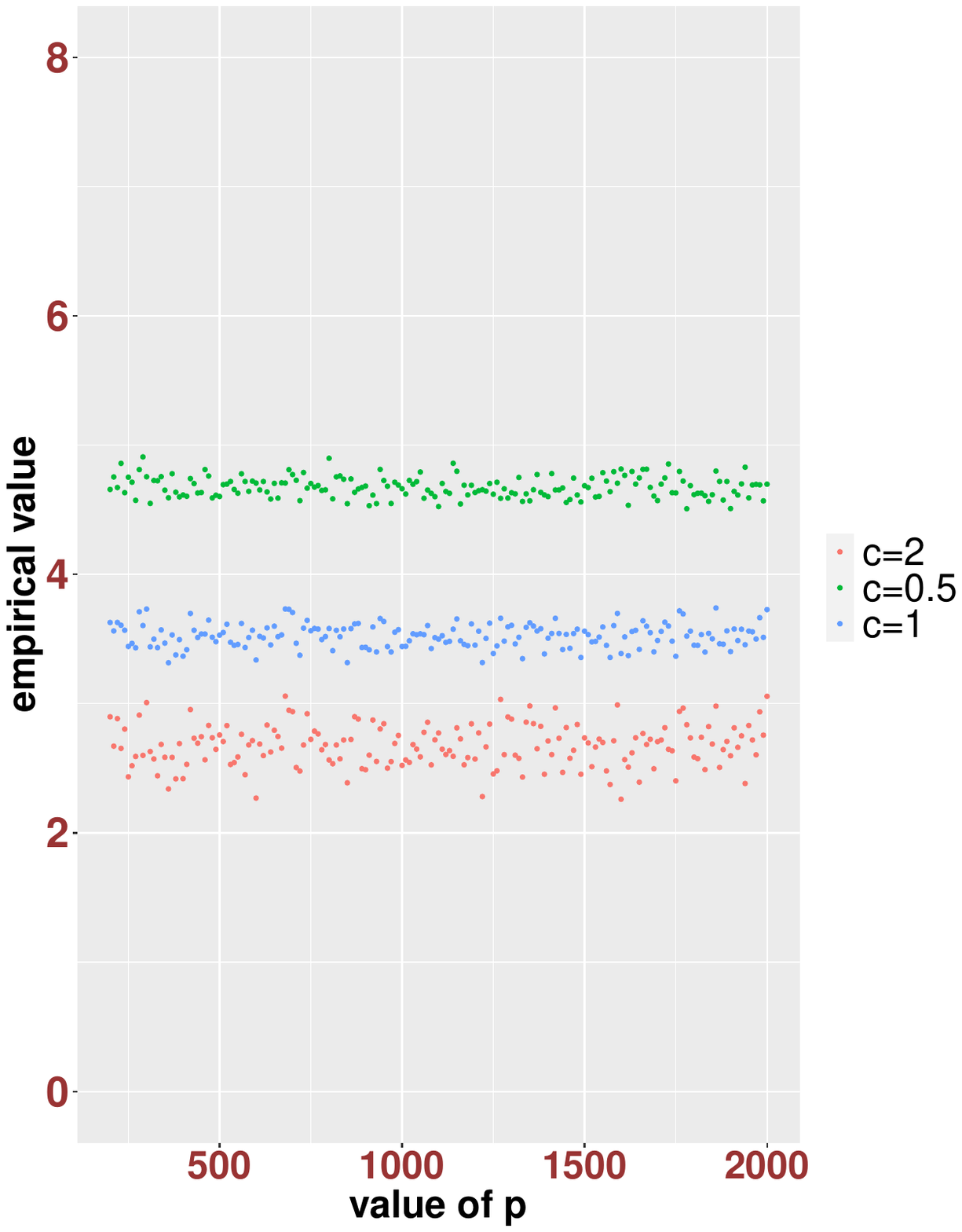}
\caption{Empirical values of $n^{2/3}(\lambda_{2}(\Wb_x)-\gamma_{\nu_x}(1)).$ For each $p,$ we run 1,000 simulations. }\label{fig_conjecture}
\end{figure}    
\end{remark}

}

\subsection{Eigenvalues of the NCCA matrix}\label{Subsection:eigenvalue convergence rate of NCCA matrix}

Based on the affinity matrix analysis, we have the following main theorem.
\begin{theorem}[Edge eigenvalues of the NCCA and AD matrices]\label{thm_edgevalue}  
For any fixed $L \in \mathbb{N}$ and some constant $C>0$, with probability at least $1-O(n^{-1/2})$ we have 
\begin{equation*}
\left| \lambda_i(n^2 \mathbf{S}_{xy})-\frac{\gamma_{\nu_{xy}}(1)}{f^2(2)} \right| \leq  Cn^{-1/9},  
\end{equation*}
for all $3 \leq i \leq L$  when $n$ is sufficiently large. The same result holds for the AD matrix $\mathbf{A}_{xy}$.
\end{theorem}
Note that in general the eigenvalues of $n^2 \mathbf{S}_{xy}$ are not necessary real and $\nu_{xy}$ is defined on the real axis. Theorem \ref{thm_edgevalue} implies that the imaginary part of the $i$-th eigenvalue of $n^2 \mathbf{S}_{xy}$ is negligible for all $3\leq i \leq L$ when $n$ is sufficiently large. 

{
\begin{remark} \label{eq_firstthree}
Due to the independence and the structure of $n^2\mathbf{S}_{xy}$, we claim that the first eigenvalue is with high probability a trivial one; that is, for some constant $C>0$, with probability at least $1-O(n^{-1/2})$
\begin{equation} \label{eq_first}
|\lambda_1(\mathbf{S}_{xy}) - 1|\leq
 C n^{-1/2}.
\end{equation}
Similar discussion holds for the AD matrix $\mathbf{A}_{xy}$, while for the AD matrix, the first eigenvalue is always a trivial one, since $\mathbf{A}_x \mathbf{A}_y \mathbf{1}=\mathbf{1}$. 
Finally, we remark that due to the nature of the perturbation argument in the proof of Theorem \ref{thm_edgevalue}, we cannot characterize the convergent limit of the second eigenvalue. We conjecture that $\lambda_2(n^2 \mathbf{S}_{xy})$ will also converge to $\gamma_{\nu_{xy}}(1)$ with high probability. We will study this problem in the future work. 
\end{remark}
}

{
The AD algorithm has a natural extension to multiple sensors \cite{Lederman_Talmon_Wu_Lo_Coifman:2015}. Suppose that we can observe $r>2$ point clouds, 
\begin{equation*}
\mathcal{X}_k:=\{\xb^k_i\}_{i=1}^n \subset \mathbb{R}^{p_k}, \ 1 \leq k \leq r.  
\end{equation*} 
The AD in Section \ref{sec:alog} can be generalized to handle multiple sensors by considering 
$\prod_{i=1}^r \Ab_{x_i}$,
where $\Ab_{x_i}$ is the normalized kernel affinity matrix for the point cloud $\mathcal{X}_i.$ 
It is thus natural to ask if we can extend the results of Theorem \ref{thm_edgevalue} to the case when we have a fixed number of multiple sensors. { Unfortunately, we cannot directly generalize our analysis policy to this case. First, the proof of Theorem \ref{thm_edgevalue} relies on the conclusion that the imaginary parts of the eigenvalues of $n^2 \Ab_{xy}$ are negligible with high probability, which comes from the fact that $n^2 \Ab_{xy}$ is close to $f^{-2}(2) \Wb_x \Wb_y$ (c.f. (\ref{eq_diagonacontrol})) in the sense of matrix norm. Moreover, the later matrix $f^{-2}(2) \Wb_x \Wb_y$  is clearly similar to the positive semi-definite matrix $f^{-2}(2) \Wb_y^{1/2} \Wb_x^{1/2} (\Wb_y^{1/2} \Wb_x^{1/2})^\top$. Consequently, with high probability, the eigenvalues of $n^2 \Ab_{xy}$ can be well approximated by some nonnegative real numbers. However, such a relation does not hold even for a product of three matrices, say $\prod_{i=1}^3 \Ab_{\xb_i}$. In general, the eigenvalues of a product of more than two kernel affinity matrices are complex numbers. In Figure \ref{fig_comparematrix}, we report some numerical simulations to illustrate this phenomenon by exploring the ratios between the imaginary and real parts of the eigenvalues. It can be observed that the imaginary parts of the eigenvalues of the AD matrices of three sensors are not negligible compared with the real parts. Second, as we will see from the discussion of the proof strategy in Section \ref{sec:proofstrategy}, the study of AD matrices of  three sensors will be reduced to the form of  $\mathbf{U}_1 \Sigma_1 \mathbf{U}_1^\top \mathbf{U}_2 \Sigma_2 \mathbf{U}_2^\top \mathbf{U}_3 \Sigma_3 \mathbf{U}_3^\top,$ where $\mathbf{U}_i$, $i=1,2,3,$ are Haar orthogonal random matrices. Such a matrix product cannot be reduced to the free multiplication of random matrix model (\ref{model H=AUBUt}). To our knowledge, the products of random matrices involving more than two Haar orthogonal random matrices have not been well studied yet. To sum up, the study of the eigenvalues of more than two sensors is out of the scope of this paper, and we will pursue this direction in our future work.  

\begin{figure}[ht]
\centering
	\includegraphics[width=0.35\textwidth]{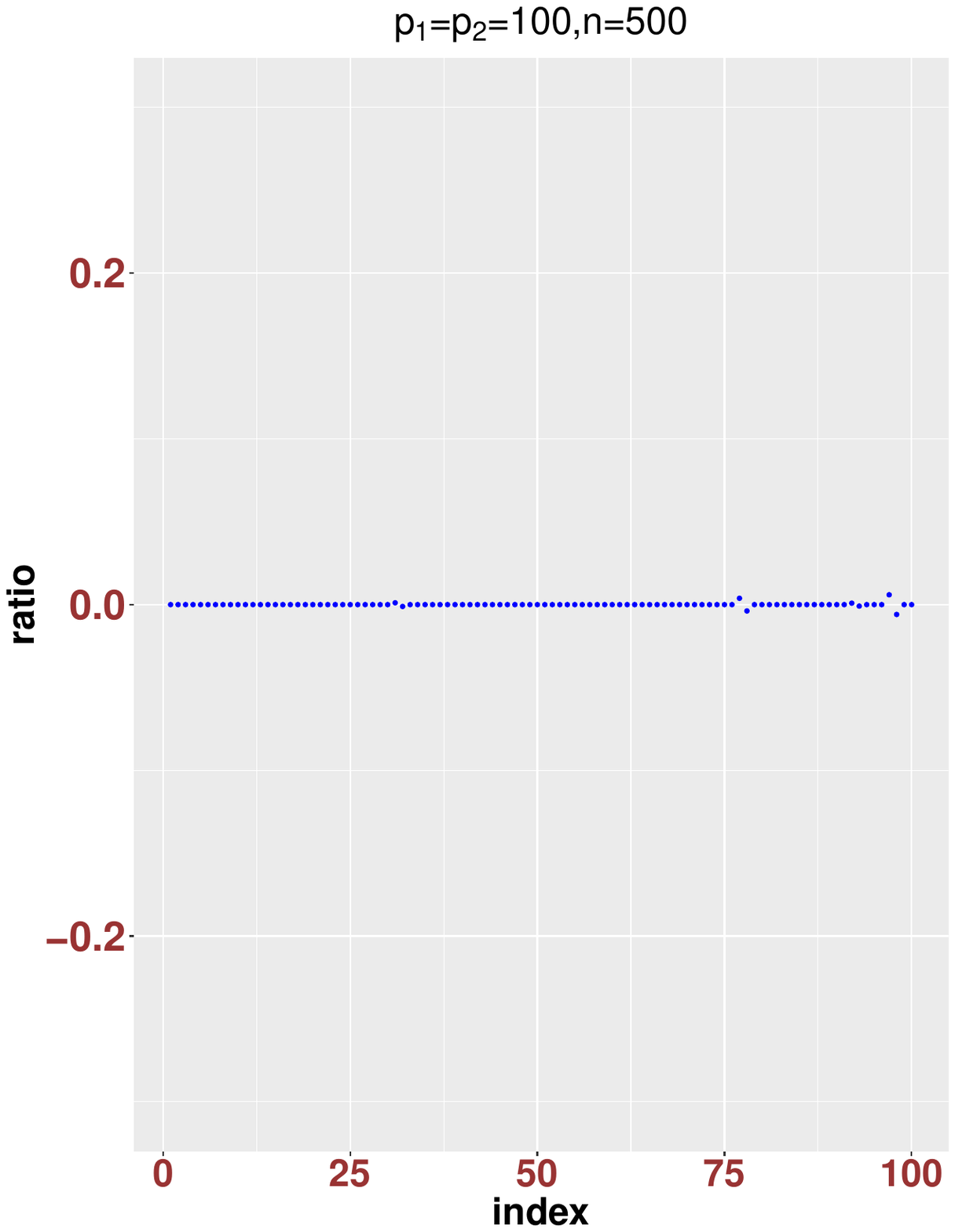}
		\includegraphics[width=0.35\textwidth]{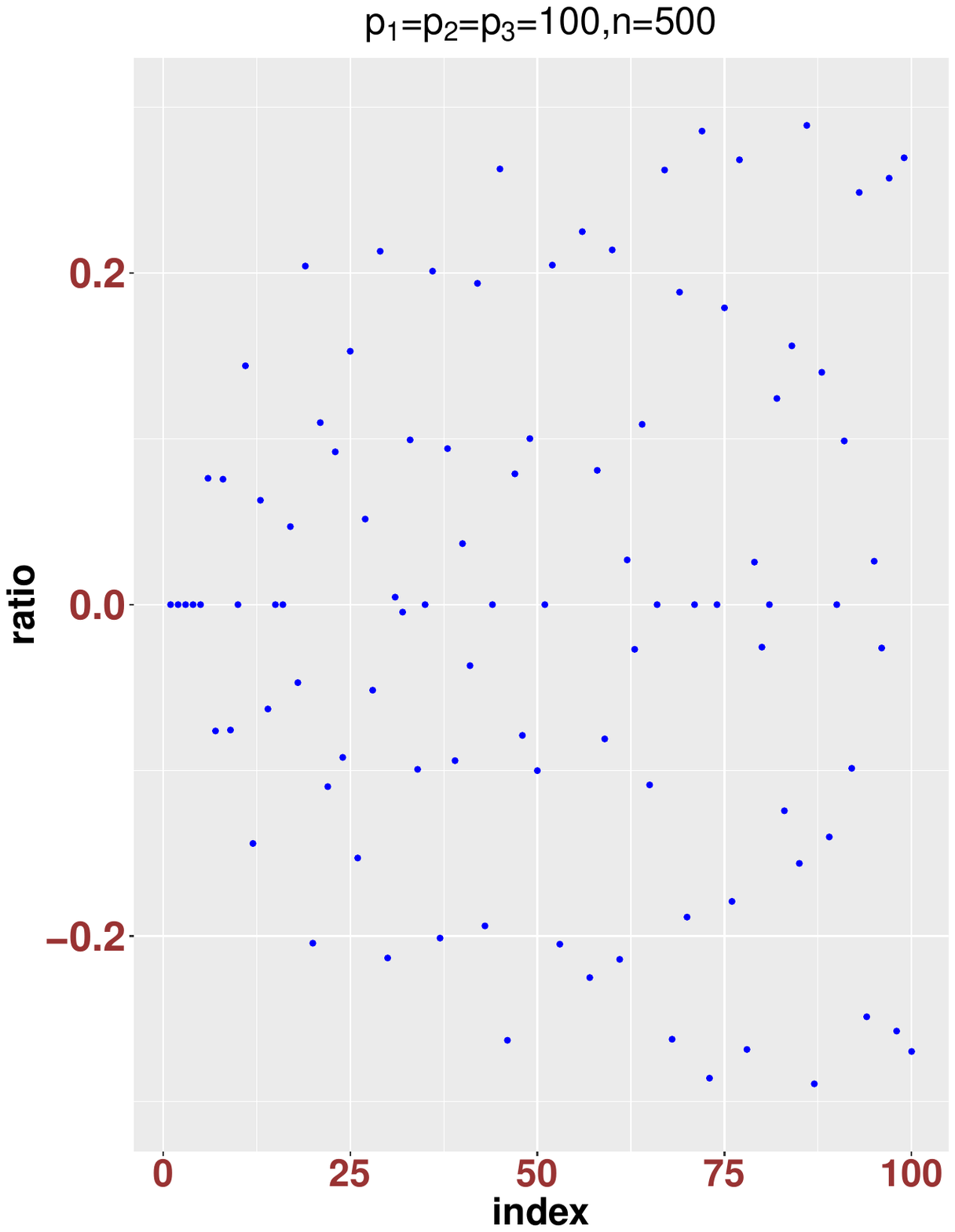}\\
	\includegraphics[width=0.35\textwidth]{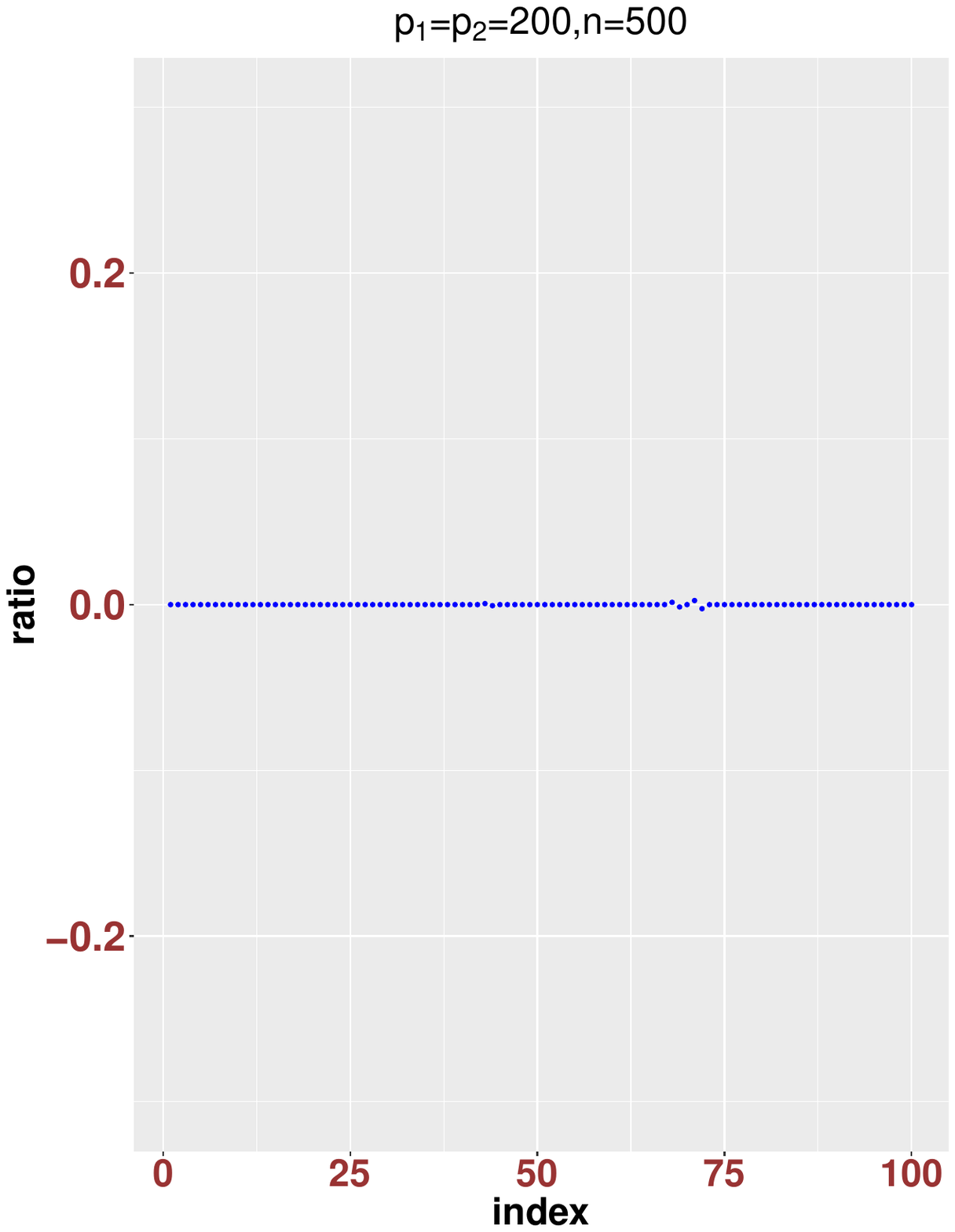}
		\includegraphics[width=0.35\textwidth]{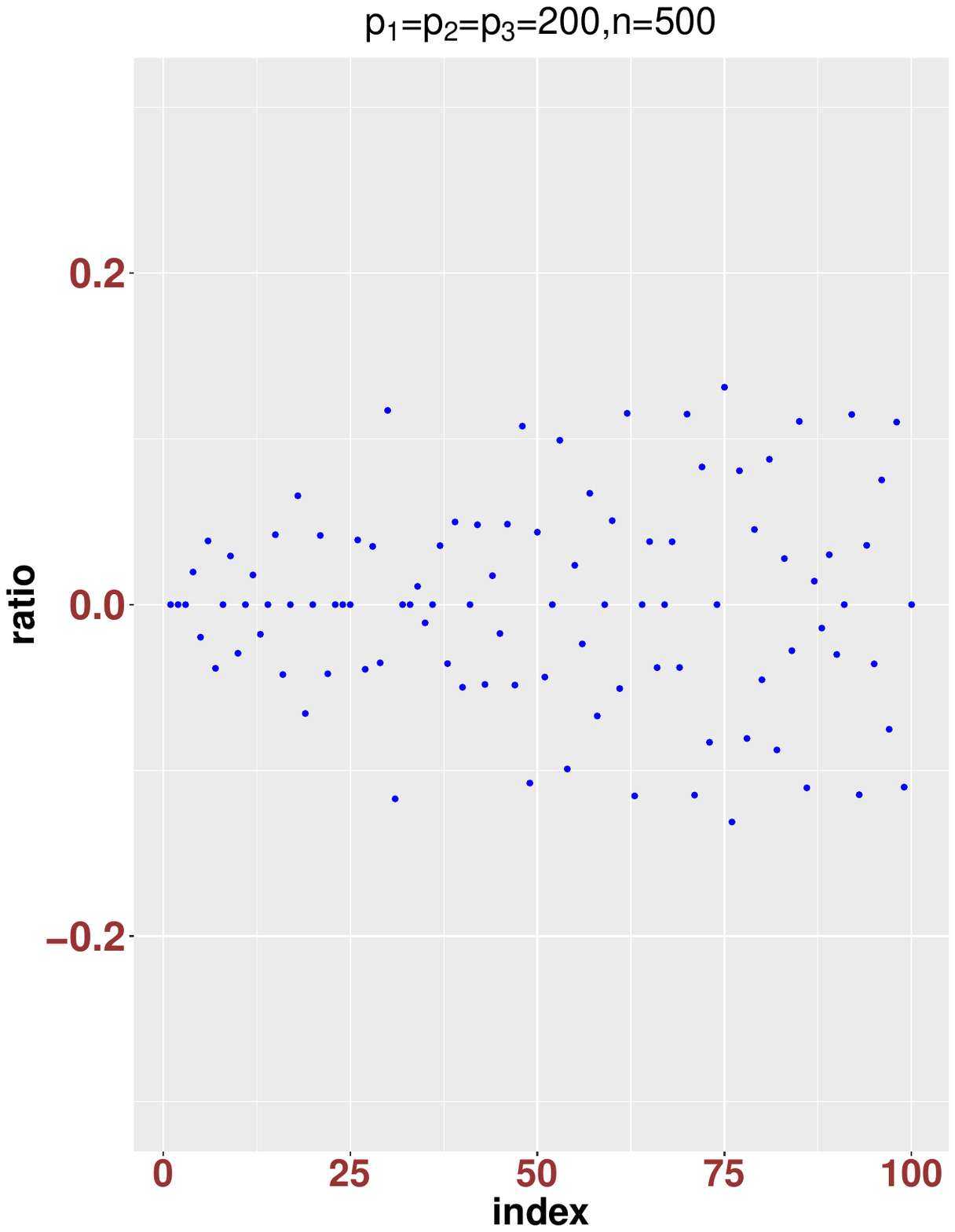}
	\caption{{Left column: the ratios of the imaginary parts and real parts of the eigenvalues of the AD matrices from two independent point clouds. The upper panel reports results when $p_1=p_2=100, n=500,$ whereas the lower panel reports results when $p_1=p_2=200, n=500.$ Right column: the ratios of the imaginary parts and real parts of the eigenvalues of the AD matrices from three independent point clouds. The upper panel reports results when $p_1=p_2=p_3=100, n=500,$ whereas the lower panel reports results when $p_1=p_2=p_3=200, n=500.$ It can be seen that when there are two sensors, the imaginary parts are negligible compared with their corresponding real parts, whereas the imaginary parts are not negligible when there are three sensors. In the simulations, we only record the results of the first 100 eigenvalues.}}\label{fig_comparematrix}
\end{figure}
 
}

}
On the other hand, {due to the mutual information nature of NCCA, an extension of NCCA to multiple sensors is different from that of} AD {and a different treatment is needed}. In general, we may ask the following natural question. Take a random subset $\mathcal{I}\subset \{1,\ldots,r\}$ and an arbitrary order denoted as $\pi$.
Define $\Bb_{\mathcal{I},\pi}$ to be a product of $\{\Ab_{x_i}\}_{i\in  \{1,\ldots,r\}\backslash \mathcal{I}}$ and $\{\Ab_{x_i}^\top\}_{i\in \mathcal{I}}$ according to the order $\pi$. We may ask under the pair of $\mathcal{I}$ and $\pi$, what is the spectral behavior of $\Bb_{\mathcal{I},\pi}$, or even $\sum_{l=1}^L(-1)^{v(l)}\Bb_{\mathcal{I}_l,\pi_l}$ for $L$ different pairs of $\mathcal{I}$ and $\pi$, where $v$ is a $L$-dim vector with $1$ or $-1$ in its entries. This particularly interesting problem is a natural extension of symmetrized or anti-symmetrized AD considered in \cite{shnitzer2018recovering}.
While it is out of the scope of this paper, we will explore it in our future work.

\begin{remark}\label{Remark:nonHaar}
In this paper, we obtain our main results under the assumption of white noise. This is mainly because we need the fact that the eigenvectors of the Wishart matrix  (i.e., sample covariance matrix with zero-mean Gaussian random vectors with isotropic covariance matrix) are Haar distributed. For the general covariance matrices of $\{\xb_i\},$ even though a lot of literature claims that the eigenvectors are asymptotically Haar distributed, for instance, \cite{Bourgade2017}, to our knowledge there is no rigorous work characterizing or proving it. Since it has been shown that the entries of the eigenvectors of sample covariance matrix have magnitudes of order $n^{-1/2}$ \cite{ding_2019} for a general class of covariance matrices, we conjecture that our results hold true for general covariance matrices satisfying certain regularity assumption (c.f. Assumption \ref{assu_regularity}).   
\end{remark}

\subsection{Proof strategy} \label{sec:proofstrategy}  
We briefly describe our proof strategy. For the convergence rates of eigenvalues of kernel affinity matrix in Section \ref{Subsection:eigenvalue convergence rate of kernel matrix}, we will employ a discussion similar to \cite{cheng2013,elkaroui2010} to provide convergent rates for the Stieltjes transforms of the ESDs of $\mathbf{W}_x$ and $\mathbf{W}_y$. This is done using the Taylor expansion and concentration inequalities. In particular, we use a Chernoff bound to control the off-diagonal entries and use the Bernstein's inequality to control the diagonal entries of $\mathbf{W}_x$ and $\mathbf{W}_y$. Then, the spectral convergent rate is obtained using the Hoffman-Wielandt inequality.  As a byproduct, we show that
\begin{equation}\label{eq_diagonacontrol}
\norm{n(\mathbf{D}_x)^{-1}-f^{-1}(2)}=O_{\prec}(n^{-1/2})\,.
\end{equation}
For the edge eigenvalues of the NCCA matrix in Section \ref{Subsection:eigenvalue convergence rate of NCCA matrix}, we approximate $n^2\mathbf{S}_{xy}$ by matrices that we can control. In light of (\ref{eq_diagonacontrol}), we only need to study the product of affinity matrices $f^{-2}(2)\mathbf{W}_x \mathbf{W}_y$.
Next, we approximate the spectrum of $\mathbf{W}_x$ locally by that of $\mathbf{K}_x+\mathrm{Sh}_1+\mathrm{Sh}_2$, where $\mathbf{K}_x \in \mathbb{R}^{n \times n}$ is defined as    
\begin{equation}\label{eq_kmatrix}
\mathbf{K}_x(i,j)=
\begin{cases}
f(2)-2f'(2)\xb_i^\top \xb_j  & i \neq j \,,\\
f(0) & i=j\,.
\end{cases}
\end{equation} 
Since a finite-rank perturbation will not influence the spectrum (c.f. (ii) of Lemma \ref{lem_collecttion}), it suffices to study the product of shifted sample Gram matrices in the form of $\Zb_{xy}=(\alpha+\beta \Xb^\top \Xb)(\alpha+\beta \Yb^\top \Yb)$, where $\alpha, \beta>0$.
 It is clear that the ESD of $\alpha+\beta \Xb^\top\Xb$ satisfies a shifted MP law. Due to the independence, the LSD of $\Zb_{xy}$ is the free multiplicative convolution of two shifted MP laws. To obtain the rigidity of eigenvalues, we keep track of the closeness of the Stieltjes transforms of $\mu_{\Zb_{xy}}$ and its deterministic counterpart. Using the rigidity of eigenvalues for sample Gram matrices (c.f. Lemma \ref{lem_rigidity}), we construct a high probability event $\Omega_1$ such that when conditional on $\Omega_1$,  there exists a diagonal deterministic positive definite matrix $\Sigma_y$ (c.f. equation (\ref{defn_sign})) so that the eigenvalues of $\Zb_{xy}$ are close to those of   
\begin{equation*}
\tilde{\Zb}_{xy}:=\big( \alpha+\beta \Xb^\top \Xb \big) \Ub_y\Sigma_y\Ub_y^\top\,.
\end{equation*}
As a result, we have reduced analyzing $n^2\mathbf{S}_{xy}$ to analyzing $\tilde{\Zb}_{xy}$ with the convergence rate tracked. Since the  difference between $n^2 \mathbf{S}_{xy}$ and $\tilde{\Zb}_{xy}$ is of finite rank, Weyl's inequality can be applied to control the spectral discrepancy. 
Finally, we study the eigenvalues of $\tilde{\Zb}_{xy}$. By using the rigidity of sample Gram matrices again, we find another high probability event $\Omega_2$ such that when conditional on $\Omega_2$, there exists a diagonal matrix $\Sigma_x$ (c.f. equation (\ref{eq_lambdan})) so that the eigenvalues of $\tilde{\Zb}_{xy}$ are close to those of
\begin{equation*}
\breve{\Zb}_{xy}:=\Ub_x \Sigma_x \Ub_x^\top \Ub_y\Sigma_y\Ub_y^\top\,,
\end{equation*} 
where $\Ub_x\in \mathrm{O}(n)$ is the eigenvector matrix for $\Xb^\top \Xb$. Since the null data follows the isotropic Gaussian random vector, $\mathbf{U}_x$ is a random orthogonal matrix Haar distributed on $\mathrm{O}(n)$ \cite{mehta2004random}. As a result, with high probability, $\breve{\Zb}_{xy}$ is similar to a product of two symmetric matrices with a Haar orthogonal conjugation, $\Sigma_x \Ub \Sigma_y\Ub^\top$, where $\Ub$ is Haar distributed. The local laws and rigidity of this kind of multiplicative random matrix model will be studied by taking the Levy distance between the pair $(\Sigma_x, \Sigma_y)$ and its deterministic counterpart into account. Due to the square root behavior of the MP law, the Levy distance between $\Sigma_x$ and $\Sigma_y$ and their deterministic counter parts are of order $n^{-2/3}$. This leads to the local laws and rigidity of $\breve{\Zb}_{xy}$ and we conclude our proof.
It is remarkable that due to (\ref{eq_diagonacontrol}), the study of the AD matrix is essentially the same as the above analysis of the NCCA matrix. 

{
\subsection{Statistical applications} One main purpose of the NCCA and the AD is detecting whether the point clouds $\mathcal{X}$ and $\mathcal{Y}$ contain common information. A fundamental related problem is checking if $\mathcal{X}$ and $\mathcal{Y}$ have common information, or are independent. Below we propose a solution to this question.

Consider the hypothesis test: 
\begin{equation}
\mathbf{H}_0: \ \mathcal{X}\ \text{and} \ \mathcal{Y} \ \text{are independent}\,.  
\end{equation}
Motivated by Theorem \ref{thm_edgevalue}, we find that the edge eigenvalues are close to each other starting from the third one.   Denote the real parts of the eigenvalues of $n^2 \mathbf{S}_{xy}$ as $\text{Re} \ \lambda_1 \geq \text{Re} \ \lambda_2 \ldots \geq \text{Re} \ \lambda_n$. For a given threshold $L >3$,  we define the statistic
\begin{equation}\label{Definition:statistic T} 
\mathbb{T}(L):=\max_{3 \leq i \leq L} \left| \frac{ \text{Re} \ \lambda_{i}}{\text{Re} \ \lambda_{i+1}} \right|.
\end{equation}
Note that a similar statistic has been used in estimating the number of factors in the factor model in \cite{lam2012}, the number of signals in \cite{ding2017a}, and the number of spikes of deformed random matrices in  \cite{DYS, PASSEMIER2014173}. 
Together with Theorem \ref{thm_edgevalue}, we have the following property of the statistic $\mathbb{T}(L)$.
\begin{theorem}\label{thm_ratio}
For a fixed $L  \geq 3$, 
\begin{equation}
\mathbb{P} ( |\mathbb{T}(L)-1| \leq \theta ) \rightarrow 1\,,
\end{equation}
when $\theta=\theta(n)>0$, $\theta \rightarrow 0$ and $\theta n^{1/9} \rightarrow \infty$ when $n\to\infty$.
\end{theorem}

\begin{proof}[Proof of Theorem \ref{thm_ratio}]  
The proof follows a direct computation based on Theorem \ref{thm_edgevalue}. 
Since $\gamma_{\nu_{xy}}(1)=O(1)$ (See \cite[Theorem 2.6]{JHC}, which is restated in Theorem \ref{lem_bao2}), the eigenvalues $\lambda_3,\ldots,\lambda_{L+1}$ are of order $1$ when $n$ is sufficiently large. Thus, by Theorem \ref{thm_edgevalue}, for $i=1,\ldots,L$, there exists some constant $C>0$ such that with probability at least $1-O(n^{-1/2})$,
\begin{equation}
\left| \frac{\lambda_i}{\lambda_{i+1}}-1 \right|=\left|\frac{\lambda_i-\lambda_{i+1}}{\lambda_{i+1}}\right| \leq C n^{-1/9}\,,
\end{equation}
which leads to $\mathbb{T}(L)=1+O(n^{-1/9})$ holds with probability at least $1-O(n^{-1/2})$.
This concludes our proof. 
\end{proof}

Based on Theorem \ref{thm_ratio}, we should reject the null hypothesis if $\mathbb{T}(L)$ is ``much'' larger than $1$.  For the practical implementation, we have two parameters to choose: the number of eigenvalues $L$ and the threshold $\theta$. Since a larger $L$ yields a (possibly) larger value of $\mathbb{T}(L)$, we prefer a larger $L$ to make the statistic sensitive to the alternatives. Furthermore, for the purpose of hypothesis testing, we need to provide a significance level when $\mathbf{H}_0$ is rejected. To achieve these goals, we can employ a resampling technique introduced in \cite{PASSEMIER2014173}. A systematic treatment of statistical inference for AD or NCCA will be explored in our future work, including the non-null case and power analysis.

}

\section{Proofs of main theorems}\label{sec:mainproof}

We now prove our main theorems -- Theorems \ref{lem_rigiditydistance} and \ref{thm_edgevalue} and Corollary \ref{coro_lesspertub}. The key ingredients are the \emph{averaged local laws} for $\mathbf{W}_x$ and $\mathbf{W}_y$ and their products. Such results state that, for instance, $m_{\mathbf{W}_x}(z)$ is close to that of $m_{\mathbf{K}_x}(z)$, where $\mathbf{K}_x$ is defined in (\ref{eq_kmatrix}), when $z=E+\mathrm{i} \eta \in \mathsf{S}(1/4,\tau)$.

\begin{remark}
In the literature of Random Matrix Theory on the derivation of local laws, the parameter set is usually chosen as $\mathsf{S}(1,\tau)$, while the results in this paper come with $\mathsf{S}(1/4,\tau)$. While it is sufficient for our purpose, it may not be optimal in the sense of local law. We will strengthen this local law in future work. 
\end{remark}

\subsection{Proof of Theorem \ref{lem_rigiditydistance} and Corollary \ref{coro_lesspertub}}
We introduce some key lemmata.
The first lemma leads to the local law for $\mathbf{W}_x$ and $\mathbf{W}_y$ for {$z=E+\mathrm{i}\eta \in \mathsf{S}(1/4,\tau)$.} We will focus our discussion on $\mathbf{W}_x$ and similar results hold for $\mathbf{W}_y$. 
Denote 
\begin{equation}\label{eq_ktilde}
\widetilde{\Kb}_x=\mathbf{K}_x-\mathrm{Sh}_0\,,
\end{equation}
where $\mathbf{K}_x$ is defined in \eqref{eq_kmatrix} and
\begin{equation}\label{eq_matkotherform}
\breve{\Kb}_x=-2f'(2)\Xb^\top \Xb+ \varsigma \mathbf{I}\,.
\end{equation}
\begin{lemma}  \label{lem_localkernel}   
{For $z=E+\mathrm{i} \eta \in \mathsf{S}(1/4,\tau)$ and a large integer $q \geq 1$, when $n$ is sufficiently large, we have}
\begin{align*}
&{\mathbb{E}\left|m_{\mathbf{W}_x}(z)-m_{\mathbf{K}_x}(z) \right|^q \leq C\left( \frac{ \log n}{\sqrt{n} \eta^2} \right)^q ,} \\
&\mathbb{E} \left| m_{\mathbf{K}_x}(z)-m_{\widetilde{\Kb}_x}(z) \right|^q \leq \left( \frac{2}{ n \eta} \right)^q,
\end{align*}
and 
\begin{equation*}
 \mathbb{E} \left| m_{\tilde{\mathbf{K}}_x}(z)-m_{\breve{\mathbf{K}}_x}(z)  \right|^q \leq C\left(\frac{\log n}{\sqrt{n} \eta^2} \right)^q. 
\end{equation*}
where $C>0$ is a constant.  
\end{lemma}

The proof is postponed to Appendx \ref{app:local}.
In light of Lemma \ref{lem_localkernel}, it suffices to 
study the local law of the matrix $\breve{\mathbf{K}}_x$. Since $f$ is decreasing and satisfies (\ref{eq_fassum}),  $\breve{\mathbf{K}}_x$ is positive definite.  As discussed in Section \ref{sec:preliminary}, the local laws of $-2f'(2)\Xb^\top \Xb$ has been established (c.f. Lemma \ref{lem_locallaw}). Since $\varsigma \mathbf{I}$ is an isotropic shift, we directly obtain the LSD of $\breve{\mathbf{K}}_x$. 
Indeed, we have 
\begin{equation*}
m_{\breve{\mathbf{K}}_x}(z)=m_{\breve{\mathbf{K}}_x-\varsigma}(z-\varsigma)=m_{-2f'(2)\Xb^\top \Xb}(z-\varsigma)\,,
\end{equation*}
and by Lemma \ref{lem_locallaw}, we have  
\begin{equation*}
m_{-2f'(2)\Xb^\top \Xb}(z-\varsigma) = m_{\mu_{c_1,-2f'(2)}}(z-\varsigma)+O_{\prec}( (n \eta)^{-1})\,.
\end{equation*}
By the inversion formula, we see that the LSD of $\breve{\mathbf{K}}_x$ is 
$\nu_x$ defined in \eqref{Definition special measure for kernel matrix}.

We also need the following two propositions about $\widetilde{\mathbf{W}}_x$ defined in \eqref{eq_defntildem}. The first proposition states that $\widetilde{\mathbf{W}}_x$ is bounded with high probability.  

\begin{proposition}\label{prop_norm} Denote $\gamma_+$ as the right-most endpoint of the support of $\nu_x$. Then for some constant $C>0$, with probability $1-O(n^{-1/2})$, we have
\begin{equation*}
\left|\norm{\widetilde{\mathbf{W}}_x} - \gamma_+\right|\leq C n^{-1/9}\,.
\end{equation*}
\end{proposition}

The next proposition implies that the ESD of $\widetilde{\mathbf{W}}_x$ is ``close'' to $\nu_x$. 
\begin{proposition}\label{prop_largscaleset} 
For some constant $C>0$ and small constant $\epsilon>0$, with probability $1-O(n^{-1/2})$, we have uniformly 
\begin{equation*}
\left|\mu_{\widetilde{\mathbf{W}}_x}(I)-\nu_x(I)\right| \leq Cn^{-1/4+\epsilon}\,,
\end{equation*}
for any interval $I \subset \mathbb{R}$.
\end{proposition}
The proofs of Propositions \ref{prop_norm} and \ref{prop_largscaleset} are postponed to Appendices \ref{app:local} and \ref{section proof of propositions prop_largscaleset}.
We now finish the proofs of Theorem \ref{lem_rigiditydistance}. We basically follow the proof of \cite[Lemma 3.2]{ding2018} except that we prove the averaged local law on $\mathsf{S}(1/4,\tau)$. 

\begin{proof}[Proof of Theorem \ref{lem_rigiditydistance}] 
We only deal with the case $i \leq  n/2$ and the case when $i > n/2$ can be dealt with in a similar way. 
Denote the eigenvalues of $\widetilde{\mathbf{W}}_x$ by $\lambda_1\geq\lambda_2\ldots\geq\lambda_n$. We first prepare the following  claims.

\underline{Claim 1.} We connect the stochastic quantity $\lambda_i$ and deterministic quantity $\gamma_{\nu_x}(i)$, where $\nu_x$ is defined in \eqref{Definition special measure for kernel matrix}.
Clearly, we have by definition
\begin{align}\label{eq_fbound0}
\mu_{\widetilde{\mathbf{W}}_x}([\lambda_i,\infty))=\frac{i}{n}.
\end{align} 
For $E \geq 0$, define the function  
\begin{equation}\label{eq_deff}
F(E):=\nu_x([E, \infty))\,.
\end{equation}
Recall the typical location, $\gamma_{\nu_x}(i)$, in Definition \ref{Definition typical locations}; that is, we have
\begin{align}\label{eq_fbound1}
F(\gamma_{\nu_x}(i))+\frac{1}{2n}=\frac{i}{n}.
\end{align} 
By \eqref{eq_fbound0} and \eqref{eq_fbound1}, the stochastic quantity $\lambda_i$ and deterministic quantity $\gamma_{\nu_x}(i)$ are connected via $i/n$ and we find that 
\begin{align}\label{eq_fbound}
F(\gamma_{\nu_x}(i))+\frac{1}{2n}=\mu_{\widetilde{\mathbf{W}}_x}([\lambda_i,\infty))\,.
\end{align} 
By Proposition \ref{prop_largscaleset}, we have $|F(\lambda_i)-\mu_{\widetilde{\mathbf{W}}_x}([\lambda_i,\infty))|=O(n^{-1/4+\epsilon})$ for some small $\epsilon>0$ with probability $1-O(n^{-1/2})$, and hence for some constant $C>0$ and small constant $\epsilon>0$,
\begin{align}\label{eq_fbound}
|F(\gamma_{\nu_x}(i))-F(\lambda_i)| \leq C n^{-1/4+\epsilon}
\end{align} 
with probability $1-O(n^{-1/2})$.

Therefore, by Propositions \ref{prop_norm} and \ref{prop_largscaleset}, there exists  some constant $C>0$ so that we can find an event $\Omega$ with probability at least $1-O(n^{-1/2})$ such that when conditional on $\Omega$, 
\begin{equation}\label{eq_expand}
|\norm{\widetilde{\mathbf{W}}_x}- \gamma_+|\leq Cn^{-1/9}\ \text{ and } \ |F(\gamma_{\nu_x}(i))-F(\lambda_i)| \leq Cn^{-1/4+\epsilon}\,.
\end{equation}

\underline{Claim 2.} We control the typical location. We claim that for a fixed small $\vartheta>\epsilon$, $\gamma_{\nu_x}(i)\geq \gamma_+-n^{-1/9+\vartheta}$ if and only if $i \leq C_0 n^{5/6+3\vartheta/2}$ for some constant $C_0>0$. Indeed, when $\gamma_{\nu_x}(i)\in[\gamma_+-n^{-1/9+\vartheta},\gamma_+]$, we have
\begin{equation}\label{eq_contra}
F(\gamma_{\nu_x}(i))=\int_{\gamma_{\nu_x}(i)}^{\gamma_+}{\nu_x}(dx)\sim \frac{2}{3}(\gamma_+-\gamma_{\nu_x}(i))^{3/2}\leq \frac{2}{3}n^{-1/6+3\vartheta/2}\,,
\end{equation}
where $\sim$ holds by the square root behavior of the MP law near the edge (see, e.g., equation (\ref{eq)squareroot})). 
Therefore, the only if part holds for some constant $C_0>0$ by taking \eqref{eq_fbound1} into consideration. The if part is straightforward by reversing the argument.

With the above {two} claims, 
we consider two cases to finish the proof. 
First, suppose that $i \leq C_0n^{5/6+3\vartheta/2}$.
By the claim, we have
\begin{equation}\label{eq_firstregion}
\gamma_{\nu_x}(i)\in [\gamma_+-n^{-1/9+\vartheta}, \gamma_+].
\end{equation}
If $\lambda_i\geq \gamma_+-n^{-1/9+2\vartheta}$ holds,
when combined with \eqref{eq_expand},
we have for some constant $C>0$,
\begin{equation*}
|\lambda_i-\gamma_{\nu_x}(i)|\leq n^{-1/9+2\vartheta}+n^{-1/9}\leq C n^{-1/9+2\vartheta}\,,
\end{equation*}
when $n$ is sufficiently large and when conditional on $\Omega$. 
If $\lambda_i<\gamma_+-n^{-1/9+2\vartheta}$ holds, we show a contradiction.  
By the definition of $F$, we have that 
\begin{align*}
F(\lambda_i)  \geq F(\gamma_+-n^{-1/9+2 \vartheta })& = \int_{\gamma_+-n^{-1/9+2 \vartheta}}^{\gamma_+}{\nu_x}(dx) \\
& \sim \frac{2}{3}(n^{-1/9+2\vartheta})^{3/2}= \frac{2}{3} n^{-1/6+3 \vartheta}\,.
\end{align*}
We hence conclude from (\ref{eq_expand}) that 
\begin{equation*}
F(\gamma_{\nu_x}(i)) \geq C n^{-1/6+3\vartheta} \,,
\end{equation*}
for some constant $C>0$. This is a contradiction to (\ref{eq_contra}) when $n$ is sufficiently large. As a result, we conclude that when $i \leq C_0n^{5/6+3\vartheta/2}$, with probability at least $1-O(n^{-1/2})$, for some constant $C>0$,
 \begin{equation}
|\lambda_i-\gamma_{\nu_x}(i)| \leq C n^{-1/9+2\vartheta}\,.
\end{equation}

Second, we consider the case $C_0n^{5/6+3\vartheta/2} \leq i \leq n/2$. When conditional on $\Omega$, by the {second} claim we have  $\gamma_{\nu_x}(i)<\gamma_+-n^{-1/9+\vartheta}$. Also, we have 
\begin{align*}
F(\gamma_{\nu_x}(i)) \geq \int_{\gamma_+-n^{-1/9+3 \vartheta/2}}^{\gamma_+}{\nu_x}(dx)  & \sim \frac{2}{3}(n^{-1/9+3\vartheta/2})^{3/2}  \\
& = \frac{2}{3}n^{-1/6+9\vartheta/4}\,,
\end{align*}
where $\sim$ again comes from the square root behavior of the MP law near the edge; that is, there exists some constant $c>0$ so that when $n$ is sufficiently large,   
\begin{equation}\label{F(gammanux) bounded from below}
F(\gamma_{\nu_x}(i))  \geq c n^{-1/6+9 \vartheta /4}\,,
\end{equation}
which asymptotically dominates $n^{-1/4+\epsilon}$. 
Again, due to the square root behavior of the MP law near the edge, we deduce that 
\begin{equation}\label{proof MPlaw edge eq_expand}
F(\lambda) \sim (\gamma_+-\lambda)^{3/2} \,,
\end{equation} 
when $0<\lambda<\gamma_+$.  
Therefore, when $n$ is sufficiently large, by \eqref{eq_expand} and \eqref{F(gammanux) bounded from below}, we know $F(\gamma_{\nu_x}(i))\sim F(\lambda_i)$, which by \eqref{proof MPlaw edge eq_expand} leads to
\begin{equation}\label{proof theorem lambdai gammai the same}
\gamma_+-\lambda_i \sim \gamma_+-\gamma_{\nu_x}(i)\,.
\end{equation} 
To continue, note that since $F'(\lambda) \sim (\gamma_+-\lambda)^{1/2}$ by \eqref{eq)squareroot}, \eqref{proof theorem lambdai gammai the same} leads to 
 $F'(\lambda_i) \sim F'(\gamma_{\nu_x}(i))$. 
As a result, by (\ref{eq_fbound1}), the square root behavior of the MP law \eqref{proof MPlaw edge eq_expand} and $F(\gamma_{\nu_x}(i))\sim F(\lambda_i)$, we have 
\begin{equation}\label{proof theorem lambdai right most edge gap}
\gamma_+-\gamma_{\nu_x}(i) \sim i^{2/3} n^{-2/3},
\end{equation} 
which leads to 
\[
F'(\gamma_{\nu_x}(i)) \sim i^{1/3} n^{-1/3}
\] 
since $F'(\lambda) \sim (\gamma_+-\lambda)^{1/2}$ and $F'(\lambda_i) \sim F'(\gamma_{\nu_x}(i))$. 
Finally, for any $\xi$ between $\lambda_i$ and $\gamma_{\nu_x}(i)$, we have $\gamma_+-\xi \sim \gamma_+-\gamma_{\nu_x}(i)$ and hence 
\begin{equation}\label{proof theorem control F' in xi}
F'(\xi) \sim F'(\gamma_{\nu_x}(i))\sim i^{1/3} n^{-1/3}.
\end{equation} 
Now we put everything together. By the mean value theorem, we have
\begin{equation}\label{eq_comprae}
|\lambda_i-\gamma_{\nu_x}(i)| = \frac{|F(\lambda_i)-F(\gamma_{\nu_x}(i))|}{|F'(\xi)|},
\end{equation}
where $\xi$ is between $\lambda_i$ and $\gamma_{\nu_x}(i)$. 
When conditional on $\Omega$, by (\ref{eq_expand}) and \eqref{proof theorem control F' in xi},
\begin{equation}\label{eq_comprae}
\frac{|F(\lambda_i)-F(\gamma_{\nu_x}(i))|}{|F'(\xi)|} \leq C (i/n)^{-1/3} n^{-1/4+\epsilon},
\end{equation}
for some constant $C>0$. This concludes our proof. 

\end{proof}

Before concluding this section, we finish the proof of  
Corollary \ref{coro_lesspertub}. 
\begin{proof}[Proof of Corollary \ref{coro_lesspertub}] Note that by Weyl's inequality, we have  
\begin{align*}
\left| \lambda_i(\mathbf{W}_x-\mathrm{Sh}_0-\mathrm{Sh}_1)-\gamma_{\nu_x}(i) \right| \leq \left| \lambda_i(\widetilde{\mathbf{W}}_x)-\gamma_{\nu_x}(i) \right| +\norm{\mathrm{Sh}_2}\,. 
\end{align*}
The proof follows from Theorem \ref{lem_rigiditydistance} and the claim that
\begin{equation}\label{Bound of Sh1+Sh1}
\norm{\mathrm{Sh}_1}+\norm{\mathrm{Sh}_2}=O_{\prec} \big( n^{-1/2} \big)\,. 
\end{equation}
To prove \eqref{Bound of Sh1+Sh1}, since $\psi \mathbf{1}^\top$ and $\mathbf{1}\psi^\top$ are rank-one matrices,  we have 
\begin{equation}
\norm{\mathrm{Sh}_1} \leq |2f'(2)| |\psi^\top \mathbf{1}|=|2f'(2)| \Big|\sum_{i=1}^n \psi_i\Big|\,. 
\end{equation}
By definition, $\psi_i$, $i=1,2,\cdots,n$, are centered and independent and $\psi_i=\frac{1}{p_1} \sum_{i=1}^{p_1} z_i^2-1 $ with $\{z_i\} $ are  i.i.d. 
Gaussian random variables. Therefore, we get from Bernstein's inequality (c.f. Lemma \ref{lem_bernstein}) that 
\begin{equation}
\Big| \sum_{i=1}^n \psi_i \Big|=O_{\prec}\big (n^{-1/2} \big)\,.
\end{equation}
Second, we have 
\begin{equation}
\norm{\mathrm{Sh}_2} \leq \Big| \frac{f^{''}(2)}{2} \left( 2|\mathbf{1}^\top (\psi \circ \psi)|+2\psi^\top \psi \right) \Big|=| 2f^{''}(2) | \sum_{i=1}^n \psi_i^2\,,
\end{equation}
where in the second step we use the definition of Hadamard product. We claim that when $n$ is large enough, we have 
\begin{equation}\label{eq_cccclaim}
\Big|\sum_{i=1}^n \psi_i^2\Big|=O_{\prec}\big( n^{-1/2} \big)\,,
\end{equation}
and this proves Corollary \ref{coro_lesspertub}. The proof of (\ref{eq_cccclaim}) is similar to the \textit{truncation and centralization} step in the proof of \cite[Theorem 2.2]{elkaroui2010}. For some small $\epsilon>0$, since we have sufficient high absolute moments, we can truncate $\psi^2_i, i=1,2,\cdots, n$ at the level $B_p=p^{-\delta}$ with $\delta=1/2-\epsilon$ such that a.s. the data matrix is not changed. Then we conclude (\ref{eq_cccclaim}) using \cite[Lemma A.3]{elkaroui2010} by setting $2/m=-\delta$.

\end{proof}

\subsection{Proof of Theorem \ref{thm_edgevalue}}
We introduce some notations. 
Denote $\Sigma_x$ and $\Sigma_y$ as $n \times n$ diagonal matrices satisfying
\begin{equation}\label{eq_lambdan}
\Sigma_x=
\begin{cases}
\text{diag}\{ \gamma_{\nu_x}(1), \ldots, \gamma_{\nu_x}(n) \} & \ \text{if} \ n \leq p_1\\
\text{diag}\{\gamma_{\nu_x}(1), \ldots, \gamma_{\nu_x}(p_1), \varsigma, \ldots, \varsigma \} & \ \text{if} \ n > p_1
\end{cases},
\end{equation} 
and
\begin{equation}\label{defn_sign}
\Sigma_y=
\begin{cases}
\text{diag}\{ \gamma_{\nu_y}(1), \ldots, \gamma_{\nu_y}(n) \} & \ \text{if} \ n \leq p_2\\
\text{diag}\{\gamma_{\nu_y}(1), \ldots, \gamma_{\nu_y}(p_2), \varsigma, \ldots, \varsigma \} & \ \text{if} \ n > p_2
\end{cases},
\end{equation}
where  $\nu_x$ and $\nu_y$ are defined in \eqref{Definition special measure for kernel matrix}.
Clearly, $\Sigma_x$ and $\Sigma_y$ are deterministic and positive definite.  
Denote 
\begin{equation}\label{eq_defnd}
\Qb_{xy}=f^{-2}(2)\Ub \Sigma_x \Ub^\top \Sigma_y.
\end{equation}
Let $\Ub \in O(n)$ be a Haar distributed orthonormal random matrix. 
It will be seen from the proof of Theorem \ref{thm_edgevalue} that $n^2 \mathbf{S}_{xy}$ can be effectively reduced to $\Qb_{xy}$, which follows the free multiplication of random matrices model in \eqref{model H=AUBUt}. 

\begin{proof}[Proof of Theorem \ref{thm_edgevalue}] 

The proof is based on a series of reductions. We start with the proof for the NCCA matrix $\mathbf{S}_{xy}$. First of all, we show that the eigenvalues of $n^2 \mathbf{S}_{xy}$ are actually close to those of a diagonalizable matrix. We first study $\widetilde{\mathbf{W}}_x$. By Weyl's inequality (see (v) in Lemma \ref{lem_collecttion}), we have %
\begin{align}\label{eq_weylkeybound}
\Big | \lambda_i(\widetilde{\mathbf{W}}_x)-\lambda_i(\widetilde{\mathbf{W}}_x+\mathrm{Sh}_1 &+\mathrm{Sh}_2) \Big | \leq  \norm{\mathrm{Sh}_1}+\norm{\mathrm{Sh}_2} =O_{\prec}( n^{-1/2})\,,
\end{align}
where the last inequality comes from the bound shown in \eqref{Bound of Sh1+Sh1}.
Since $\mathbf{W}_x=\widetilde{\mathbf{W}}_x+\mathrm{Sh}_0+\mathrm{Sh}_1+\mathrm{Sh}_2$, by Weyl's inequality, the fact that $\mathrm{Sh}_0$ is of rank one, and (\ref{eq_weylkeybound}), we obtain that for $1\leq i \leq n-1$, 
\begin{equation}\label{eq_secondpertub}
\left| \lambda_{i+1}(\widetilde{\mathbf{W}}_x)-\lambda_i(\mathbf{W}_x) \right|=O_{\prec} \big( n^{-1/2}\big)\,.
\end{equation} 
Denote 
\begin{equation} \label{eq_sxy1}
\mathbf{S}^{(1)}_{xy}:=(n^{-1} \mathbf{D}_x)^{-1} \widetilde{\mathbf{W}}_x \widetilde{\mathbf{W}}_y(n^{-1} \mathbf{D}_y)^{-1}. 
\end{equation}
Using a discussion similar to (\ref{eq_secondpertub}), we conclude that for $1\leq  i \leq n-2$,
\begin{equation}\label{eq_1}
\left| \lambda_{i+2}(n^2 \mathbf{S}_{xy})-\lambda_{i}(\mathbf{S}^{(1)}_{xy}) \right|=O_{\prec} \big( n^{-1/2} \big)\,. 
\end{equation}
Therefore, it suffices to study the eigenvalues of $\mathbf{S}^{(1)}_{xy}$. Denote 
\begin{equation}\label{eq_s2xy}
\mathbf{S}^{(2)}_{xy}:=f^{-2}(2) \widetilde{\mathbf{W}}_x \widetilde{\mathbf{W}}_y\,. 
\end{equation}
We have
\begin{align}
\mathbf{S}^{(1)}_{xy}-\mathbf{S}^{(2)}_{xy}=&\, \left( (n^{-1} \mathbf{D}_x)^{-1}-f^{-1}(2) \right) \widetilde{\mathbf{W}}_x  \widetilde{\mathbf{W}}_y (n^{-1} \mathbf{D}_y)^{-1}\label{Equation Difference S1 and S2}\\
&+f^{-1}(2) \widetilde{\mathbf{W}}_x \widetilde{\mathbf{W}}_y \left((n^{-1} \mathbf{D}_y)^{-1} -f^{-1}(2)\right)\,.\nonumber
\end{align} 
For any fixed large constant $C_1>0$, denote $\Omega$ to be the event, where there exist $ i, j$ so that either of the followings happens
\begin{equation}\label{eq_largderivationbound}
|\xb_i^\top \xb_j-\delta_{ij}|>\sqrt{\frac{C_1 \log p_1}{p_1}}, \quad  |\yb_i^\top \yb_j-\delta_{ij}|>\sqrt{\frac{C_1 \log p_2}{p_2}}.
\end{equation}
By the large deviation inequality (see Lemmata \ref{lem_bernstein} and \ref{lem_chernoff} for details), we know that the probability of the event $\Omega$ is less than $n^{-C_1/2}$. Hence, it suffices to consider the set $\Omega^c$. We have the following results.
\begin{lemma}\label{lem_diagonal}
When conditional on $\Omega^c$, there exists some constant $C>0$ so that when $n$ is sufficiently large, we have 
\begin{equation*}
\norm{(n\mathbf{D}_x)^{-1}-f^{-1}(2) \mathbf{I}_n}  \leq C \frac{\log n}{\sqrt{n}}, 
\end{equation*}
\begin{equation*}
 \norm{(n\mathbf{D}_y)^{-1}-f^{-1}(2) \mathbf{I}_n} \leq C \frac{\log n}{\sqrt{n}}.
\end{equation*}
\end{lemma}
The proof of Lemma \ref{lem_diagonal} can be found in Appendix \ref{app:local}. By Proposition \ref{prop_norm},  when $n$ is sufficiently large, for some constant $C>0$, with probability at least $1-O(n^{-1/2})$, we have 
\begin{equation}
\norm{\widetilde{\mathbf{W}}_x} \leq C, \,\,  \norm{\widetilde{\mathbf{W}}_y} \leq C\,.
\end{equation}
Therefore, when conditional on $\Omega^c$, by Weyl's inequality, Lemma \ref{lem_diagonal} and \eqref{Equation Difference S1 and S2}, there exists some constant $C>0$ so that when $n$ is sufficiently large, we have
\begin{equation}\label{eq_2}
\left| \lambda_i(\mathbf{S}^{(1)}_{xy})-\lambda_i(\mathbf{S}^{(2)}_{xy}) \right| \leq C \frac{\log n}{\sqrt{n}}.
\end{equation} 
Hence, it suffices to study the eigenvalues of $\mathbf{S}^{(2)}_{xy}$. 
Recall (\ref{eq_matkotherform}) and define 
\begin{equation}
\breve{\Kb}_y:=-2f'(2) \Yb^\top \Yb+\varsigma \mathbf{I}\,,
\end{equation}
and define 
\begin{equation}\label{eq_s2}
 \mathbf{S}^{(3)}_{xy}=f^{-2}(2)\breve{\mathbf{K}}_x \breve{\Kb}_y\,, 
\end{equation}
which is clearly diagonalizable.
The following lemma says that the Stieltjes transforms of $\mathbf{S}^{(2)}_{xy}$ and $\mathbf{S}^{(3)}_{xy}$ are close.
\begin{lemma} \label{lem_2.4.2}
For any $z=E+ \mathrm{i}\eta \in \mathsf{S}(1/4,\tau)$ and a large  integer $p$, there exists some constant $C>0$ so that when $n$ is sufficiently large, we have
\begin{equation}
 \mathbb{E}\left|m_{\mathbf{S}^{(2)}_{xy}}(z)-m_{\mathbf{S}^{(3)}_{xy}}(z)\right|^p \leq  \left(\frac{C \log n}{\sqrt{n} \eta^2} \right)^p\,. 
\end{equation}
\end{lemma}
The proof of Lemma \ref{lem_2.4.2} is postponed to Appendix \ref{app:local}. Next, we state the local law for the product of random matrices, and its proof is postponed to Appendix \ref{app:bao}. 

\begin{lemma} \label{lem_2.4.3} 
For any $z=E+ \mathrm{i}\eta \in \mathsf{S}(1/4,\tau)$ and a large integer $p$, there exists some constant $C>0$ so that when $n$ is sufficiently large, we have
\begin{equation}
\mathbb{E}\left|m_{\mathbf{S}^{(3)}_{xy}}(z)-m_{\mathbf{Q}_{xy}}(z)\right|^p \leq \left(\frac{C \log n}{ \sqrt{n} \eta^2} \right)^p\,,
\end{equation}
where $\mathbf{Q}_{xy}$ is defined in \eqref{eq_defnd}. Furthermore, we have that
\begin{equation}
m_{\mathbf{Q}_{xy}}(z)=m_{\Sigma_x  \boxtimes \Sigma_y}(z)+O_{\prec}\Big(\frac{1}{n \eta}\Big). 
\end{equation}
\end{lemma}

It clear that  from Lemmata \ref{lem_2.4.2} and \ref{lem_2.4.3}, the Stieltjes transforms of $\mathbf{S}^{(2)}_{xy}$ and $\mathbf{Q}_{xy}$ are close; that is, for any $z=E+\mathrm{i}\eta \in \mathsf{S}(1/4,\tau)$, when $n$ is sufficiently large
\begin{equation}\label{diff:S2xy and Qxy}
\mathbb{E}\left| m_{\mathbf{S}^{(2)}_{xy}}(z)-m_{\mathbf{Q}_{xy}}(z) \right|^p \leq \left(\frac{C \log n}{\sqrt{n} \eta^2} \right)^p,
\end{equation}
for some constant $C>0$.
With the above reduction, we follow the argument similar to that of Theorem \ref{lem_rigiditydistance} to conclude our proof.  We prove analogous results of Propositions \ref{prop_norm} and \ref{prop_largscaleset} for the free multiplicative convolution.  We summarize them as the following Proposition \ref{pro_ncca_main} and leave the proof to Appendix \ref{section proof of propositions prop_largscaleset}.

\begin{proposition}\label{pro_ncca_main} 
For any interval $I \subset \mathbb{R}$, with probability at least $1-O(n^{-1/2})$,  we have uniformly that
\begin{equation}\label{eq_fffbb}
\left| \mu_{\mathbf{S}^{(2)}_{xy}}(I)-\mu_{\Sigma_x \boxtimes \Sigma_y}(I) \right| \leq Cn^{-1/4+\epsilon}\,,
\end{equation}
where $\epsilon>0$ is a small constant and $C>0$.
Denote the right-most endpoint of the support of $\mu_{\Sigma_x \boxtimes \Sigma_y}$ by $\gamma_{+}$. We have with probability at least $1-O(n^{-1/2})$ 
\begin{equation} \label{eq_bound}
\left|\norm{\mathbf{S}^{(2)}_{xy}}-\gamma_{+} \right| \leq Cn^{-1/9}\,. 
\end{equation}
\end{proposition}

With Lemma \ref{lem_2.4.3}, Proposition \ref{pro_ncca_main} and the square root behavior of free multiplicative convolution stated in (ii) of Lemma \ref{lem_addtivemeasurelemma}, we can repeat the proofs heading from (\ref{eq_deff}) to (\ref{eq_comprae}) to show that for $1 \leq i \leq n/2$, there exist some constants $C,C_1>0$ and small fixed constants $\epsilon>0$ and $\vartheta>\epsilon$, { 
\begin{align}\label{Proof theorem 3.4 Qxy gamma difference}
\left| \lambda_i(\mathbf{Q}_{xy})-\frac{\gamma_{\Sigma_x \boxtimes \Sigma_y}(i)}{f^2(2)} \right|  \leq
\begin{cases}
\ C n^{-1/9+2\vartheta} &  1 \leq i \leq C_1 n^{5/6+3\vartheta/2} ; \\
\ C n^{1/12+\epsilon} i^{-1/3} &  C_1 n^{5/6+3\vartheta/2} \leq i \leq n/2 \,.
\end{cases}
\end{align}
}
We omit details here.  
To complete the proof for the NCCA matrix, we need the following rigidity estimate, which is an analog of \cite[Lemma 3.13]{bao20171}. We leave its proof to Appendix \ref{OS:Proof:Lemma5.9}.
\begin{lemma}\label{lem_proofrigidityest} 
For any fixed $L \in \mathbb{N}$, we have 
\begin{equation}
|\gamma_{\Sigma_x \boxtimes \Sigma_y}(i)-\gamma_{\nu_{xy}}(i)| \leq  n^{-4/9+\delta}, \ \text{for all} \ 3 \leq i \leq L\,, 
\end{equation}
where $\delta>0$ is a small constant. 
\end{lemma}
Therefore, together with (\ref{eq_1}), (\ref{eq_2}), {(\ref{diff:S2xy and Qxy})} and \eqref{Proof theorem 3.4 Qxy gamma difference}, we conclude that 
\begin{align} 
|\lambda_i(n^2 \mathbf{S})-\gamma_{\nu_{xy}}(i) |=&\,|\lambda_i(\mathbf{Q}_{xy})-\gamma_{\Sigma_x \boxtimes \Sigma_y}(i)| \nonumber\\
&+ |\gamma_{\nu_{xy}}(i)-\gamma_{\Sigma_x \boxtimes \Sigma_y}(i)| +O_{\prec}(n^{-1/2})\,,  \label{eq_3}
\end{align}
when $3 \leq i \leq n/2$;
by combining (\ref{eq_3}) and  Lemma \ref{lem_proofrigidityest} with $j=1$, we conclude our proof for the NCCA matrix.

To finish the proof, we prove the same result for the AD algorithm. Denote 
\begin{equation}
\Ab^{(1)}_{xy}:=(n\Db_x)^{-1} \widetilde{\Wb}_x (n \Db_y)^{-1} \widetilde{\Wb}_y\,.
\end{equation}
Then by a discussion similar to (\ref{eq_1}), we find that for $ i \geq 3$, \begin{equation}
\left| \lambda_i(n^2 \Ab_{xy})-\lambda_{i-2} (\Ab_{xy}^{(1)})  \right|=O_{\prec}\big(n^{-1/2}\big)\,. 
\end{equation}
Recall (\ref{eq_s2xy}). Note  
\begin{align}
\mathbf{A}^{(1)}_{xy}-\mathbf{S}^{(2)}_{xy}=&\, \left( (n \mathbf{D}_x)^{-1}-f^{-1}(2) \right) \widetilde{\mathbf{W}}_x (n \mathbf{D}_y)^{-1} \widetilde{\mathbf{W}}_y \nonumber \\
&+f^{-1}(2) \widetilde{\mathbf{W}}_x \left((n \mathbf{D}_y)^{-1} -f^{-1}(2)\right) \widetilde{\mathbf{W}}_y . \label{eq_addecom}
\end{align} 
Recall $\Omega$ in the proof of Theorem \ref{thm_edgevalue} satisfying (\ref{eq_largderivationbound}) so that $\Omega^c$ is a high probability event. By (\ref{eq_addecom}), Lemma \ref{lem_diagonal}, and a discussion similar to (\ref{eq_2}), when conditional on $\Omega^c$, for some constant $C>0$, we have
\begin{equation}
\left| \lambda_i(\Ab^{(1)}_{xy})-\lambda_i(\Sb_{xy}^{(2)}) \right| \leq C \frac{\log n}{\sqrt{n}} \,,
\end{equation}
when $n$ is sufficiently large. Then, the proof for the AD matrix follows from Proposition \ref{pro_ncca_main} and Lemma \ref{lem_proofrigidityest}. The proof of Theorem \ref{thm_edgevalue} is now completed.
\end{proof}

\section{Discussion}\label{sec:discussion}

We study the spectral properties of two kernel-based sensor fusion algorithms, including NCCA and AD under the null case in the high dimensional regime. The local laws of random matrices and the free probability framework are applied for this purpose.
Moreover, we provide the convergence rates of the eigenvalues of the affinity matrix commonly encountered in the machine learning society; for example, the spectral clustering algorithm or nonlinear dimensional reduction.
To the best of our knowledge, this is the first work applying the local laws of random matrices and the free probability framework to study nonlinear kernel-based machine learning algorithms. Since the sensor fusion problem gets more and more attention in scientific fields, we expect to have more opportunities to apply the free probability framework and local laws of random matrices in our future research. 

The current study opens a door to several interesting problems. 
First, the number of sensor fusion algorithms is growing, ranging from linear to nonlinear and two to multiple sensors, and it might not be reasonable to analyze each algorithm separately. For example, while the symmetric and anti-symmetric AD look similar to AD, the analysis in the current paper cannot be directly applied to study them. May we develop a unified framework to study those algorithms? 
Second, in the non-null case, that is, when the data is contaminated by noise, to the best of our knowledge, there is limited work studying even the single sensor case. Specifically, while we have some knowledge about how the kernel-based affinity matrix behaves when the dataset is noisy \cite{ElKaroui:2010a,ElKaroui_Wu:2016b}, a full understanding of the distribution of its eigenvalues and eigenvectors is still missing. A new analysis scheme to study the non-null case is in urgent demand.  
Third, recall that the AD matrix is a row stochastic matrix. As a product of two kernel-based affinity matrices, it is in general not diagonalizable. We can view it as a kernel-based affinity matrix with an asymmetric kernel. 
This viewpoint leads to a natural question -- if the kernel is not symmetric, how much can we say about the spectrum, or pseudospectrum, of the kernel-based affinity matrix? This problem is not unique to the sensor fusion problem. It appears more naturally in other algorithms, like the locally linear embedding (LLE) \cite{Wu_Wu:2019}. In LLE, the established ``affinity matrix'' is in general not symmetric due to the asymmetric data geometric structure. This naturally links LLE to an asymmetric kernel. See Figure 2.1 in \cite{Wu_Wu:2019} for an example of the spectrum of LLE under the null case.
Fourth, the whole argument in the current paper can be carried over to the subgaussian, or more general setup. However, as is discussed in Remark \ref{Remark:nonHaar}, to our knowledge, while people in general agree that the eigenvector of the Gram matrix asymptotically follow the Haar distribution, we cannot find a paper providing a precise proof of this statement. Once we establish this proof, the whole statement in this paper can be carried over directly.
Finally, there are several technical challenges we encountered in this work that deserve further explorations. Due to the nature of the entrywise perturbation argument when we study the boundedness of the affinity matrix norm, we can only achieve an $O_p$ statement. A different approach is needed if we want to establish a stochastic dominant statement. The spectral parameter set considered in the microscopic spectral analysis is $\mathsf{S}(1/4,\tau)$, which might not be optimal. This limitation comes from the Hoffman-Wielandt inequality employed to approximate the Stieltjer transform. A more efficient bound is needed in order to achieve the expected optimal set, $\mathsf{S}(1,\tau)$. 
We will explore these problems in future work.

\section*{Acknowledgments} 
The authors  would like to thank Hong Chang Ji  for helpful discussions on the local laws of free multiplicative convolution. X. Ding would like to thank the support and hospitality of CRM when he attended the second part of the workshop ``New Developments in Free Probability and Applications''. { The authors also want to thank the Associated Editor and two anonymous referees for their comments, which have improved the paper significantly. }

\bibliographystyle{abbrv}
\bibliography{ncca}

\appendix

	\setcounter{page}{1}
	\renewcommand{\thepage}{SI.\arabic{page}}
	\renewcommand{\thesection}{\Alph{section}}

\section{Preliminary results}\label{sec:app:pre}
In this section, we summarize some preliminary results that will be used in the proofs. First of all, we record the Bernstein's inequality for Gaussian random variables. It is used for controlling the norm of a Gaussian random vector in Appendix \ref{app:local}.   

\begin{lemma}(Bernstein's concentration inequality, \cite[Example 2.11]{mw2019high}) \label{lem_bernstein}
Let $x_1, \ldots, x_n$ be independent standard Gaussian random variables. Then for every $0<t<1$, we have 
\begin{equation}
\mathbb{P}\left\{ \Big|\frac{1}{n} \sum_{i=1}^n x^2_i-1 \Big| \geq t \right\} \leq  2\exp \left(-nt^2/8 \right).
\end{equation}
\end{lemma}
Further, we need the following Chernoff bound to control $|\xb_i \xb_j^\top|, i \neq j$.
\begin{lemma} \label{lem_chernoff} 
For an i.i.d. sequence $\{\xb_i\}_{i=1}^n\subset \mathbb{R}^{p_1}$ satisfying 
\begin{equation} \label{eq_xyassum}
\mathbb{E} \xb_i= \mathbf{0}_{p_1},\, \text{Cov}(\xb_i) =p_1^{-1}\Ib_{p_1}
\end{equation}
and
\begin{equation}\label{eq_suggaussian}
\mathbb{E}(\exp(s\sqrt{p_1} x_{ij})) \leq \exp(s^2/2)\,,
\end{equation}
where $\xb_i=(x_{i1},\ldots,x_{ip_1})^\top$, we have for $i \neq j$ and $t>0$ that 
\begin{equation}\label{eq_offcontrol}
\mathbb{P}\big( |\xb_i^\top \xb_j|> t \big) \leq \exp(-p_1 t^2 /2)\,. 
\end{equation}
\end{lemma}

Note that in general a random variable satisfying \eqref{eq_xyassum} and \eqref{eq_suggaussian} is called sub-Gaussian. We call a random vector whose entries are sub-Gaussian random variables a sub-Gaussian random vector. If $\mathbf{x}_i\in \mathbb{R}^{p_1}$ is a sub-Gaussian random vector, we also mean that \cite{HKZ} 
\begin{equation}
\mathbb{E} (\exp( \sqrt{p}_1 \mathbf{a}^\top \mathbf{x}_i)) \leq \exp(\| \mathbf{a} \|_2^2/2)\,.\label{eq_subgaussianseconddefintion}
\end{equation}
A Gaussian random variable is clearly sub-Gaussian and a Gaussian random vector is also sub-Gaussian. For the sake of self-containedness, we provide the proof.

\begin{proof}[Proof of Lemma \ref{lem_chernoff}]
Note that \eqref{eq_suggaussian} means the entries of $\xb_i$ are sub-Gaussian random variables. First of all, for $t  \in (0,1)$, we have that 
\begin{align}
\mathbb{E}\Big( \exp\left(\frac{t p_1 \norm{\xb_i}_2^2}{2}\right)\Big) \nonumber 
&=\frac{1}{(2 \pi t)^{p_1/2}} \int_{\mathbb{R}^{p_1}} \exp(-\norm{\ab}_2^2/2t) \mathbb{E} (\exp(\sqrt{p}_1\ab^\top \xb_i)) d \ab\nonumber \\
& \leq \frac{1}{(2 \pi t)^{p_1/2}} \int_{\mathbb{R}^{p_1}} \exp(-\norm{\ab}_2^2/2t+\norm{\ab}_2^2/2) d \ab\label{eq_preboundXXX} \\
&=\frac{1}{(2 \pi(1-t))^{p_1/2}}\,,\nonumber 
\end{align}
where the first identity comes from $$(2\pi t)^{-p_1/2}\int_{\mathbb{R}^{p_1}}e^{-\|\ab\|^2_2/(2t)+\sqrt{p_1}\ab^\top \xb_i-tp_1\|\xb_i\|_2^2/2}d\ab=1.$$
By (\ref{eq_subgaussianseconddefintion}), we find that 
\begin{equation}\label{eq_bbb222XXX}
\mathbb{E}[\exp(t p_1 \xb_i^\top \xb_j)\vert \mathbf{x}_i ] \leq \exp(t^2 p_1 \| \mathbf{x}_i \| /2 ).
\end{equation}
As a consequence, we find that
\begin{align}
\mathbb{E}[\exp(t p_1 |\xb_i^\top \xb_j|)] & \leq \mathbb{E}[\exp(t^2 p_1 \norm{\xb_i}_2^2)] \nonumber\\
& \leq \frac{1}{(2 \pi (1-t^2))^{p_1/2}}\,,\nonumber
\end{align} 
where for the first inequality we use the independence of $\xb_i$ and $\xb_j$ when $i \neq j$ and conditional expectation in light of (\ref{eq_bbb222XXX}) and the second inequality follows from (\ref{eq_preboundXXX}).
Therefore, together with Markov's inequality, for $\epsilon>0$  and $t\in(0,1)$, we have that
\begin{align}
\mathbb{P} \Big( |\xb_i^\top \xb_j|>\epsilon \Big) &=\mathbb{P} \Big(\exp(tp_1  |\xb_i^\top \xb_j|)>\exp(t p_1 \epsilon) \Big) \nonumber\\
& \leq \min_{t \in (0,1)}    \frac{\exp(-t p_1 \epsilon)}{(1-t^2)^{p_1/2}}  \,.\nonumber
\end{align}
Observe that
\begin{align}
\frac{1}{e^{t p_1 \epsilon}} \Big( \frac{1}{\sqrt{1-t^2}} \Big)^{p_1}&=\exp(-tp_1 \epsilon-[p_1 \log (1-t^2)]/2) \nonumber\\
&\leq \exp(-t p_1\epsilon+p_1 t^2/2)\,,\nonumber
\end{align} 
where we use the fact that 
\begin{equation*} 
\log (1-t) \leq -t 
\end{equation*}
when $t \in (0,1)$.
By setting $t=\epsilon$,  we have 
\begin{equation*}
\min_{t \in (0,1)} \exp(-t p_1 \epsilon) \Big(\frac{1}{\sqrt{1-t^2}} \Big)^{p_1} \leq \exp(-p_1 \epsilon^2/2).
\end{equation*}
This finishes our proof. 
\end{proof}

Next, we summarize the {\em averaged local law} for sample Gram matrices as described in Section \ref{sec:preliminary}. Let $\Sigma \in \mathbb{R}^{p \times p}$ be a positive definite matrix with eigenvalues $\lambda_1 \geq \ldots \geq \lambda_p$. We assume that for some (small) constant $\tau>0$, 
\begin{equation}
\lambda_1 \leq \tau^{-1}\,\mbox{  and  }\,\mu_\Sigma([0,\tau]) \leq 1- \tau\,.
\end{equation}
This means that the spectrum of $\Sigma$ is not concentrated in $0$. 
It is known \cite[Lemmata 2.4 and 2.5]{Knowles2017} that for the $g$ defined in (\ref{eq_defnf}), there exist $2q$ critical points for some positive integer $q$; that is, there exist $x_1 \geq \ldots \geq x_{2q}$ so that $g'(x_i)=0$ for $i=1,\ldots,2q$. Set $a_i:=g(x_i)$ for $i=1,\ldots,2q$. We say that there exist \emph{$q$ bulk components} for $\Sigma$ so that the $i$-th bulk is supported on $ [a_{2i}, a_{2i-1}]$. Next, we introduce the key assumptions which make the local law valid. 
\begin{assumption} [Regularity condition, Definition 2.7 of \cite{Knowles2017}] \label{assu_regularity} 
Fix $\tau>0$. 
\begin{enumerate}
\item[(i)] We say that the edge $a_k$, $k=1,2,\ldots,2q$, is regular if 
\begin{equation}
a_{k} \geq \tau, \,\, \min_{l \neq k} |a_k-a_l| \geq \tau, \,\, \min_i|x_k+\lambda_i^{-1}| \geq \tau.
\end{equation}
\item[(ii)] We say that the $k$-th bulk component, $k=1,\ldots, q$, is regular if for any fixed $\tau'>0$ there exists a constant $c \equiv c_{\tau, \tau'}$, such that the asymptotical probability density associated with $m_{(\Sigma_{p_1}^{1/2} \Xb)^\top\Sigma_{p_1}^{1/2}\Xb}(z)$ in $[a_{2k}+\tau', a_{2k-1}-\tau']$ is bounded from below by $c$.
\end{enumerate}
\end{assumption}
Recall the spectral parameter set $\mathsf{S}(\alpha, \tau)$ defined in (\ref{eq_sprime}).
\begin{lemma}[Averaged local laws of sample covariance matrices \cite{Knowles2017}] \label{lem_locallaw} 
Suppose Assumption \ref{assu_regularity} holds true. Then
\begin{equation}
m_{(\Sigma_{p_1}^{1/2} \Xb)^\top\Sigma_{p_1}^{1/2}\Xb}(z)=m_{\mu_{c_1,\texttt{gMP}}}(z)+O_{\prec}\Big(\frac{1}{n \eta}\Big)\,,
\end{equation}
where $z \in \mathsf{S}(1,\tau)$.
\end{lemma}

\begin{remark} \label{rem:localawlarge} In the proofs of the present paper, we will mainly consider the following two different diagonal matrices:
\begin{equation}\label{exam_sigma0}
\Sigma:= \sigma^2 \mathbf{I}_p
\end{equation}
for some $\sigma>0$, or
\begin{equation}\label{exam_sigma}
\Sigma:=\Sigma_x \ \text{or} \ \Sigma_y\,,
\end{equation}
where $\Sigma_x$ and $\Sigma_y$ are defined in (\ref{eq_lambdan}) and (\ref{defn_sign}) respectively. 
For these two cases, it is easy to check that $(i)$ and $(ii)$ of Assumption \ref{assu_regularity} are satisfied with $q=1$. See, for instance, \cite[Example 2.9]{Knowles2017}. Furthermore, we are indeed dealing with a sample Gram matrix with an isotropic shift and it is clear that the averaged local law still holds true for such a shifted matrix. 
\end{remark}
The next result states that the eigenvalues of a sample Gram matrix are close to the quantiles of the generalized MP law in the setting of (\ref{exam_sigma}). The main idea is to control the randomness of the sample Gram matrix by the quantiles of the generalized MP law, which are deterministic quantities.

\begin{lemma}(Rigidity of eigenvalues, \cite[Theorems 3.14 and 3.15]{Knowles2017}) \label{lem_rigidity} 
In the settings of \eqref{exam_sigma0} or \eqref{exam_sigma}, denote the eigenvalues of $(\Sigma \Xb)^\top \Sigma\Xb $ as $\lambda_1 \geq \ldots \geq \lambda_n$, where $\Xb\in \mathbb{R}^{p\times n}$. 
Then we have for $1 \leq j \leq \min\{p,n\}$,
\begin{equation*}
|\lambda_j-\gamma_{\mu_{c,\texttt{gMP}}}(j)| =O_{\prec} (\tilde{j}^{-1/3} n^{-2/3})\,,
\end{equation*}
where $c=p/n$ and $\tilde{j}=\min\{n \wedge p+1-j,j\}$. 
\end{lemma}
Note that the results hold true immediately for the sample Gram matrix with an isotropic shift. In the next lemma, we summarize some relevant properties of the MP law, which have been stated in several random matrix literature; for instance, see \cite[equations (1.9) and (3.32)]{ding_2019}. %

\begin{lemma}[Properties of the Marchenko-Pastur law] \label{lem_mplaw} Denote the right-most edge for the MP law as $\lambda_+$. 
Take $z \in \{z=E+\mathrm{i} \eta: |E-\lambda_+| \leq \tau,  \ n^{-1+\tau}\leq  \eta \leq \tau^{-1}\}$, where $\tau>0$ is a small constant defined in $\mathsf{S}(\alpha, \tau)$. Denote $\kappa=|E-\lambda_+|$. We have 
\begin{equation*}
\operatorname{Im}m_{\mu_{c_1,\sigma^2}}(z) \sim
\left\{\begin{array}{ll}
\sqrt{\kappa+\eta}& \ \text{if} \ E \leq \lambda_+, \\
\frac{\eta}{\sqrt{\kappa+\eta}}& \ \text{if} \ E \geq \lambda_+\,.
\end{array}\right.
\end{equation*} 
Denote the density function associated with $\mu_{c_1,\sigma^2}$ as $\rho_{c_1,\sigma^2}$. Then, for some small constant $\delta>0$, we have that for all $x \in [\lambda_+-\delta, \lambda_+]$,
\begin{equation}\label{eq)squareroot}
\rho_{c_1,\sigma^2}(x) \sim \sqrt{\lambda_+-x}. 
\end{equation} 
\end{lemma}
In the literature, \eqref{eq)squareroot} is usually referred to the {\em square root behavior near the edge}.
Again, the above results hold true for isotropic shifted sample Gram matrix with a slight modification. We next summarize the results for the functional calculus, which is used to prove the rigidity of eigenvalues.

\begin{lemma} [Helffer-Sj{\H o}strand formula, Appendix C of \cite{locallawnote}]\label{lem_funcal} Let $n \in \mathbb{N}$ and $h \in \mathcal{C}^{(n+1)}(\mathbb{R})$. Define the {\em almost analytic extension of $h$ of degree $n$} through 
\begin{equation}
\tilde{h}_n(x+\mathrm{i}y):=\sum_{k=0}^n \frac{1}{k!} (\mathrm{i} y)^k h^{(k)}(x)\,,
\end{equation}
where $x,y\in \mathbb{R}$ and $h^{(k)}$ means the $k$-th derivative of $h$.
Let $\chi \in \mathcal{C}_c^{\infty}(\mathbb{C}, [0,1])$ be a smooth cutoff function. Then for any $\lambda \in \mathbb{R}$ satisfying $\chi(\lambda)=1$ we have 
\begin{equation}
h(\lambda)=\frac{1}{\pi} \int_{\mathbb{C}} \frac{\bar{\partial} (\tilde{h}_n(z) \chi(z))}{\lambda-z} d z\,,
\end{equation}
where $d z$ denotes the Lebesgue measure on $\mathbb{C}$ and $
\bar{\partial}:=\frac{1}{2}(\partial_x+\mathrm{i} \partial_y)$ is the antiholomorphic derivative. Specifically, for a Hermitian matrix $\mathbf H$, we have 
\begin{equation}\label{rem:d4}
h(\mathbf H)=\frac{1}{\pi} \int_{\mathbb{C}} \bar{\partial}(\tilde{h}_n(z) \chi(z)) G_{\mathbf H}(z) d z,
\end{equation}
provided $\chi$ is chosen such that $\chi=1$ on the spectrum of $\mathbf H$. 
\end{lemma}

Next, we collect some important matrix identities and inequalities for reference.

\begin{lemma}[Collection of matrix identities and inequalities]\label{lem_collecttion} Take two $n \times
 n$ Hermitian matrices $\Ab$, $\Bb$ and $z=E+\mathrm{i} \eta \in \mathbb{C}^+$. Let $\{\lambda_i(\Ab)\}$ and $\{\lambda_i(\Bb)\}$ be the eigenvalues of $\Ab$ and $\Bb$ respectively and $\mu_{\Ab}$ and $\mu_{\Bb}$ be the ESDs of $\Ab$ and $\Bb$ respectively. 
\begin{enumerate}
\item[(i)] (Hoffman-Wielandt inequality \cite[equation (1.67)]{taotopics})   We have 
\begin{equation*}
\sum_{i=1}^n |\lambda_i(\Ab)-\lambda_i(\Bb)|^2 \leq \text{tr}\{ (\Ab-\Bb)^2 \}\,.  
\end{equation*}
\item[(ii)] (Bound Stieltjes transform using rank \cite[Lemma F.5]{locallawnote}) We have  
\begin{equation*}
|m_\Ab(z)-m_\Bb(z)| \leq \frac{\text{rank}\{(\Ab-\Bb)\}}{n}  \min \left \{ \frac{2}{\eta}, \ \frac{\norm{\Ab-\Bb}}{\eta^2} \right \}\,.
\end{equation*}
\item[(iii)] (Bound Levy distance of ESDs of matrices \cite[Theorem A.45]{bai2010spectral}) Denote 
the Levy distance of $\mu_\Ab$ and $\mu_\Bb$ as $\mathcal{L}(\mu_\Ab, \mu_\Bb)$. We have 
\begin{equation*}
\mathcal{L}(\mu_\Ab, \mu_\Bb) \leq \norm{\Ab-\Bb}\,.
\end{equation*}
\item[(iv)] (Bound trace of product of matrices \cite[equation (1)]{FLF}) Suppose that $\Ab$ and $\Bb$ are positive semi-definite. We have 
\begin{equation*}
\text{tr}(\Ab \Bb)\leq \lambda_1(\Ab) \text{tr} \Bb\,. 
\end{equation*} 
\item[(v)] (Weyl's inequality for perturbed Hermitian matrices, \cite[Section 1.3]{taotopics}) For $1 \leq i,j \leq n$ so that $i+j-1 \leq n$, we have 
\begin{equation*}
\lambda_{i+j-1}(\Ab+\Bb) \leq \lambda_i(\Ab)+\lambda_j(\Bb)\,. 
\end{equation*}
\item[(vi)] (Bound Levy distance \cite[Proposition 4.13]{Ber93}) For any probability measures $\mu_{\Ab}$, $\mu_{\Bb}$, $\mu_{\alpha}$, $\mu_{\beta}$ on $\mathbb{R}$, we have
\begin{align*}
&\mathcal L(\mu_{\alpha \boxplus \beta}, \mu_{\Ab \boxplus \Bb}) \leq \mathcal L(\mu_\Ab, \mu_{\alpha})+\mathcal L(\mu_\Bb, \mu_{\beta})\,,\\
&\mathcal L(\mu_{\alpha \boxtimes \beta}, \mu_{\Ab \boxtimes \Bb}) \leq \mathcal L(\mu_\Ab, \mu_{\alpha})+\mathcal L(\mu_\Bb, \mu_{\beta})\,.\nonumber
\end{align*}
\end{enumerate}
\end{lemma}

\begin{lemma}(Approximation of distance kernel matrix)\label{lem_matrixapproximation} Consider the $n \times n$ matrix $\mathbf{W}_x$ defined in (\ref{eq_defnmxmy}).  Recall $\psi=(\psi_1,\ldots, \psi_n)$ with $\psi_i=\norm{\xb_i}_2^2-1, i=1,2,\cdots, n$. Let 
\begin{align}
\mathbf{K}^d_x&=\breve{\mathbf K}_x+\mathrm{Sh}_0+\mathrm{Sh}_1+\mathrm{Sh}_2\,.
\end{align} 
Then for some small constants  $c>0$ and $\varsigma>0$, when $n$ is sufficiently large,  we have that with probability at least $1-O(n^{-1/2-c})$
\begin{equation}
\norm{\mathbf{W}_x-\Kb^d_x} \leq n^{-\varsigma} \,. 
\end{equation}
\end{lemma}
In \cite{elkaroui2010}, the author stated the result without keeping track of the probability under a more general setup. Our proof replies on the proof of \cite[Theorem 2.2]{elkaroui2010} by keeping track of the detailed probability. 
\begin{proof} [Proof of Lemma \ref{lem_matrixapproximation}] First of all, since we are dealing with Gaussian random variables, in the \emph{truncation and centralization} step of the proof of \cite[Theorem 2.2]{elkaroui2010}, we shall choose the level $B_p=p^{-\delta}$. Now we plug $B_p$ into the equations regarding the computation of probabilities, for instance, the last equation on Page 25 and the first four equations on Page 26 of \cite{elkaroui2010}. We can conclude the proof by using Chebyshev's inequality. 
\end{proof}

\section{A brief summary of free probability theory}\label{sec:appen:freeprob}
For the sake of self-containedness, in this section, we summarize some notations and preliminary results from the free probability theory that are used in this paper. For more details, we refer the readers with interest to \cite{freebook,Rao2008,Tulino2004}. We first introduce some definitions.  
\begin{definition}[$\eta$-transform] For a probability measure $\mu$ supported on $(0,\infty)$, the $\eta$-transform of $\mu$ is defined as 
\begin{equation}
\eta_{\mu}(z)=\int \frac{1}{1+z\lambda} \mu(d\lambda), \ z \in \mathbb{C}^+.
\end{equation} 
\end{definition}
It is easy to see that $\eta$-transform is closely related to the Stieltjes transform via
\begin{equation*}
\eta_{\mu}(z)=\frac{1}{z}m_{\mu}\Big(-\frac{1}{z}\Big)\mbox{  and  }m_{\mu}(z)=-\frac{1}{z}\eta_{\mu}\Big(-\frac{1}{z}\Big)\,.
\end{equation*} 
\begin{definition}[$S$-transform] For a probability measure $\mu$ supported on $(0,\infty)$, the $S$-transform of $\mu$ is defined as 
\begin{equation*}
S_{\mu}(z)=-\frac{z+1}{z} \eta_{\mu}^{-1} (z+1), \ z \in \mathbb{C}^+\,.
\end{equation*}
 \end{definition}
 It is remarkable that the $S$-transform of $\mu_{c,1}$ is (see, e.g., \cite{Tulino2004}) 
 \begin{equation*}
 S_{\mu_{c,1}}(z)=\frac{1}{1+cz}\,.
\end{equation*}  

In the proof of local laws for products of random matrices, it is more convenient to use the \emph{$M$-transform}.  We first introduce its definition. 
\begin{definition}[$M$-transform]\label{defn_mtransform}  For a probability measure $\mu$ supported on the positive real line, the $M$-transform of $\mu$, denoted as $M_\mu:\mathbb{C}\backslash [0,\infty)\to\mathbb{C}$, is defined as  
\begin{equation}
M_{\mu}(z):=1-\left( \int \frac{x}{x-z} \mu(dx) \right)^{-1}\,.
\end{equation}
\end{definition}

It is clear that the $M$-transform and the Stieltjes transform are related via
\begin{equation}
M_{\mu}(z)=\frac{z m_{\mu}(z)}{1+z m_{\mu}(z)}\, , \ z \in \mathbb{C}^+\,.
\end{equation}

Next, we introduce the additive and  multiplicative subordination functions, which can be used to characterize the system of Stieltjes transforms of  addition and product of two random matrices. For details, we refer readers with interest  to \cite{Belinschi2007, CDM}. 
Denote by $F_{\mu}$  the negative reciprocal of the Stieltjes transform, i.e., 
\begin{equation}\label{negative reciprocal relationship between F and m}
F_{\mu}(z):=-\frac{1}{m_{\mu}(z)},  \ z \in \mathbb{C}^+\,. 
\end{equation}
Note that $F_{\mu}: \ \mathbb{C}^+ \rightarrow \mathbb{C}^+$ is analytic such that 
\begin{equation}\label{eq_fovereta}
\lim_{\eta \uparrow \infty} \frac{F_{\mu}(\mathrm{i} \eta )}{\mathrm{i} \eta}=1\,. 
\end{equation}
Conversely, if $F: \mathbb{C}^+ \rightarrow \mathbb{C}^+$ is an analytic function such that (\ref{eq_fovereta}) holds, then $F$ is the negative reciprocal Stieltjes transform of a probability measure (see, e.g.  \cite{akhiezer1965classical}). We first  introduce the results on the free additive convolution. 

\begin{proposition}\label{prop_add}(Additive subordination property, Theorem 4.1 of \cite{Belinschi2007}) Given two Borel probability measures, $\mu$ and $\nu$, on $\mathbb{R}$, there exist unique analytic functions, $\omega_\mu, \omega_\nu: \mathbb{C}^+ \rightarrow \mathbb{C}^+$, such that 
\begin{enumerate}
\item[(i)] for all $z \in \mathbb{C}^+$, we have $\operatorname{Im} \omega_\mu(z)\geq \operatorname{Im} z$, $\operatorname{Im} \omega_\nu(z) \geq \operatorname{Im} z$, and 
\begin{equation}
\lim_{\eta \uparrow \infty} \frac{\omega_\mu(\mathrm{i} \eta)}{\mathrm{i} \eta}=\lim_{\eta \uparrow \infty} \frac{\omega_\nu( \mathrm{i} \eta)}{\mathrm{i} \eta}=1\,; 
\end{equation}
\item[(ii)] for all $z \in \mathbb{C}^+$, we have 
\begin{equation}\label{eq_subordination}
\omega_\nu(z)+\omega_\mu(z)-z=F_{\mu}(\omega_\nu(z))=F_{\nu}(\omega_\mu(z))\,. 
\end{equation}
\end{enumerate}
\end{proposition}

The functions $\omega_\mu$ and $\omega_\nu$ are referred to as the \emph{subordination functions}.  Note that the analytic function $F: \mathbb{C}^+ \rightarrow \mathbb{C}^+$ defined by  
\begin{equation}
F(z):=F_{\mu}(\omega_\nu(z))=F_{\nu}(\omega_\mu(z)),
\end{equation}
satisfies \eqref{eq_fovereta}, so it 
is the negative reciprocal of the Stieltjes transform of a probability measure, which is defined as the \emph{free additive convolution} of $\mu$ and $\nu$ and denoted as $\mu \boxplus \nu$.
Based on \eqref{negative reciprocal relationship between F and m} and \eqref{eq_subordination}, we  get the following system of subordination equations
\begin{align} 
& m_{\mu}(\omega_\nu(z))=m_{\nu}(\omega_\mu(z)) = m_{\mu \boxplus \nu}(z)\,,\label{eq_defnsub}  \\
& \omega_\mu (z)+\omega_\nu(z)-z=-\frac{1}{m_{\mu \boxplus \nu}(z)}\,.\nonumber 
\end{align}

Analogously, we have the results for the free multiplicative convolution of random matrices. 
Define the {\em $\psi$-transform} of a probability measure $\mu $ supported on $[0, \infty)$ as 
\begin{equation*}
\psi_{\mu}(z)=\int \frac{z \lambda}{1-z\lambda} \mu (d \lambda), \ z \in \mathbb{C}^+\,. 
\end{equation*}
We also define 
\begin{equation*}
\varphi_{\mu}(z)=\frac{\psi_{\mu}(z)}{1+\psi_{\mu}(z)}\,.
\end{equation*}

\begin{proposition} [Multiplicative subordination property, Section 2.2.2 of \cite{CDM}]\label{prop_muloriginal} Given two Borel probability measures $\mu, \nu$ supported on $[0, +\infty)$, there exist two unique analytic functions $\omega_\mu, \omega_\nu:  \ \mathbb{C} \backslash [0, +\infty) \rightarrow \mathbb{C} \backslash [0, +\infty)$ so that   
\begin{enumerate}
\item[(i)] $\pi>\operatorname{arg} \omega_\mu(z) \geq \operatorname{arg} z$ and $\pi>\operatorname{arg} \omega_\nu(z) \geq \operatorname{arg} z$ for $z \in \mathbb{C}^+$ \,;
\item[(ii)] $\omega_\mu(z)$ and $\omega_{\nu}(z)$ satisfy the following relation
\begin{equation}\label{eq_f1f2}
\frac{\omega_\mu(z)\omega_\nu(z)} {z}=\varphi_{\mu}(\omega_\nu(z))=\varphi_{\nu}(\omega_\mu(z)).
\end{equation}
\end{enumerate}
The functions $\omega_\mu$ and $\omega_\nu$ are called the \emph{subordination functions}. 
\end{proposition}

The analytic function $\varphi: \mathbb{C}^+ \rightarrow \mathbb{C}^+$ defined by  
\begin{equation}
\varphi(z):=\varphi_{\mu}(\omega_\nu(z))=\varphi_{\nu}(\omega_\mu(z)),
\end{equation}
is related to the Stieltjes transform of a probability measure, which is defined as the \emph{free multiplicative convolution} of $\mu$ and $\nu$ and denoted as $\mu \boxtimes \nu$. 
To make clear the relationship between Proposition \ref{prop_muloriginal} and $M$-transform, we need the following facts (see also \cite{JHC}).
Suppose the assumptions of Proposition \ref{prop_muloriginal} hold true. 
\begin{proposition}[Proposition 2.8 of \cite{JHC}]\label{prop_muloriginal 2about M} 
When neither $\mu$ nor $\nu$ is $\delta_0$, we define two analytic functions $\Omega_{\mu}(z), \ \Omega_{\nu}(z)$ as 
\begin{equation}\label{eq_omega}
\Omega_{\mu}(z):=\frac{1}{\omega_{\mu}(1/z)}, \ \Omega_{\nu}(z):=\frac{1}{\omega_{\nu}(1/z)}. 
\end{equation}
The analytic functions $\Omega_{\mu}$ and $\Omega_{\nu}$ map $\mathbb{C}^+$ into itself and satisfy:
\begin{enumerate}
\item[(i)] $\text{arg} \ \Omega_{\mu}(z) \geq \text{arg} \ z$ and $\text{arg} \ \Omega_{\nu}(z) \geq \text{arg} \ z$; 
\item[(ii)] $M_{\mu}(\Omega_{\nu}(z))=M_{\nu}(\Omega_{\mu})(z)=M_{\mu \boxtimes \nu}(z)$; 
\item[(iii)] $\Omega_{\mu}(z) \Omega_{\nu}(z)=z M_{\mu \boxtimes \nu}(z)$. 
\end{enumerate}
\end{proposition}
Based on Propositions \ref{prop_muloriginal} and \ref{prop_muloriginal 2about M}, we have 
\begin{equation*}
\varphi(z)=\frac{\omega_\mu(z)\omega_\nu(z)} {z}=\frac{1}{z\Omega_\mu(1/z)\Omega_\nu(1/z)}=\frac{1}{z^2M_{\mu\boxtimes \nu}(1/z)}\,,
\end{equation*}
which is linked back to the Stieltjes transform of $\mu\boxtimes \nu$ by the definition of the $M$-transform. Furthermore, when $z\in \mathbb{C}\backslash [0,\infty)$, we denote 
\begin{equation}\label{NotationL definition}
L_{\mu \boxtimes \nu}(z):=\frac{M_{\mu \boxtimes \nu}(z)}{z}\,. 
\end{equation}
It is easy to see that 
\begin{equation*}
\varphi(1/z)=\frac{z}{L_{\mu \boxtimes \nu}(z)}\,.
\end{equation*}


\section{Technical proofs of Lemmata \ref{lem_localkernel}, \ref{lem_diagonal} and  \ref{lem_2.4.2} and Proposition \ref{prop_norm}}\label{app:local}

\subsection{Proof of Lemma \ref{lem_localkernel}}\label{section: proof of lemma lem_localkernel}

Define two matrices $\mathbf{O},\mathbf{P} \in \mathbb{R}^{ n \times n}$ by
\begin{equation*}
\mathbf{O}(i,j):=\norm{\xb_i}_2^2+\norm{\xb_j}_2^2-2\,, \ \mathbf{P}(i,j):=\xb_i^\top \xb_j
\end{equation*}
for $i,j=1,\ldots,n$. Note that $\mathbf{P}=\Xb^\top\Xb$.
By the Taylor expansion, when $i \neq  j$, we have 
\begin{align}
\mathbf{W}_x(i,j)=&\,f(2)+f'(2)(\mathbf{O}(i,j)-2\mathbf{P}(i,j))\label{eq_decompositionw}\\
&+f''(\zeta_{ij})(\mathbf{O}(i,j)-2\mathbf{P}(i,j))^2\,,\nonumber
\end{align}
where $\zeta_{ij} \in ( \Wb_x(i,j) \wedge 2, \Wb_x(i,j) \vee 2) $. Clearly, $\mathbf{W}_x(i,i)=f(0)$ for $i=1,\ldots,n$. First of all, we notice that the matrix $\mathbf{O}$ can be written as 
\begin{equation*}
\mathbf{O}=\mathbf{1} \mathbf{v}^\top+\mathbf{v} \mathbf{1}^\top,
\end{equation*}
where $\mathbf{v}\in \mathbb{R}^n$ with the $i$-th entry $\norm{\xb_i}_2^2-1$ for $i=1,2,\ldots, n$.  We hence conclude that $\mathbf{O}$ is a rank-2 matrix. Define $\check{\Wb}_{x}\in \mathbb{R}^{n\times n}$ to be  
\begin{equation*}
\check{\Wb}_{x}(i,j):=\left\{
\begin{array}{ll}
f(2)-2f'(2)\mathbf{P}(i,j)+4f''(\zeta_{ij})\mathbf{P}(i,j)^2 &i\neq j\\
f(0)& i=j\,.
\end{array}
\right.
\end{equation*}
We first start with the control of the difference between $m_{\check{\Wb}_x}$ and $m_{\Kb_x}(z)$. The control of the difference between $m_{\Wb_x}$ and $m_{\check{\Wb}_x}$ enjoys a similar discussion.
 Denote the sequence of eigenvalues of $\check{\Wb}_{x}$ and $\mathbf{K}_x$ as $\{\lambda_i(\check{\Wb}_{x})\}_{i=1}^n$ and $\{\lambda_i(\mathbf{K}_x)\}_{i=1}^n$ in the decreasing order respectively. For $z=E+\mathrm{i}\eta \in \mathsf{S}(1/4,\tau)$, we have
\begin{align} 
&|m_{\check{\Wb}_{x}}(z)-m_{\Kb_x}(z)|^p\label{eq_stilebound}\\
\leq &\, \left( \frac{1}{n} \sum_{i=1}^n \frac{|\lambda_i(\check{\Wb}_{x})-\lambda_i(\mathbf{K}_x)|}{|\lambda_i(\check{\Wb}_{x})-z||\lambda_i(\mathbf{K}_x)-z|} \right)^{p}\,, \nonumber 
\end{align}
which, by the Cauchy-Schwarz inequality, is bounded by 
\begin{align}
&\, \frac{1}{n^p} \bigg[ \Big(\sum_{i=1}^n |\lambda_i(\check{\Wb}_{x})-\lambda_i(\mathbf{K}_x)|^2 \Big)\times \Big( \sum_{i=1}^n \frac{1}{|\lambda_i(\check{\Wb}_{x})-z|^2|\lambda_i(\mathbf{K}_x)-z|^2} \Big) \bigg]^{p/2} \nonumber \\
\leq &\, \frac{1}{n^{p/2} \eta^{2p}} ( \text{tr} \{(\check{\Wb}_{x}-\mathbf{K}_x)^2\} )^{p/2}, \label{eq_finalcontrolbound}
\end{align}
where the last inequality comes from the Hoffman-Wielandt inequality (see (i) of Lemma \ref{lem_collecttion}) and the trivial bound $|\lambda-z|^{-1}\leq \eta^{-1}$ for any $\lambda\in \mathbb{R}$. The rest of the proof leaves to control the right-hand side of (\ref{eq_finalcontrolbound}).


By Lemma \ref{lem_bernstein}, for $\epsilon>0$, we have  
\begin{equation}\label{eq_diagcontrol}
\mathbb{P} \big( \big| \norm{\xb_i}_2^2-1 \big| >\epsilon \big) \leq 2 \exp(-p_1 \epsilon^2/8)\,.
\end{equation}
We can therefore bound $\norm{\xb_i}, i=1,2,\ldots, n$, using the above inequality. To control $|\xb_i^\top \xb_j|, 1 \leq i \neq j \leq n$, we apply Lemma \ref{lem_chernoff}.  
With (\ref{eq_diagcontrol}) and (\ref{eq_offcontrol}), we now head to conclude our proof. Denote $\Omega$ to be the event such that there exists $ i \neq j $ such that the following event happens
\begin{equation} \label{Condition on the boundedness of inner product of xi and xj}
|\xb_i^\top \xb_j|>\sqrt{\frac{C_1 \log p_1}{p_1}}\,,  
\end{equation}
where $C_1\geq p$. 
We clearly have 
\begin{align}
\mathbb{E} \Big| m_{\check{\Wb}_{x}}(z)-m_{\Kb_x}(z) \Big|^p  =\, \mathbb{E} \left[ 1_{\Omega} |m_{\check{\Wb}_{x}}(z)-m_{\Kb_x}(z)|^p \right]
\,+\mathbb{E} \left[ 1_{\Omega^c} |m_{\check{\Wb}_{x}}(z)-m_{\Kb_x}(z)|^p \right].\label{eq_headstart}
\end{align}
By (\ref{eq_offcontrol}), we have that $\mathbb{P}(\Omega)=O(p_1^{-C_1/2})$  when $n$ is sufficiently large. Therefore, by the trivial bound (\ref{eq_stitrivial}),  we have 
\begin{align}
 \mathbb{E} \big[ 1_{\Omega} |m_{\check{\Wb}_{x}}(z)-m_{\Kb_x}(z)|^p \big]
 \leq\,  \mathbb{P}(\Omega)\times 2^p (\|m_{\check{\Wb}_{x}}(z)\|_\infty^p+\|m_{\Kb_x}(z)\|_\infty^p) \leq \frac{{C_2}2^{p}}{p_1^{{C_1}/2}\eta^{p}} \label{eq_inequality1}
\end{align}
{for some constant $C_2>0$.}
By (\ref{eq_finalcontrolbound}), we have that
\begin{align}
&\mathbb{E} \Big[ 1_{\Omega^c} |m_{\check{\Wb}_{x}}(z)-m_{\Kb_x}(z)|^p \Big] \label{eq_inequality2}\\
\leq&\, \mathbb{E} \Big[ 1_{\Omega^c} \frac{1}{n^{p/2} \eta^{2p}} \left( \sum_{i,j} (\check{\Wb}_{x}(i,j)-\Kb_x(i,j))^2 \right)^{p/2} \Big] \nonumber  \\
 \leq&\, \mathbb{E} \Big[1_{\Omega^c} \frac{4 {C_3}}{n^{p/2} \eta^{2p}}  \Big( \sum_{i \neq j} |\xb_i^\top \xb_j|^4 \Big)^{p/2} \Big] \nonumber\\
 \leq&\, \frac{4{C_4} (\log p_1)^{p} }{n^{p/2}  \eta^{2p}}\,,\nonumber
\end{align}
for some constants {$C_3,C_4>0$}, where the second inequality comes from the fact that when $i=j$, $\check{\Wb}_{x}(i,j)-\Kb_x(i,j)=0$, when $i\neq j$, $\check{\Wb}_{x}(i,j)-\Kb_x(i,j)=4f''(\zeta_{ij})(\mathbf{P}\circ \mathbf{P})_{ij}=4f''(\zeta_{ij})|\xb_i^\top \xb_j|^2$, and $f$ is twice continuously differentiable, and the last inequality comes from \eqref{Condition on the boundedness of inner product of xi and xj}.

We next discuss the control of the difference between $m_{\Wb_x}$ and $m_{\check{\Wb}_x}.$ Denote a $n\times n$ matrix $\Wb_{x}^c$ by
\begin{equation*}
\Wb^c_x(i,j):=\left\{
\begin{array}{ll}
\check{\Wb}_x(i,j)-4f^{''}(\zeta_{ij}) \Ob(i,j) \Pb(i,j) & i \neq j\\
f(0) & i=j\,.
\end{array}
\right.
\end{equation*}
First of all, since $\Ob$ is a rank-2 matrix, by using the rank bound, that is, $\text{rank}(\Ab \circ \Bb) \leq \text{rank}(\Ab) \text{rank}(\Bb)$, and (ii) of Lemma \ref{lem_collecttion}, we have that 
\begin{equation}\label{eq_rrrr}
|m_{\Wb_x}(z)-m_{\Wb_x^c}(z)| \leq \frac{12}{n \eta}.
\end{equation}
Hence, when $n$ is sufficiently large, we only need to control the difference between $m_{\Wb_x^c}(z)$ and $m_{\check{\Wb}_x}(z)$. By a discussion similar to (\ref{eq_finalcontrolbound}), we have 
\begin{equation}
|m_{\check{\Wb}_{x}}(z)-m_{\Wb^c_x}(z)|^p \leq \frac{1}{n^{p/2} \eta^{2p}}  ( \text{tr} \{(\check{\Wb}_{x}-\mathbf{W}^c_x)^2\} )^{p/2}\,,
\end{equation}
and by a similar calculation, we have
\begin{align}
&\mathbb{E} \Big[ 1_{\Omega^c} |m_{\check{\Wb}_{x}}(z)-m_{\Wb^c_x}(z)|^p \Big] \nonumber\\
\leq&\, \mathbb{E} \Big[ 1_{\Omega^c} \frac{1}{n^{p/2} \eta^{2p}} \left( \sum_{i,j} (\check{\Wb}_{x}(i,j)-\Wb^c_x(i,j))^2 \right)^{p/2} \Big] \nonumber \\
 \leq&\, \mathbb{E} \Big[1_{\Omega^c} \frac{4 {C_5}}{n^{p/2} \eta^{2p}}  \Big( \sum_{i \neq j} (\Ob(i,j)\Pb(i,j))^2 \Big)^{p/2} \Big] \nonumber\\
 \leq&\, \frac{4{C_6} (\log p_1)^{p} }{n^{p/2}  \eta^{2p}}\,,\nonumber
\end{align}
for some constants {$C_5,C_6>0$}, where in the last inequality we use the fact that $\Ob(i,j)=O_{\prec}(n^{-1/2})$ and $\Pb(i,j)=O_{\prec}(n^{-1/2})$ by Lemmata \ref{lem_bernstein} and \ref{lem_chernoff}.
Together with a discussion similar to (\ref{eq_headstart}) and (\ref{eq_inequality1}), we conclude that
\begin{equation}\label{eq_transfereq}
{\mathbb E}\,|m_{\Wb^c_x}-m_{\check{\Wb}_x}|^p \leq {C_7} \frac{(\log p_1)^p}{n^{p/2} \eta^{2p}}\,,
\end{equation}
{for some constants $C_7>0$}.
When combined with the assumption (\ref{eq_boundc1c2}), (\ref{eq_rrrr})  and (\ref{eq_transfereq}), we have proved the first claim in Lemma \ref{lem_localkernel}.

The second claim comes from the fact that $\text{rank}(\mathbf{K}_x-\widetilde{\Kb}_x)\leq 1$. By (ii) of Lemma \ref{lem_collecttion}, we conclude the second claim.

The third claim comes from a similar strategy as that for the first claim. We first rewrite $\tilde{\Kb}_x$ as 
\begin{equation*}
\tilde{\Kb}_x=-2f'(2) \Xb^\top \Xb+(f(0)-f(2)) \Ib+2f'(2) \Mb\,,
\end{equation*}
where $\Mb$ is a diagonal matrix with
\begin{equation*}
\Mb=\text{diag} \{\norm{\xb_1}_2^2, \ldots, \norm{\xb_n}_2^2\}\,.
\end{equation*}
By (\ref{eq_stilebound}) and a discussion similar to (\ref{eq_inequality2}), since $\breve{\Kb}_x-\tilde{\Kb}_x=2f'(2)(\mathbf{M}-\mathbf{I})$, we readily obtain that 
\begin{align*}
\mathbb{E} & \left[ 1_{\Omega^c} |m_{\tilde{\Kb}_x}(z)-m_{\breve{\Kb}_x}(z)  |^p \right] \nonumber \\
& \leq \mathbb{E} \Big[ 1_{\Omega^c} \frac{1}{n \eta^4} \sum_{i,j} (\tilde{\Kb}_{x}(i,j)-\breve{\Kb}_x(i,j))^2 \Big]^{p/2} \nonumber  \\
 & \leq \mathbb{E} \Big[1_{\Omega^c} \frac{C_1}{n \eta^4}  \Big( \sum_{i} | \norm{\xb_i}_2^2-1|^2 \Big) \Big]^{p/2}  \\
& \leq \Big( \frac{{C_2} n \log p_1}{n^2 \eta^4} \Big)^{p/2}\,,\nonumber
\end{align*}
{for some constants $C_1,C_2>0$}, where in the second inequality we use (\ref{eq_diagcontrol}). This concludes our proof by the assumption (\ref{eq_boundc1c2}).

\subsection{Proof of Lemma \ref{lem_diagonal}}\label{section: proof of lemma lem_diagonal}

We only prove the results for $\Db_x$.
We adapt the same notations used in the proof of Theorem \ref{thm_edgevalue}, particularly the $\Omega$ defined in \eqref{eq_largderivationbound}.
Note that 
\begin{align}
&\norm{n(\Db_x)^{-1}-f^{-1}(2)\mathbf{I}_n)}\nonumber\\
=\,&\max_i \Big| \Big(n^{-1}\sum_{j=1}^n \Wb_x(i,j)\Big)^{-1}-f^{-1}(2) \Big|\nonumber\\
=\,&\max_i \Big| \frac{n^{-1}\sum_{j=1}^n \Wb_x(i,j)-f(2)}{f(2)n^{-1}\sum_{j=1}^n \Wb_x(i,j)} \Big| \,.\nonumber
\end{align}
By the expansion (\ref{eq_decompositionw}), on $\Omega^c$, we have that for any $i$, 
\begin{align}
\Big| \frac{1}{n} \sum_{j=1}^n \Wb_x(i,j)-f(2) \Big|
&\leq \, \frac{1}{n} \sum_{j \neq  i}^n \Big| f'(2)(\mathbf{O}(i,j)-2\mathbf{P}(i,j))+f^{''}(\zeta_{ij})(\mathbf{O}(i,j)-2\mathbf{P}(i,j))^2 \Big|\label{eq_boundmatrixdiag}\\
&\,+\frac{|\Wb_x(i,i)|+|f(2)|}{n}\,.\nonumber 
\end{align}
First, by (\ref{eq_largderivationbound}) and the assumption (\ref{eq_boundc1c2}), on $\Omega^c$, we readily obtain that for some constant $C>0$,
\begin{align}
& \frac{1}{n} \sum_{j \neq i}^n |\mathbf{P}(i,j)| \leq  \sqrt{\frac{C \log n}{n}}, \nonumber\\
& \frac{1}{n} \sum_{j \neq i}^n|\mathbf{O}(i,j)| \leq \sqrt{\frac{C \log n}{n}}\,,\label{Proof SI13 Rij difference equation} \\ 
& \frac{1}{n} \sum_{j \neq i} (\mathbf{O}(i,j)-2\mathbf{P}(i,j))^2 \leq C \frac{\log n}{n}.\nonumber
\end{align}
Moreover, since $\Wb_x(i,i)=f(0)$ and $f(2)$ is bounded, we conclude that on $\Omega^c$,
\begin{equation*}
\Big| \frac{1}{n} \sum_{j=1}^n \Wb_x(i,j)-f(2) \Big| \leq \sqrt{\frac{C_1 \log n}{n}} \,,
\end{equation*}
for some constant $C_1>0$. As $f(2)>0$, when $n$ is large enough, we find that on $\Omega^c$,
\begin{equation*}
\Big| \frac{1}{n} \sum_{j=1}^n \Wb_x(i,j) \Big|>\frac{f(2)}{2}\,.
\end{equation*}
Then the proof follows from (\ref{eq_boundmatrixdiag}).

\subsection{Proof of Lemma  \ref{lem_2.4.2}}\label{Section proof of lem_2.4.2}

We adapt the same notations used in the proof of Theorem \ref{thm_edgevalue}, particularly the $\Omega$ defined in \eqref{eq_largderivationbound}.
We will  use a discussion similar to the equations heading from (\ref{eq_headstart}) to (\ref{eq_inequality2}) to conclude our proof. 
Similar to the decomposition in (\ref{eq_headstart}), we have  
\begin{align*}
\mathbb{E} \Big| m_{\mathbf{S}^{(2)}_{xy}}(z)-m_{\Sb^{(3)}_{xy}}(z) \Big|^p & =  \mathbb{E} \Big[ 1_{\Omega} |m_{\mathbf{S}^{(2)}_{xy}}(z)-m_{\Sb^{(3)}_{xy}}(z) |^p \Big] \\
& +\mathbb{E} \Big[ 1_{\Omega^c} |m_{\mathbf{S}^{(2)}_{xy}}(z)-m_{\Sb^{(3)}_{xy}}(z) |^p \Big].
\end{align*}
By the assumption (\ref{eq_boundc1c2}),  we have that $\mathbb{P}(\Omega)=p_1^{-C_1/2}$.
Using (\ref{eq_stitrivial}) and the Cauchy-Schwarz inequality, together with a discussion similar to (\ref{eq_inequality1}), we obtain that 
\begin{equation*}
\mathbb{E}\left[ 1_{\Omega}\big| m_{\mathbf{S}^{(2)}_{xy}}(z)-m_{\Sb^{(3)}_{xy}}(z)  \big|^p \right] \leq \frac{2^{p}}{p_1^{C/2} \eta^{p}}\,. 
\end{equation*}
On the other hand, by the same bound like that of \eqref{eq_finalcontrolbound}, for some constant $C>0$, we have
\begin{align*}
\mathbb{E} & \Big[ 1_{\Omega^c}|m_{\mathbf{S}^{(2)}_{xy}}(z)-m_{\Sb^{(3)}_{xy}}(z) |^p \Big]  \nonumber \\
& \leq C \mathbb{E} \Big[ 1_{\Omega^c} \frac{1}{n \eta^4} \operatorname{tr}\big\{(\mathbf{S}_{xy}^{(2)}-\mathbf{S}_{xy}^{(3)})^2\big\} \Big]^{p/2} \nonumber\,.
\end{align*}
To control $\operatorname{tr}\big\{(\mathbf{S}_{xy}^{(2)}-\mathbf{S}_{xy}^{(3)})^2\big\}$, we denote $\mathcal{E}_x:=\widetilde{\mathbf{W}}_x-\breve{\mathbf{K}}_x$ and $\mathcal{E}_y:= \widetilde{\Wb}_y-\breve{\Kb}_y$ for brevity. Since
\begin{align}\label{eq_keydecom}
& \mathbf{S}^{(2)}_{xy}- \mathbf{S}^{(3)}_{xy} = \frac{1}{f^2(2)} ( \mathcal{E}_x \widetilde{\mathbf{W}}_y+ \breve{\mathbf{K}}_x \mathcal{E}_y )\,,
\end{align}
we have
\begin{align} 
&\big( \mathbf{S}^{(2)}_{xy}- \mathbf{S}^{(3)}_{xy} \big)^2 \,\frac{1}{f^4(2)} \Big( \mathcal{E}_x  \widetilde{\Wb}_y  \mathcal{E}_x  \widetilde{\Wb}_y+\mathcal{E}_x  \widetilde{\Wb}_y \breve{\Kb}_x \mathcal{E}_y +\breve{\Kb}_x \mathcal{E}_y \mathcal{E}_x  \widetilde{\Wb}_y+ \breve{\Kb}_x \mathcal{E}_y \breve{\Kb}_x \mathcal{E}_y \Big)\,.\label{eq_expansion}
\end{align}
Due to the similarity of the items in (\ref{eq_expansion}), we only discuss how to control the trace of the first and second terms on its right-hand side.  By (iv) of Lemma \ref{lem_collecttion}, we have 
\begin{align}
\text{tr} \mathcal{E}_x  \widetilde{\Wb}_y  \mathcal{E}_x  \widetilde{\Wb}_y &\,\leq \lambda_1(\widetilde{\Wb}_y) \text{tr}  ( \mathcal{E}_x  \widetilde{\Wb}_y  \mathcal{E}_x )\label{tracekeybound} \\
&\,=\lambda_1(\widetilde{\Wb}_y) \text{tr} (\mathcal{E}_x^2 \widetilde{\Wb}_y)  \leq (\lambda_1(\widetilde{\Wb}_y))^2 \text{tr} (\mathcal{E}_x^2),\nonumber
\end{align}
and 
\begin{align}
\text{tr} \mathcal{E}_x  \widetilde{\Wb}_y \breve{\Kb}_x \mathcal{E}_y & \leq \lambda_1(\widetilde{\Wb}_y) \text{tr}  ( \breve{\Kb}_x \mathcal{E}_y  \mathcal{E}_x )\label{tracekeybound1} \\
& \leq \lambda_1(\widetilde{\Wb}_y) \lambda_1(\breve{\Kb}_x)  \text{tr} (\mathcal{E}_y \mathcal{E}_x)  \nonumber \\
& \leq \lambda_1(\widetilde{\Wb}_y) \lambda_1(\breve{\Kb}_x)  (\text{tr} (\mathcal{E}_x^2)+\text{tr} (\mathcal{E}_y^2) )\,,\nonumber
\end{align}
where in the last step we use the Cauchy-Schwarz inequality (Recall that for any two matrices $\Ab, \Bb \in \mathbb{R}^{n \times n}$, the standard inner product is defined as $\text{tr}(\Ab^\top \Bb)$.)
With (\ref{tracekeybound}) and (\ref{tracekeybound1}), we know that for some constant $C>0$,
\begin{align}
\mathbb{E} \Big[ 1_{\Omega^c}|  m_{\mathbf{S}^{(2)}_{xy}}(z)-m_{\Sb^{(3)}_{xy}}(z) |^p \Big] 
\leq\, C \mathbb{E}\Big[&1_{\Omega^c} \frac{1}{n\eta^4} \Big( \text{tr}(\mathcal{E}_x^2) (\lambda_1(\widetilde{\mathbf{W}}_y))^2\nonumber+\text{tr}(\mathcal{E}_y^2) (\lambda_1(\widetilde{\mathbf{W}}_x))^2\nonumber\\
&\qquad+2\lambda_1(\widetilde{\Wb}_y) \lambda_1(\breve{\Kb}_x)  (\text{tr} (\mathcal{E}_x^2)+\text{tr} (\mathcal{E}_y^2 ) \Big) )  \Big]^{p/2}\label{eq_decompose111}\,.
\end{align}
To finish the proof, we bound $\lambda_1(\breve{\Kb}_x)$, $\lambda_1(\widetilde{\Wb}_x)$, $\lambda_1(\widetilde{\Wb}_y)$, $\text{tr} (\mathcal{E}_x^2)$ and $\text{tr} (\mathcal{E}_y^2)$. Recall the definition of $\widetilde{\Wb}_x$ in (\ref{eq_defntildem}). Define
\begin{equation*}
\widetilde{\Wb}_x^{(1)}:=\widetilde{\Wb}_x+\mathrm{Sh}_2\,.
\end{equation*}
By (\ref{Bound of Sh1+Sh1}), we find that 
\begin{equation*}
\| \widetilde{\Wb}_x^{(1)} -\widetilde{\Wb}_x \|=O_{\prec}\big( n^{-1/2} \big).
\end{equation*}
Moreover, by the Taylor expansion (\ref{eq_decompositionw}) and the fact $\widetilde{\Wb}_x^{(1)}=\Wb_x-\mathrm{Sh}_0-\mathrm{Sh}_1$, we find that
\begin{align*}
&\widetilde{\Wb}_x^{(1)}(i,j)\\
=&\,\left\{
\begin{array}{ll}
-2f'(2)\mathbf{P}(i,j)+f^{''}(\zeta_{ij})(\mathbf{O}(i,j)-2\mathbf{P}(i,j))^2 & i\neq j\\
f(0)-f(2)-2f'(2)[\|\xb_i\|^2-1] & i=j\,.
\end{array}
\right.
\end{align*}
Let $\mathbf{R}_x$ be the matrix with entries 
\begin{equation*}
\mathbf{R}_x(i,j)=\left\{
\begin{array}{ll}
f^{''}(\zeta_{ij})(\mathbf{O}_{ij}-2\mathbf{P}_{ij})^2 & i\neq j\\
0&i=j\,,
\end{array}
\right.
\end{equation*} 
and $\mathbf{S}_x$ be a diagonal matrix so that 
\begin{equation*}
\mathbf{S}_x(i,i)=-2f'(2) \mathbf{I}\,,
\end{equation*} 
for $i=1,2,\cdots, n$. Then we have 
\begin{equation}\label{eq_2w}
 \widetilde{\Wb}_x^{(1)}=-2f'(2)\Xb^\top \Xb-\Sb_x+\mathbf{R}_x+(f(0)-f(2)) \mathbf{I}\,.
\end{equation}
First, when conditional on $\Omega$, clearly we have that $\| \Sb_x \|=O(1)$. Second, using Lemma \ref{lem_rigidity} with $\Sigma=\mathbf{I}$, we find $\| -2f'(2) \Xb^\top \Xb \|=O_{\prec}(1)$. Next, we control the norm of $\mathbf{R}_x$. 
Note that when $i\neq j$, 
\begin{equation*}
|\mathbf{R}_x(i,j)|=O_{\prec}(n^{-1}), 
\end{equation*}
since $\Rb_x(i,j)=f^{''}(\zeta_{ij}) \Ob_{ij}^2-2\Ob_{ij} \Pb_{ij}+\Pb_{ij}^2$, $\Ob_{ij}=O_{\prec}(n^{-1/2})$ and $\Pb_{ij}=O_{\prec}(n^{-1/2})$.
Specifically, note that $\log(n)$ is controlled by $n^\epsilon$ for any $\epsilon>0$, so we have this stochastic dominant control. By the Gershgorin circle theorem, we have 
\begin{equation*}
\| \mathbf{R}_x  \| \leq \max_i \Big(\sum_{j \neq i} |\mathbf{R}_x(i,j)| \Big)\,,
\end{equation*}
and hence 
\begin{equation*}
\| \mathbf{R}_x \|=O_{\prec}(1)\,. 
\end{equation*}
By (\ref{eq_2w}), we readily obtain that 
\begin{equation}\label{eq_matrixbound}
\lambda_1(\widetilde{\mathbf{W}}_x) =O_{\prec}(1)\,.  
\end{equation} 
By a similar argument, we have $\lambda_1(\widetilde{\Wb}_{y}) =O_{\prec}(1)$.
Moreover, by a discussion similar to (\ref{eq_inequality2}),  for some constant $C_1>0$,
\begin{align}
&\mathbb{E} \left[ 1_{\Omega^c } \operatorname{tr} \left(\mathcal{E}_x^2 \right) \right] \leq C_1 \log p_1\,, \label{eq_weilsimilar}\\
&\mathbb{E} \left[ 1_{\Omega^c } \operatorname{tr} \left(\mathcal{E}_y^2 \right) \right] \leq C_1 \log p_2\,.\nonumber
\end{align}
By combining (\ref{eq_decompose111}), (\ref{eq_matrixbound}) and (\ref{eq_weilsimilar}), we therefore conclude the proof using the assumption (\ref{eq_boundc1c2}).

\subsection{Proof of Proposition \ref{prop_norm}}\label{section proof of propositions prop_norm}

Recall the definition of $\widetilde{\mathbf{W}}_x$ in (\ref{eq_defntildem}). By Lemma \ref{lem_matrixapproximation}, for some small constants $c>0$ and $\varsigma>0$, with probability at least $1-O(n^{-1/2-c})$, we have
\begin{equation} \label{eq_converge}
\norm{\widetilde{\mathbf{W}}_x-\breve{\mathbf{K}}_x} \leq  n^{-\varsigma}/2\,, 
\end{equation}
where $\breve{\mathbf{K}}_x$ is defined in (\ref{eq_matkotherform}).  By Remark \ref{rem:localawlarge} and Lemma \ref{lem_rigidity},  we have with probability greater than $1-n^{-D}$ for some large $D>0$ that 
\begin{equation*}
\lambda_1(\breve{\mathbf{K}}_x) \leq \gamma_+ + n^{-\varsigma}/2\,.
\end{equation*}
As a consequence, together with (\ref{eq_converge}), we conclude that with probability at least $1-O(n^{-1/2-c})$,
\begin{equation}\label{Bound of tildeWx norm with high probability}
\lambda_1(\tilde{\mathbf{W}}_x) \leq \gamma_++n^{-\varsigma}\,.
\end{equation}
Hence, for {any given} small $\epsilon>0,$ 
it suffices to prove that with probability greater than $1-O(n^{-1/2-c})$, there is no eigenvalue of $\widetilde{\mathbf{W}}_x$ in 
\begin{equation}
\mathrm{I}:=[\gamma_++n^{-1/9+{C_0}\epsilon}, \gamma_++n^{-\varsigma}]\,,\label{Definition:Proof Proposition5.3 I}
\end{equation}
where {$C_0>0$} is a constant.

To achieve this goal, we prepare some quantities. Since $\text{rank}(\widetilde{\mathbf{W}}_x-\mathbf{W}_x) = \text{rank}(\mathrm{Sh}_0+\mathrm{Sh}_1+\mathrm{Sh}_2)\leq 6$, we get from (ii) of Lemma \ref{lem_collecttion} that 
\begin{equation}\label{eq_control}
|m_{\widetilde{\mathbf{W}}_x}(z)-m_{\mathbf{W}_x}(z)| \leq \frac{C}{n \eta}\,,
\end{equation}
for some constant $C>0$. Furthermore, by Lemma \ref{lem_localkernel} and the Markov inequality, we conclude that 
{we can find a small} $\epsilon_0>0$ {satisfying $\epsilon_0<1/18$} and a large $p_0\geq 1$ so that with probability greater than $1-n^{-\epsilon_0 p_0}$,
\begin{equation}\label{proof prop 5.3 Wx cKx Stieljes difference}
|m_{\widetilde{\mathbf{W}}_x}(z)-m_{\breve{\Kb}_x}(z)| \leq \frac{1}{n^{1/2-\epsilon_0} \eta^{2}} \,,
\end{equation}
for any $z\in \mathsf{S}(1/4,\tau)$. Note that {we can choose} $p_0$ sufficiently large so that $D:=\epsilon_0 p_0>0$ is sufficiently large. Denote $\Omega$ to be the associated event.
Below, we focus our discussion on $\Omega$. Since $\mathsf{S}(1/4,\tau)\subset \mathsf{S}(1,\tau)$, by combining Lemma \ref{lem_locallaw} and \eqref{proof prop 5.3 Wx cKx Stieljes difference}, we have 
\begin{equation}\label{eq1}
|m_{\widetilde{\mathbf{W}}_x}(z)-m_{\nu_x}(z)| \leq \frac{1}{n^{1/2-\epsilon_0} \eta^2},
\end{equation} 
when conditional on $\Omega$, where recall that $\nu_x$ is the LSD of $\breve{\mathbf{K}}_x$ defined in \eqref{Definition special measure for kernel matrix}. 
We now choose $z=E+\mathrm{i}\eta \in \mathbb{C}^+$ so that $E\in\mathrm{I}$ defined in \eqref{Definition:Proof Proposition5.3 I} and $\eta$ satisfies 
\begin{equation}\label{eta_use}
\eta=\eta(E):=n^{-1/9-C_1 \epsilon}\sqrt{\kappa(E)}\,,
\end{equation}
where $\kappa(E):=|E-\gamma_+|$ and $C_1>0$ is chosen so that $z \in \mathsf{S}(1/4,\tau)$ {and $C_1\leq (\frac{1}{54}-\frac{\epsilon_0}{3})\frac{1}{\epsilon}+\frac{C_0}{3}$. Note that $C_1$ exists since $\epsilon_0$ satisfies $\epsilon_0<1/18$}.

By Lemma \ref{lem_mplaw}, for any such $z$, for some constant $C_2>0$, we have
\begin{align}
|\operatorname{Im} m_{\nu_x}(z)| \leq C_2 \frac{\eta}{\sqrt{\kappa (E)+\eta}}
 \leq\, C_2 \frac{\eta}{\sqrt{\kappa (E)}}=C_2 n^{-1/9-C_1 \epsilon}\,.\label{eq2}
\end{align}
Hence, by (\ref{eq1}) and (\ref{eq2}), when $n$ is sufficiently large, we have
\begin{align}
|\operatorname{Im} m_{\widetilde{\mathbf{W}}_x}(z)| 
\leq & \,|m_{\widetilde{\mathbf{W}}_x}(z)-m_{\nu_x}(z)|+|\operatorname{Im} m_{\nu_x}(z)| \label{eq_contro}\\
\leq &\, n^{-1/2+\epsilon_0} \eta^{-2}+C_2 n^{-1/9-C_1 \epsilon}\nonumber\\
\leq &\,2C_2 n^{-1/9-C_1\epsilon}\nonumber\,,
\end{align}
where the last inequality holds since $n^{-1/2+\epsilon_0} \eta^{-2}\leq n^{-1/6+(2C_1-{C_0})\epsilon+\epsilon_0}$, which holds due to \eqref{eta_use} and the fact that $n^{-1/9+{C_0}\epsilon} \leq \kappa(E) \leq n^{-\varsigma}$ by the assumption that $E\in\mathrm{I}$, and {the choice of $C_1$ that satisfies $C_1\leq (\frac{1}{54}-\frac{\epsilon_0}{3})\frac{1}{\epsilon}+\frac{C_0}{3}$}.

With the above preparation, now we assume that there exists an eigenvalue in $\mathrm{I}$ when conditional on $\Omega$. Without loss of generality, suppose $\lambda_1(\widetilde{\mathbf{W}}_x) \in \mathrm{I}$. Set $z'=\lambda_1(\widetilde{\mathbf{W}}_x)+\mathrm{i} \eta$, where $\eta$ is chosen in  (\ref{eta_use}). We have
\begin{align}\label{eq_imktildedecom}
\operatorname{Im} m_{\breve{\mathbf{K}}_x}(z')&\,=\frac{1}{n} \sum_{j} \frac{\eta}{(\lambda_j(\breve{\mathbf{K}}_x)-\lambda_1(\widetilde{\mathbf{W}}_x))^2+\eta^2} \nonumber \\
&\,=\frac{1}{n} \sum_{j} \frac{\eta}{(\lambda_j(\breve{\mathbf{K}}_x)-\gamma_+-\kappa_1)^2+\eta^2}\,,  
\end{align}
where $\kappa_1:=\kappa(\lambda_1(\widetilde{\mathbf{W}}_x))$.
Note that since $n^{-1/9+{C_0}\epsilon}\leq \kappa_1\leq n^{-\varsigma}$, $n^{-\varsigma}<1$, and $\eta=n^{-1/9-C_1\epsilon}\sqrt{\kappa_1}$, we know 
\begin{equation}\label{eq_flu0000}
\kappa_1 \geq \eta,
\end{equation} 
on $\mathrm{I}$. Since $n^{-\varsigma}<1$, we have $ \kappa_1^{3/2}< 1$.
By the same argument for \eqref{proof theorem lambdai right most edge gap} based on the square root behavior of $\nu_x$ (c.f.. (\ref{eq)squareroot})), Lemma \ref{lem_rigidity} and \eqref{proof theorem lambdai gammai the same}, we deduce that for all $j=1,\ldots,n$, 
\begin{equation}\label{eq_flu}
|\gamma_+-\lambda_j(\breve{\mathbf{K}}_x)| \sim j^{2/3} n^{-2/3}\,,
\end{equation}
which leads to
\begin{equation*}
|\gamma_+-\lambda_j(\breve{\mathbf{K}}_x)| \leq \sqrt{C_3} \sqrt{\kappa_1}\,,
\end{equation*}
when $j < \kappa_1^{3/4}n$, where $C_3>0$ is  some fixed constant. On the other hand, when $j \geq  \kappa_1^{3/4} n$, it is easy to see from (\ref{eq_flu}) that when $n$ is sufficiently large 
\begin{equation}\label{eq_gegekap1}
|\gamma_+-\lambda_j(\breve{\Kb}_x)|^2 >  C_3\kappa_1. 
\end{equation}
Note that when $n$ is sufficiently large, we have $(n^{-\varsigma})^{3/4}<1$. Hence, when $n$ is sufficiently large, by (\ref{eq_imktildedecom}), (\ref{eq_flu}) and \eqref{eq_gegekap1}, we have 
\begin{align*}
\left|\operatorname{Im} m_{\breve{\mathbf{K}}_x}(z')\right| \geq &\,\frac{1}{n} \sum_{j= \kappa_1^{3/4} n}^{ n} \frac{\eta}{2(\lambda_j(\breve{\mathbf{K}}_x)-\gamma_+)^2+2\kappa_1^2+\eta^2}\nonumber\\
\geq &\,  \frac{\eta}{(3+2C_3) n} n(1-\kappa_1^{3/4}) \frac{1}{\kappa_1}   \\
= &\, \frac{1}{(3+2C_3)} n^{-1/9-C_1 \epsilon} (1-\kappa_1^{3/4})  \frac{1}{\sqrt{\kappa_1}}  \,.\nonumber
\end{align*}
Since $\kappa_1 \leq n^{-\varsigma}$, when $n$ is large, together with (\ref{proof prop 5.3 Wx cKx Stieljes difference}), we find that this is a contradiction to (\ref{eq_contro}) and hence there is no eigenvalue lying in $\mathrm{I}$. {Indeed, by the choice of $z'$, we have $|\operatorname{Im} m_{\widetilde{\mathbf{W}}_x}(z')-\operatorname{Im}m_{\breve{\Kb}_x}(z')| \leq n^{-1/6+(2C_1-C_0)\epsilon}$. Also, by the choice of $C_1$, we know $n^{-1/6+(2C_1-C_0)\epsilon}\leq n^{-1/9-C_1\epsilon-\epsilon_0}$. As a result, by the triangular inequality, we should have $|\operatorname{Im} m_{\widetilde{\mathbf{W}}_x}(z')|\geq|\operatorname{Im}m_{\breve{\Kb}_x}(z')|-|\operatorname{Im} m_{\widetilde{\mathbf{W}}_x}(z')-\operatorname{Im}m_{\breve{\Kb}_x}(z')|\geq \frac{1}{2(3+2C_3)}n^{-1/9-C_1\epsilon+\varsigma/2}-n^{-1/9-C_1\epsilon-\epsilon_0}$, which asymptotically is greater than $2C_2n^{-1/9-C_1\epsilon}$, which contradicts (\ref{eq_contro}).} This concludes our proof.

\section{Technical proofs of Propositions \ref{prop_largscaleset} and \ref{pro_ncca_main}}\label{section proof of propositions prop_largscaleset}

We follow the proof backbone of \cite{ERDOS2012,locallawnote} but provide details for our setup for the self-containedness.
  
\begin{proof}[Proof of Proposition \ref{prop_largscaleset}]  

Denote 
\begin{equation}
\hat{\nu}:=\mu_{\widetilde{\Wb}_x}-\nu_x, \ 
\hat{m}:=m_{\widetilde{\Wb}_x}-m_{\nu_x}\,.
\end{equation}
Let $\gamma_-$ and $\gamma_+$ be the left-most and right-most edges of $\nu_x$ respectively.  Clearly, $[\gamma_-,\gamma_+]$ is included in the range of the support of $\widetilde{\Wb}_x$. Therefore, we have $\supp(\mu_{\tilde{\Wb}_x}-\nu_x)\subset \supp \mu_{\tilde{\Wb}_x}$. By the control of the largest eigenvalue of $\tilde{\Wb}_x$ in \eqref{Bound of tildeWx norm with high probability} and the smallest eigenvalue of $\tilde{\Wb}_x$ by a similar argument, we know that for a fixed $\theta>0$, when $n$ is sufficiently large, $\supp \mu_{\tilde{\Wb}_x}\subset [\gamma_--\theta,\gamma_++\theta]$ holds with probability at least $1-O(n^{-1/2})$.
Thus, from now on, we focus our discussion on this probability event and call it $\Omega_1$. Hence, the support of $\hat{\nu}$ is in $[\gamma_--\theta,\gamma_++\theta]$.

To the end of the proof, fix a small $\epsilon>0$ and define $q=q(n):=n^{-1/4+\epsilon}$. Then, for any interval $I \subset [\gamma_--\theta, \gamma_++\theta]$, we choose a smoothed indicator function $h \equiv h_{I, \eta} \in \mathcal{C}_c^{\infty}(\mathbb{R}; [0,1])$ satisfying $h(x)=1$ for $x \in I$, and $h(x)=0$ for $\text{dist}(x, I) \geq q$, $\norm{h'}_{\infty} \leq B q^{-1}$ and $\norm{h^{''}}_\infty \leq B q^{-2}$ for $B>0$.  Next, we choose a smooth and even cutoff function $\chi \in \mathcal{C}_c^{\infty}(\mathbb{R};[0,1])$ satisfying $\chi (y)=1$ for $|y| \leq 1,\ \chi (y)=0$ for $|y| \geq 2$, and $\norm{\chi'}_{\infty} \leq B. $ 
Using the Helffer-Sj{\"o}strand formula from Lemma \ref{lem_funcal} and \eqref{rem:d4} with $n=1$, we obtain
\begin{equation}
h(\lambda)=\frac{1}{2 \pi} \int_{\mathbb{R}^2} \frac{\mathrm{i} y h^{''}(x) \chi (y)+\mathrm{i}(h(x)+\mathrm{i} yh'(x)) \chi' (y)}{\lambda-x-\mathrm{i} y} dx dy
\end{equation}
for any $\lambda\in \mathbb{R}$ so that $\chi(\lambda)=1$. 
By a direct calculation, we have 
\begin{align*}
\int h(\lambda) \hat{\nu}(d \lambda)=\frac{1}{2 \pi}  \int_{\mathbb{R}^2}\big[\mathrm{i} y h^{''}(x) \chi (y)+\mathrm{i}(h(x)+\mathrm{i} yh'(x)) \chi' (y)\big] \hat{m}(x+\mathrm{i} y)  dx dy\,.
\end{align*}
Note that $\mathrm{i} y h^{''}(x) \chi (y)+\mathrm{i}(h(x)+\mathrm{i} yh'(x)) \chi' (y)=0$ when $y=0$. Since $\int h(\lambda) \hat{\nu}(d \lambda)$ on the left hand side is real, by a direct expansion, we obtain that (for instance see \cite[equations (8.2)-(8.4)]{locallawnote})
\begin{align}
&\int h(\lambda) \hat{\nu}(d\lambda)\nonumber  \\
=&-\frac{1}{2 \pi}\iint_{|y| \leq q}  h^{''}(x) \chi (y) y \operatorname{Im} \hat{m}(x+\mathrm{i} y) dx dy \label{eq_bb1} \\ 
&-\frac{1}{2 \pi} \iint_{|y|>q} h''(x) \chi(y) y \operatorname{Im} \hat{m}(x+\mathrm{i}y) dx dy \label{eq_bb2} \\
& +\frac{\mathrm{i}}{2 \pi} \iint (h(x)+\mathrm{i} yh'(x)) \chi'(y) \hat{m} (x+\mathrm{i}y) dx dy. \label{eq_bb3}
\end{align}
We now control (\ref{eq_bb1})-(\ref{eq_bb3}). By \eqref{eq1}, we find that there exists a high probability event $\Omega_2$ such that for small $\epsilon_0>0$,
\begin{equation}
|\hat{m}(x+\mathrm{i} y)| \leq \frac{1}{n^{1/2-\epsilon_0} y^2},\label{Proof proposition 5.4 bound hat m}
\end{equation}
for $\epsilon_0\leq x\leq\epsilon_0^{-1}$ and $y \in [q, \epsilon^{-1}]$. We restrict our discussion on the high probability event $\Omega_1\cap \Omega_2$. By definition, $\chi'$ is supported in $[-2,2]\backslash (-1,1)$. By the fact that the measure of the support of  $h'$ is at most $2 \eta$, we find that there exists some constant $C_1>0$ such that
\begin{align*}
|(\ref{eq_bb3})| 
 \leq & \, n^{-1/2+\epsilon_0} \int_{\mathbb{R}^2} (|h(x)| y^{-2} |\chi' (y)|+y^{-1} |h'(x) \chi'(y)|) dx dy  \\
 \leq\,& C_1 n^{-1/2+\epsilon_0}. 
\end{align*}
Next we control (\ref{eq_bb1}). First of all, since
\begin{equation*}
y \operatorname{Im}m_{\widetilde{\Wb}_x}(x+\mathrm{i} y)=\frac{1}{n}\sum_{i=1}^n \frac{y^2}{(\lambda_i(\widetilde{\Wb}_x)-x)^2+y^2}\,,  
\end{equation*}
it is easy to see that the map $y \rightarrow y \operatorname{Im} m_{\widetilde{\Wb}_x}(x+ \mathrm{i} y) $ is nondecreasing and nonnegative for all $x$ and $y>0$. Therefore, we have that for any $y \in (0, q]$,
\begin{align}
y \operatorname{Im} \hat{m}(x+\mathrm{i}y)  \leq y \operatorname{Im} m_{\widetilde{\Wb}_x}(x+\mathrm{i} y) \leq q \operatorname{Im} m_{\widetilde{\Wb}_x}(x+\mathrm{i} 
q) \leq C_1 q\,,\nonumber
\end{align} 
where in the first inequality, we use the fact that both $y \operatorname{Im}m_{\widetilde{\Wb}_x}(x+\mathrm{i}y) $ and $y \operatorname{Im}m_{\nu_x}(x+\mathrm{i}y)$ are non-negative, and in the third step, we use  Lemma \ref{lem_mplaw} and the fact that $\sqrt{\kappa+q}$ is of order $1$ since $\kappa$ is of order $1$ and $q=n^{-1/4+\epsilon}$. As a consequence, we obtain that for some constant $C_2>0$,
\begin{equation*}
|(\ref{eq_bb1})| \leq C_2 q\,,
\end{equation*}
where we use the fact that the support of $h''$ has Lebesgue measure of order $\eta$. 
Finally, we control (\ref{eq_bb2}). By a direct calculation, we have 
\begin{align}
(\ref{eq_bb2}) = &\,\mathrm{i}\int h'(x) \int_{|y|> q} \chi(y) y\partial_y \operatorname{Im} \hat{m}(x+\mathrm{i}y) dy dx\nonumber \\
=&\,\mathrm{i}\chi(q)q \int h'(x) \operatorname{Im} \hat{m}(x+\mathrm{i}q)dx\\
&\,- \mathrm{i}\iint_{|y|> q} h'(x) (y \chi'(y) +\chi(y)) \operatorname{Im}  \hat{m}(x+\mathrm{i}y)dx dy\,,\nonumber
\end{align}
where the first equality comes from the integration by parts in $x$, the fact that $\partial_x\operatorname{Im} m_{\nu_x}(x+\mathrm{i} q)=-i\partial_y\operatorname{Im} m_{\nu_x}(x+\mathrm{i} q)$, and the fact that $\hat{m}(x+iy)\to 0$ when $x\to \infty$ cancels the boundary terms, and the second equality comes from the integration by parts in $y$. Similar to the discussion of \cite[equation (8.7)]{locallawnote}, we thus have
\begin{align}
|(\ref{eq_bb2})| \leq & \iint_{y> q} |h'(x) \chi(y) \hat{m}(x+\mathrm{i}y)| dx dy+\iint_{y> q} |h'(x) y \chi'(y) \hat{m}(x+\mathrm{i}y)|dx dy  \nonumber \\
& + q\int  |h'(x)  \hat{m}(x+\mathrm{i} q)| dx\,.\label{Proof of Proposition 5.4 control the large eta part}
\end{align}
For the first part of \eqref{Proof of Proposition 5.4 control the large eta part}, for some constants $C_1, C>0$, by \eqref{Proof proposition 5.4 bound hat m}, it is bounded by
\begin{equation*}
C_1 \int_{q}^2 \frac{1}{n^{1/2-\epsilon_0} y^2} dy \leq C n^{-1/4}\,. 
\end{equation*}
Similarly, we can deal with the other items of \eqref{Proof of Proposition 5.4 control the large eta part} and conclude that 
\begin{equation*}
|(\ref{eq_bb2})| \leq C n^{-1/4}\,. 
\end{equation*}
We thus have proved that on the event $\Omega_1 \cap \Omega_2$,
\begin{equation}\label{Proof proposition bound 1 for final one side}
\left| \int h(\lambda) \widetilde{\nu}_x(d \lambda) \right| \leq C n^{-1/4+2\epsilon}\,. 
\end{equation}

With the above preparation, we now conclude our proof. For $I \subset [\gamma_-, \gamma_+]$ and $h=h_{I,q}$, clearly we have
\begin{equation}\label{Proof proposition bound 2 for final one side}
\nu_x(I) \leq \int h(\lambda) \nu_x ( d \lambda).
\end{equation}
Moreover, with high probability, we have  
\begin{equation}\label{Proof proposition bound 3 for final one side}
\int h(\lambda) \mu_{\widetilde{\Wb}_x} (d \lambda) \leq \mu_{\widetilde{\Wb}_x} (I)+O(n^{-1/4+2\epsilon}),
\end{equation}
where we use the fact that $\mu_{\widetilde{\Wb}_x}$ is bounded conditional on $\Omega_1 \cap \Omega_2$ and $q=n^{-1/4+\epsilon}$. As a result, by combining \eqref{Proof proposition bound 1 for final one side}, \eqref{Proof proposition bound 2 for final one side} and \eqref{Proof proposition bound 3 for final one side}, we obtain
\begin{equation*}
\nu_x(I) \leq\mu_{\widetilde{\Wb}_x} (I)+O(n^{-1/4+2\epsilon}).
\end{equation*}
On the other hand, by denoting $I':=\{x \in \mathbb{R}: \text{dist} (x, I^c) \geq q\}$, by a similar argument we have 
\begin{equation*}
\nu_x(I)  \geq \int h_{I', q}(\lambda)\nu_x(d \lambda) \,,
\end{equation*}
and
\begin{equation*}
\int h_{I', q}(\lambda) \mu_{\widetilde{\Wb}_x}(d \lambda) \geq \mu_{\widetilde{\Wb}_x}(I)+O(n^{-1/4+2\epsilon})\,.
\end{equation*}
Combined with \eqref{Proof proposition bound 1 for final one side}, we have
\begin{equation*}
\nu_x(I)  \geq\mu_{\widetilde{\Wb}_x}(I)+O(n^{-1/4+2\epsilon})\,,
\end{equation*}
and hence we have with probability at least $1-O(n^{-1/2})$, for some constant $C>0$,
\begin{equation}\label{Proof Proposition 5.4 when I is inside the bulk}
|\mu_{\widetilde{\Wb}_x}(I)-\nu_x(I) | \leq C n^{-1/4+2 \epsilon}\,,
\end{equation}
for any $I\subset [\gamma_--\theta,\gamma_++\theta]$.

To finish the proof, we extend the results to any $I \subset \mathbb{R}. $ Since $\nu_x([\gamma_-, \gamma_+])=1$, by (\ref{Proof Proposition 5.4 when I is inside the bulk}) with $I=[\gamma_-, \gamma_+]$, we find that with probability at least $1-O(n^{-1/2}),$ for some constant $C>0,$
\begin{equation}\label{Proof Proposition 5.4 eq_small}
\mu_{\widetilde{\Wb}_x} (\mathbb{R} \backslash [\gamma_-, \gamma_+])  \leq C n^{-1/4+2 \epsilon}\,.
\end{equation}

For arbitrary $I \subset \mathbb{R}$, we have 
\begin{align}
\nu_x(I)=&\,\nu_x (I \cap [\gamma_-, \gamma_+])=\mu_{\widetilde{\Wb}_x} (I \cap [\gamma_-, \gamma_+])+O_{\prec}(n^{-1/4})\nonumber\\
=&\,\mu_{\widetilde{\Wb}_x} (I)+O(n^{-1/4+2 \epsilon})\,,\nonumber
\end{align}
where the first quality holds since $\nu_x (I \cap ( \mathbb{R} \backslash [\gamma_-, \gamma_+]) )=0$, the second equality comes from \eqref{Proof Proposition 5.4 when I is inside the bulk} and the final equality comes from \eqref{Proof Proposition 5.4 eq_small}. This concludes our proof. 
\end{proof}

Analogously, we can prove Proposition \ref{pro_ncca_main}. 

\begin{proof}[Proof of Proposition \ref{pro_ncca_main}] We can repeat the proofs of Propositions \ref{prop_norm} and \ref{prop_largscaleset}  by using  Lemma \ref{lem_addtivemeasurelemma}. We omit further details. 

\end{proof}

\section{Local laws for multiplication  of random matrices and proof of Lemma  \ref{lem_2.4.3}}\label{app:bao} 

Take two {\em deterministic} real diagonal matrices, $\Ab=\Ab(n):=\text{diag}\{a_1, \ldots, a_n\}$ and $\Bb=\Bb(n):=\text{diag}\{b_1,\ldots, b_n\}$, where $a_i,b_i\in \mathbb{R}$ for $i=1,\ldots,n$.  As we have seen from Section \ref{sec:proofstrategy}, the arguments boil down to study the local laws of multiplication of random matrices for the model 
\begin{equation}\label{Definition H=UAUB}
\Hb=\Ub \Ab \Ub^\top \Bb\,,
\end{equation}
where $\Ub$ is Haar distributed on $\mathcal{U}(n)$ (or $\mathrm{O}(n))$. Specifically, we are interested in \eqref{eq_defnd}, where $\Ab=\Sigma_x$ and $\Bb=\Sigma_y$ are deterministic matrices determined by the typical locations of $\nu_x$ and $\nu_y$. { We mention that the main result, Proposition \ref{pro_key}, in the section is quite general and can be of independent interests to free probabilists.}

We also assume that there are two $n$-independent absolutely continuous probability measures $\mu_{\alpha}$ and $\mu_{\beta}$ associated with $\mu_{\Ab}$ and $\mu_{\Bb}$.  We denote their densities as $\rho_{\alpha}$ and $\rho_{\beta}$.
We start by introducing the assumptions. The first assumption discusses some quantitative properties of $\mu_\alpha$ and $\mu_\beta$ and the second assumption demonstrates the relationship between $\mu_{\Ab}$ and $\mu_\alpha$ as well as $ \mu_{\Bb}$ and $\mu_\beta$. 

\begin{assumption}\label{assu_product1} 
We assume the following conditions for $\mu_\alpha$ and $\mu_\beta$: 
\begin{enumerate}
\item[(i)] Each measure has one non-trivial interval support, denoted as $[\lambda_-^{\alpha}, \lambda_+^{\alpha}]\subset (0,\infty)$ and $[\lambda_-^{\beta},\lambda_+^{\beta}]\subset (0,\infty)$ respectively,
where $|\lambda_+^{\alpha}| \leq C$ and $|\lambda_+^{\beta}| \leq C$ for some constant $C>0$.
Further, $\rho_{\alpha}$ and $\rho_{\beta}$ are strictly positive in the interior of their supports. 
\item[(ii)] In a small $\delta$-neighborhood of upper edges of the supports, these measures have a power law behavior; that is, there exists a small constant $\delta>0$ and exponents $-1<t^{\alpha}_+,t^{\beta}_+<1$ such that 
\begin{align}
& C^{-1} \leq \frac{\rho_{\alpha}(x)}{(\lambda_+^{\alpha}-x)^{t_+^{\alpha}}} \leq C, \ \forall x \in [\lambda_+^{\alpha}-\delta, \lambda_+^{\alpha}]\,, \nonumber\\
& C^{-1} \leq \frac{\rho_{\beta}(x)}{(\lambda_+^{\beta}-x)^{t_+^{\beta}}} \leq C, \ \forall x \in [\lambda_+^{\beta}-\delta, \lambda_+^{\beta}]\nonumber 
\end{align} 
hold for some constant $0<C\leq 1$. 
\end{enumerate}
\end{assumption}  

Note that for the property (ii) in Assumption \ref{assu_product1}, if $\mu_{\alpha}$ and $\mu_{\beta}$ are of types similar to the MP law, $t_+^{\alpha}=t^{\beta}_+=\frac{1}{2}$ according to the well-known square-root behavior (see, e.g., (\ref{eq)squareroot})) near the right-most edge.

\begin{assumption}\label{ass_product2} 
The relationship between the deterministic $n$-dependent ESDs $\mu_{\Ab}$ and $\mu_{\Bb}$ and the $n$-independent measures $\mu_\alpha$ and $\mu_\beta$ satisfies: 
\begin{enumerate}
\item[(iii)]  
When $n$ is sufficiently large, we have  
\begin{equation}\label{eq_levy}
\mathcal L(\mu_{\Ab}, \mu_\alpha)+\mathcal L(\mu_{\Bb}, \mu_\beta) \leq n^{-2/3+\epsilon_L},
\end{equation}
where $\mathcal{L}(\cdot, \cdot)$ is the Levy distance and $\epsilon_L>0$ is a small constant. 
\item[(iv)] For any $\delta>0$, we have 
\begin{equation*}
\sup (\text{supp} \ \mu_{\Ab}) \leq \lambda_+^{\alpha}+\delta, \ \ \sup (\text{supp} \ \mu_{\Bb}) \leq \lambda_{+}^{\beta}+\delta
\end{equation*}
and
\begin{equation}
\inf (\text{supp} \ \mu_{\Ab}) \geq \lambda_-^{\alpha}-\delta, \ \ \inf (\text{supp} \ \mu_{\Bb}) \geq \lambda_{-}^{\beta}-\delta\label{eq_edgedgedge}
\end{equation}
when $n$ is sufficiently large.
\end{enumerate}
\end{assumption}

\begin{remark} 
We mention that in \cite{DJ}, the authors derived the local laws for the same matrix model near the edge when the Levy distance between $\mu_{\Ab}$ and $\mu_{\Bb}$ and their asymptotic counterpart is of order $n^{-1}$, and 
the results are stated for the Haar unitary distributed $\Ub$.  
\end{remark}

For the sake of self-containedness, we provide the following lemma that translates the closeness of the Stieltjes' transforms and the $M$-transforms in terms of the associated Levy distance. 
\begin{lemma}\label{lem:rbound}
	Let $\delta>0$ be a fixed constant and $f: \mathbb{R}\rightarrow \mathbb{C}$ be continuously differentiable in $(\lambda_{-}^{\alpha}-\delta,\lambda_{+}^{\alpha}+\delta)$. Suppose there exists $n_{0}\in \mathbb{N}$, independent of $f$, such that $\mathcal L(\mu_{\Ab}, \mu_\alpha) \leq \delta/2$ and $\operatorname{supp} \mu_{\Ab} \in[\lambda_{-}^{\alpha}-\delta/2,\lambda_{+}^{\alpha}+\delta/2]$ for all $n \geq n_{0}$. 
	Then there exists $C>0$, independent of $f$, such that we have
	\begin{align*}
		\Big| \int_{\mathbb{R}_{+}} f(x)\ d\mu_{\alpha}(x)-\int_{\mathbb{R}_{+}} f(x) d\mu_{\Ab}(x) \Big|\leq C \norm{f'}_{\mathrm{Lip},\delta}\mathcal L(\mu_{\Ab}, \mu_\alpha),
	\end{align*}
	for all $n \geq n_{0}$, where 
	\begin{align*}
		\norm{f'}_{\mathrm{Lip},\delta}:=\sup_{{x\in[\lambda_{-}^{\alpha}-\delta,\lambda_{+}^{\alpha}+\delta]}}  |f'(x)| +\sup\Big\{ \left|\frac{f'(x)-f'(y)}{x-y} \right| 		:x,y\in[\lambda_{-}^{\alpha}-\delta,\lambda_{+}^{\alpha}+\delta], x\neq y\Big\}.
	\end{align*}
	The same result holds if we replace $\alpha$ and $\Ab$ by $\beta$ and $\Bb$ respectively.
\end{lemma}
\begin{proof}
	Denote the cumulative distribution functions of $\mu_{\alpha}$ and $\mu_{\Ab}$ by $F_{\alpha}$ and $F_{\Ab}$ respectively. Then by the definition of the Levy distance, we have
	\begin{align}
	|F_{\alpha}(x)-F_{\Ab}(x)|
		\leq \mathbf{1}_{[\lambda_{-}^{\alpha/2}-\delta/2,\lambda_{+}^{\alpha}+\delta/2]}(x)
		\times \big[F_{\alpha}(x+\mathcal L(\mu_{\Ab}, \mu_\alpha))-F_{\alpha}(x-\mathcal L(\mu_{\Ab}, \mu_\alpha))+2\mathcal L(\mu_{\Ab}, \mu_\alpha)\big]\,.\nonumber
	\end{align}
	By an integrating by parts, we have
	\begin{equation*}
	\int_{\mathbb{R}_{+}}f(x)d(\mu_{\alpha}-\mu_{\Ab})(x)=\int_{\lambda_{-}^{\alpha}-\delta/2}^{\lambda_{+}^{\alpha}+\delta/2}f'(x)(F_{\alpha}(x)-F_{\Ab}(x)) d x\,,
	\end{equation*}
	which leads to 
	\begin{align}
		&\left|\int_{\mathbb{R}_{+}}f(x)d(\mu_{\alpha}-\mu_{\Ab})(x)\right|\\
		\leq&\,2 (\lambda_{+}^{\alpha}-\lambda_{-}^{\alpha}+\delta) \mathcal L(\mu_{\Ab}, \mu_\alpha){\sup_{x\in [\lambda_{-}^{\alpha}-\delta,\lambda_{+}^{\alpha}+\delta]}}|f'(x)|+\int_{\lambda_{-}^{\alpha}-\delta/2}^{\lambda_{+}^{\alpha}+\delta/2}|f'(x)| [F_{\alpha}(x+\mathcal L(\mu_{\Ab}, \mu_\alpha))\nonumber\\
		&\qquad\qquad-F_{\alpha}(x-\mathcal L(\mu_{\Ab}, \mu_\alpha))]d x\nonumber	\\
		\leq\,& 2 (\lambda_{+}^{\alpha}-\lambda_{-}^{\alpha}+\delta)\mathcal L(\mu_{\Ab}, \mu_\alpha)\sup_{x\in [\lambda_{-}^{\alpha}+\delta,\lambda_{+}^{\alpha}+\delta]}|f'(x)| +{\int_{\lambda_{-}^{\alpha}-\delta}^{\lambda_{+}^{\alpha}+\delta}}\big[|f'(x-\mathcal L(\mu_{\Ab}, \mu_\alpha))|\nonumber\\
		&\qquad\qquad-|f'(x+\mathcal L(\mu_{\Ab}, \mu_\alpha))|\big]F_{\alpha}(x)d x\nonumber	\\
		\leq\,& 2(\lambda_{+}^{\alpha}-\lambda_{-}^{\alpha}+\delta)\norm{f'}_{\mathrm{Lip},\delta}\mathcal L(\mu_{\Ab}, \mu_\alpha)\,,\nonumber
	\end{align}	
	where {the second inequality comes from the assumption and} the last inequality we use the definition of $\norm{f'}_{\mathrm{Lip},\delta}$.
This concludes the proof.
\end{proof}

We introduce the following proposition and show how Lemma \ref{lem_2.4.3} follows from it. 
To this end, recall the first part of \cite[Theorem 2.6]{JHC}.

\begin{theorem}[The first part of Theorem 2.6 of \cite{JHC}]\label{lem_bao2QQ} Suppose that  Assumption \ref{assu_product1} holds. Then $\mu_{\alpha} \boxtimes \mu_{\beta} $ is absolutely continuous and supported on a single non-empty compact interval on $(0, \infty)$. Denote the lower and upper edges of the support of $\mu_{\alpha} \boxtimes \mu_{\beta}$ by 
\begin{equation}\label{eq_defnedgemul}
\lambda_-:=\inf (\text{supp} \mu_{\alpha} \boxtimes \mu_{\beta}), \ \lambda_+:= \sup (\text{supp} \mu_{\alpha} \boxtimes \mu_{\beta})\,.
\end{equation}
\end{theorem}

For some same $\tau>0$ used in $\mathsf{S}(1/4,\tau)$, denote another set of spectral parameter 
\begin{equation}\label{eq_spectralpara}
\mathcal{S}':=\{z=E+\mathrm{i} \eta \in \mathbb{C}^+: |E-\lambda_+| \leq \tau, \ n^{-2/3+\tau} \leq \eta \leq \tau^{-1}\}\,.
\end{equation}

\begin{proposition} \label{pro_key}
Suppose that Assumptions \ref{assu_product1} and \ref{ass_product2} hold. For $z \in \mathcal{S}'$ defined in (\ref{eq_spectralpara}) and $\Hb$ defined in \eqref{Definition H=UAUB}, we have
\begin{equation*}
|m_{\Hb}(z)-m_{\mu_{\Ab} \boxtimes \mu_{\Bb}}(z)| =O_{\prec} \Big(\frac{1}{n \eta}\Big)\,.
\end{equation*} 
\end{proposition}

We leave the technical proof of Proposition \ref{pro_key} to Appendix \ref{SI section Proof of Proposition pro_key}. 

\begin{proof}[Proof of Lemma
\ref{lem_2.4.3}] 

Denote the spectral decomposition of $\breve{\Kb}_x=\Ub_x \Lambda_x\Ub_x^\top$ and similarly $\breve{\Kb}_y=\Ub_y\Lambda_y \Ub_y^\top$. Since $\breve{\Kb}_x$ and $\breve{\Kb}_y$ are isotropic shifts of $\Xb^\top \Xb$ and $\Yb^\top \Yb$, respectively, $\Ub_x$ and $\Ub_y$ are actually the eigenvectors of $\Xb^\top \Xb$ and $\Yb^\top \Yb$, respectively, and as a consequence, $\Ub_x$ is Haar distributed. Note that $\Sb_{xy}^{(3)}$ is similar to  
 \begin{equation*}
 \Sb_{xy}^{(4)}=f^{-2}(2) \Ub_y^\top \Ub_x  \Lambda_x \Ub_x^\top \Ub_y \Lambda_y=f^{-2}(2) \Ub \Lambda_x \Ub^\top \Lambda_y\,,
 \end{equation*}
where we set $\Ub:=\Ub_y^\top \Ub_x$. So $\Sb^{(3)}_{xy}$ and $\Sb^{(4)}_{xy}$ share the same eigenvalues and hence $m_{\Sb^{(3)}_{xy}}(z)=m_{\Sb^{(4)}_{xy}}(z)$. 

Denote the eigenvalues of $\breve{\Kb}_y$ as $\lambda_1^y \geq \lambda_2^y \geq \ldots \geq \lambda_n^y$.
To simplify the notation, denote $\gamma_i^y:=\gamma_{\mu_{c_2,-2f'(2)}}(i)$.
For some small $\epsilon>0$, denote the event $\Omega_y$ by
\begin{align}
\Omega_y:=\{\mbox{For all } &j \in 1,\ldots,n \wedge p_2 \mbox{ such that }|\lambda_j^y-\gamma_j^y| \leq \tilde{j}^{-1/3}n^{-2/3+\epsilon}  \}\,,\label{eq_defnxi}
\end{align}
where $\tilde{j}=\min\{n \wedge p_2+1-j,j\}$.
 By the rigidity of eigenvalues shown in Lemma \ref{lem_rigidity}, there exists some large constant { $D \equiv D(\epsilon)>0,$ where $\epsilon$ is defined in (\ref{eq_defnxi}),} such that the event $\Omega_y$ holds true with probability at least $1-n^{-D}$. Note that if we focus on $\Omega_y$, the quantities involving $\breve{\Kb}_y$ are well controlled. %
 Let
\begin{equation*}
\mathbf{S}_{xy}^{(5)}:=f^{-2}(2) \Ub \Lambda_x \Ub^\top \Sigma_y.
\end{equation*}
Note that 
\begin{align}
\mathbb{E}\big|m_{\Sb_{xy}^{(4)}}(z)-m_{\mathbf{S}_{xy}^{(5)}}(z)\big|^p
=\mathbb{E}[1_{\Omega_y}|m_{\Sb^{(4)}_{xy}}(z)-m_{\mathbf{S}_{xy}^{(5)}}(z)|^p]+\mathbb{E}[1_{\Omega_y^c}|m_{\Sb_{xy}^{(4)}}(z)-m_{\mathbf{S}_{xy}^{(5)}}(z)|^p]\,,\label{eq_wholecontrol}
\end{align}
where $\mathbb{E}[1_{\Omega_y}|m_{\Sb^{(4)}_{xy}}(z)-m_{\mathbf{S}_{xy}^{(5)}}(z)|^p]$ can be bounded with a discussion similar to (\ref{eq_stilebound}) using (\ref{eq_defnxi}) and $\mathbb{E}[1_{\Omega_y^c}|m_{\Sb_{xy}^{(4)}}(z)-m_{\mathbf{S}_{xy}^{(5)}}(z)|^p]$ can be easily bounded by the trivial bound (\ref{eq_stitrivial}) and the fact that $\mathbb{P}(\Omega_y^c) \leq n^{-D}$. We therefore have that
\begin{equation*}
\mathbb{E}\big|m_{\Sb_{xy}^{(4)}}(z)-m_{\mathbf{S}_{xy}^{(5)}}(z)\big|^p \leq \left( C \frac{\log n}{\sqrt{n} \eta^2} \right)^p\,.
\end{equation*}
Now, it suffices to bound the difference between $m_{\Sb_{xy}^{(5)}}(z)$ and $m_{\mathbf{Q}_{xy}}(z)$. 
Similarly, we can show
\begin{equation*}
\mathbb{E}\big|m_{\Sb_{xy}^{(5)}}(z)-m_{\mathbf{Q}_{xy}}(z)\big|^p \leq \left( C \frac{\log n}{\sqrt{n} \eta^2} \right)^p\,.
\end{equation*}
The final step is controlling the closeness of $m_{\mathbf{Q}_{xy}}(z)$ and $m_{\Sigma_x \boxtimes \Sigma_y}(z)$ by Proposition \ref{pro_key}.
To apply Proposition \ref{pro_key}, we now verify Assumptions \ref{assu_product1} and \ref{ass_product2} by setting $\Ab= f^{-2}(2) \Sigma_x$, $\Bb=\Sigma_y$, $\mu_\alpha=\mu_{c_1, \texttt{gMP}}$ and $\mu_\beta=\mu_{c_2,\texttt{gMP}}$ \eqref{eq_defnf}. 

Note that $\Ub$ is Haar distributed when $\Ub_y$ is fixed. Indeed, { we focus our discussion on the high probability event $\Omega_y$}, $\Ub_y$ is a deterministic orthonormal matrix. Since the limiting law of $\breve{\Kb}_y$ is a  shifted MP law, (i) of Assumption \ref{assu_product1}  holds immediately and (ii) follows from the square root behavior of MP law (c.f. (\ref{eq)squareroot})). For Assumption \ref{ass_product2}, from Lemma \ref{lem_rigidity} and (iii) of Lemma \ref{lem_collecttion}, we get the following bound : 
\begin{equation}\label{eq_epsilondefined}
\mathcal L(\mu_{\Ab}, \mu_{c_1,\texttt{gMP}})+\mathcal L(\mu_{\Bb}, \mu_{c_2,\texttt{gMP}}) \leq n^{-2/3+\epsilon}\,.
\end{equation}
Therefore, (iii) is satisfied with $\epsilon_L=\epsilon$. (iv) again follows from the basic properties of MP law in Lemma \ref{lem_mplaw}. 
This completes the proof of Lemma \ref{lem_2.4.3}. 

\end{proof}

\section{Technical proof of Proposition \ref{pro_key}}\label{SI section Proof of Proposition pro_key}

The proof strategy of Proposition \ref{pro_key} was first employed in \cite{bao20171} to investigate the local laws of addition of random matrices. 
The strategy contains two steps. First, we show that the properties of $\mu_{\alpha} \boxtimes \mu_{\beta}$ can be extended to $\mu_{\Ab} \boxtimes \mu_{\Bb}$ with some controls. This step is necessary since the result of Proposition \ref{pro_key} contains the Stieltjes transform of $\mu_{\Ab} \boxtimes \mu_{\Bb}$.  Second, we prove the local laws.  
Since the proof is essentially the same as that of \cite[equation (2.18)]{bao20171} except that we need to prove the statement on a smaller spectral domain $\mathcal{S}'$, we can take the verbatim of the proofs in \cite[Sections 5-9]{bao20171}. We omit the proofs and only show the key inputs for the first step, which is summarized in Lemma \ref{lem_addtivemeasurelemma} below, and refer readers with interest to \cite[Sections 5-9]{bao20171} for details. 

To state Lemma \ref{lem_addtivemeasurelemma}, we need the following notations. Denote $\Omega_{\Ab}:=\Omega_{\mu_{\Ab}}$, $\Omega_{\Bb}:=\Omega_{\mu_{\Bb}}$, $\Omega_{\alpha}:=\Omega_{\mu_\alpha}$, $\Omega_{\beta}:=\Omega_{\mu_\beta}$ (\ref{eq_omega}), $M_{\Ab}:=M_{\mu_{\Ab}}$, $M_{\Bb}:=M_{\mu_{\Bb}}$, $M_{\alpha}:=M_{\mu_{\alpha}}$ and $M_{\beta}:=M_{\mu_{\beta}}$ for the $M$-transforms in Definition \ref{defn_mtransform}. Furthermore, like \eqref{NotationL definition}, we denote 
\begin{equation}\label{Definition:LA and LB}
L_{\Ab}(z):=\frac{M_{\Ab}(z)}{z} \,\,\mbox{ and }\,\,L_{\Bb}(z):=\frac{M_{\Bb}(z)}{z}\,. 
\end{equation}
Consider the following functions:
\begin{align*}
& S_{\Ab \Bb}(z):=z^2 L_{\Bb}'(\Omega_{\Ab}(z)) L'_{\Ab}(\Omega_{\Bb}(z))-1,
\end{align*}
\begin{align*}
& \mathcal{T}_{\Ab}(z):=\frac{1}{2} [z L^{''}_{\Bb}(\Omega_{\Ab}(z)) L'_{\Ab}(\Omega_{\Bb}(z))+(z L'_{\Bb}(\Omega_{\Ab}(z)))^2 L_{\Ab}^{''}(\Omega_{\Bb}(z))], \\
& \mathcal{T}_{\Bb}(z):=\frac{1}{2} [z L^{''}_{\Bb}(\Omega_{\Ab}(z)) L'_{\Ab}(\Omega_{\Bb}(z))+(z L'_{\Bb}(\Omega_{\Ab}(z)))^2 L_{\Ab}^{''}(\Omega_{\Bb}(z))]\,.
\end{align*}
By replacing the pair $(\Ab, \Bb)$ with $(\alpha, \beta)$, we can define $L_{\alpha}, L_{\beta}, S_{\alpha \beta}, \mathcal{T}_{\alpha}$ and $\mathcal{T}_{\beta}$ analogously. 

\begin{remark}
We mention that the counterparts of $S_{\Ab \Bb}(z), \mathcal{T}_{\Ab}(z)$ and $\mathcal{T}_{\Bb}(z)$ for the addition of random matrices defined in \cite[equation (3.1)]{bao20171} look slightly different from what we define here. For the reader's convenience, we briefly discuss how we obtain the above quantities.  
Consider 
\begin{equation*}
\Delta:=\{(\omega_1, \omega_2, z) \in (\mathbb{C}^+)^3: \text{arg} \ \omega_1, \text{arg} \ \omega_2 \geq \text{arg} \ z \}\subset (\mathbb{C}^+)^3,
\end{equation*}
and define $\Phi_{\alpha},\,\Phi_{\beta}: \Delta \rightarrow \mathbb{C}$ so that 
\begin{align*}
&\Phi_{\alpha}(\omega_1, \omega_2, z):=\frac{M_{{\alpha}}(\omega_2)}{\omega_2}-\frac{\omega_1}{z}\,,  \\
&\Phi_{\beta}(\omega_1, \omega_2, z):=\frac{M_{{\beta}}(\omega_1)}{\omega_1}-\frac{\omega_2}{z}\,.\nonumber 
\end{align*}
First of all, (ii) and (iii) of Proposition \ref{prop_muloriginal 2about M} can be written as 
\begin{equation}\label{eq)system}
\Phi_{\alpha \beta}(\Omega_{\alpha}(z), \Omega_{\beta}(z),z)=0\,,
\end{equation}
where $\Phi_{\alpha \beta} =(\Phi_{\alpha}, \Phi_{\beta}): \Delta \rightarrow \mathbb{C}^2$. Indeed, denoting by $D$ the partial differential operator with respect to $\omega_{1}$ and $\omega_{2}$, the first derivative of $\Phi_{\alpha\beta}$ is given by
\begin{equation*}
	D\Phi_{\alpha\beta}(\omega_{1},\omega_{2},z)=\begin{pmatrix}
		-z^{-1} & L'_{\mu_{\alpha}}(\omega_{2}) \\
		L'_{\mu_{\beta}}(\omega_{1}) & -z^{-1}
	\end{pmatrix},
\end{equation*}
whose determinant is equal to $-z^{-2} S_{\alpha\beta}(z)$ at the point $(\Omega_{\alpha}(z),\Omega_{\beta}(z),z)$. Similarly, using $\Phi_{\alpha\beta}(\Omega_{\alpha}(z),\Omega_{\beta}(z),z)=0$, we find that
\begin{align*}
	\mathcal{T}_{\alpha}(z)&= z\left[\frac{\partial}{\partial\omega_{1}}\det D \Phi_{\alpha\beta}(\omega_{1},zL_{\mu_{\beta}}(\omega_{1}),z)\right]_{\omega_{1}=\Omega_{\alpha}(z)},\nonumber \\
	\mathcal T_{\beta}(z)&=z\left[\frac{\partial}{\partial\omega_{2}}\det D\Phi_{\alpha\beta}(zL_{\mu_{\alpha}}(\omega_{2}),\omega_{2},z)\right]_{\omega_{2}=\Omega_{\beta}(z)}\,.\nonumber
\end{align*}
By replacing the pair $(\alpha,\beta)$ with $(\Ab,\Bb)$, we can define $\Phi_{\Ab \Bb}$, $S_{\Ab \Bb}$, $\mathcal{T}_{\Ab}$, and $\mathcal{T}_{\Bb}$ analogously.
\end{remark}

We now state the lemma establishing the properties of $\mu_{\Ab} \boxtimes \mu_{\Bb}$. 

\begin{lemma} \label{lem_addtivemeasurelemma} 
Suppose Assumptions \ref{assu_product1} and \ref{ass_product2} hold true. 
Denote $\kappa:=\kappa(E):=|E-\lambda_+|$, where $\lambda_+$ is defined in \eqref{eq_defnedgemul}. Then, we have the following statements: 

 \begin{enumerate}
\item[(i)] There exist constants $k>0$ and $K>0$ such that
 \begin{align*}
 & \inf_{z \in \mathcal{S}'} \text{dist}( \Omega_{\Ab}(z), \supp\mu_{\Bb}) \geq k \,,\nonumber \\
 & \inf_{z \in \mathcal{S}'} \text{dist}( \Omega_{\Bb}(z), \supp\mu_{\Ab}) \geq k\,, \\
 & k \leq |\Omega_{\Ab}(z)| \leq K, \ k \leq |\Omega_{\Bb}(z)| \leq K\nonumber 
 \end{align*}
 hold uniformly on $z\in\mathcal{S}'$.
 
\item[(ii)]  
We have uniformly for $z=E+\mathrm{i}\eta \in \mathcal{S}'$ that
 \begin{equation*}
 \operatorname{Im} m_{\mu_{\Ab} \boxtimes \mu_{\Bb}}(z)  \sim
 \begin{cases}
 \sqrt{\kappa+\eta}, & \ \text{if} \ E \in \text{supp} \mu_{\Ab} \boxtimes \mu_{\Bb}, \\
 \frac{\eta}{\sqrt{\kappa+\eta}}, & \ \text{if} \ E > \text{supp} \mu_{\Ab} \boxtimes \mu_{\Bb}.
 \end{cases}
 \end{equation*}
 
\item[(iii)]  
We have uniformly for $z\in \mathcal{S}'$ that
\begin{align}
 &S_{\Ab\Bb}(z) \sim \sqrt{\kappa+\eta}, \nonumber\\ 
 &|\mathcal{T}_{\Ab}(z)| \leq C,  \ |\mathcal{T}_{\Bb}(z)| \leq C \,.\nonumber
\end{align}
Furthermore, when $|z-E_+| \leq \tau$, there exists some constant $c>0$ such that 
\begin{equation*}
|\mathcal{T}_{\Ab}(z)| \geq c, \ |\mathcal{T}_{\Bb}(z)| \geq c. 
\end{equation*}

\item[(iv)] For $\Omega_{\Ab}, \Omega_{\Bb}$ and $S_{\Ab\Bb}$, we have 
\begin{align*}
&|S_{\Ab\Bb}'(z)| \leq  \frac{C}{\sqrt{\kappa+\eta}}\,,\nonumber\\
&|\Omega_{\Ab}'(z)| \leq  \frac{C}{\sqrt{\kappa+\eta}}, \,\, |\Omega_{\Bb}'(z)| \leq  \frac{C}{\sqrt{\kappa+\eta}}\nonumber
\end{align*}
uniformly for $z =E+\mathrm{i}\eta\in \mathcal{S}'$ and some constant $C>0$. 
\end{enumerate}
\end{lemma}

The rest of the work is devoted to the proof of Lemma \ref{lem_addtivemeasurelemma}. Its proof is split into two steps. In the first step, we derive the properties for the $n$-independent measures $\mu_{\alpha}$ and $\mu_{\beta}$. Most of the results for this step have been proved in \cite{JHC}. We summarize them in Section \ref{sec:subsubalbeta} for the sake of self-containedness. In the second step, we show that the results about $\mu_{\alpha}$ and $\mu_{\beta}$ can be carried over to the $n$-dependent measures $\mu_{\Ab}$ and $\mu_{\Bb}$. The second step is thus the main ingredient of the proof of Lemma \ref{lem_addtivemeasurelemma}, which will be given in Section \ref{sec_prooflema416}. 

\subsection{Properties of $\mu_{\alpha} \boxtimes \mu_{\beta}$}\label{sec:subsubalbeta}
We summarize some properties of the subordination  functions and the free multiplication at the regular edges under Assumptions \ref{assu_product1} and \ref{ass_product2}. Most of these results have been proved in \cite{JHC} and will be used in the proof of Lemma \ref{lem_addtivemeasurelemma}. The following lemma collects the preliminary results on the $n$-independent measure $\mu_{\alpha} \boxtimes \mu_{\beta}$. It is known from \cite[Lemma 2.9]{JHC} that both $\Omega_{\alpha}(\cdot)$ and $\Omega_{\beta}(\cdot)$ can be continuously extended to $(\mathbb{R}\cup\{\infty\})\backslash \{0\}$, with values in $\overline{\mathbb{C}^+}$.

\begin{lemma}[Lemma 6.5 and Proposition 6.10  of \cite{JHC}]\label{lem_bao}  
Suppose Assumption \ref{assu_product1}  holds. Then the following statements hold true: 
\begin{enumerate}
\item[(i)] For any compact subset $K \subset \mathbb{C}^+ \cup (0, \infty)$, there exists some constant $C>1$ such that for all $z \in K$, 
\begin{align*}
&C^{-1} \leq |\Omega_{\alpha}(z)| \leq C \,,\nonumber\\
&C^{-1} \leq  |\Omega_{\beta}(z)| \leq C\,. \nonumber
\end{align*} 

\item[(ii)] Denote 
\begin{align*}
&k_{\alpha}:=\inf_{z\in \mathbb{C}_+} \text{dist}(\Omega_{\alpha}(z), \text{supp} \ \mu_{\beta}) \,,\\
&k_{\beta}:=\inf_{z \in \mathbb{C}_+} \text{dist} (\Omega_{\beta}(z), \text{supp} \ \mu_{\alpha}) \,.\nonumber
\end{align*}
There exists some constant $\varsigma>0$ such that
\begin{equation*}
\min\{k_{\alpha}, k_{\beta}\} \geq \varsigma\,. 
\end{equation*}

\item[(iii)] Furthermore, we have that \cite[equation (6.97)]{JHC}
\begin{align*}
&0< \Omega_{\alpha}(\lambda_-) < \lambda_-^{\beta}< \lambda_+^{\beta}<\Omega_{\alpha}(\lambda_+)\,,\nonumber\\
& 0<\Omega_{\beta}(\lambda_-)<\lambda_-^{\alpha}<\lambda_+^{\alpha}<\Omega_{\beta}(\lambda_+)\,,\nonumber
\end{align*}
where $\lambda_-$ and $\lambda_+$ are defined in \eqref{eq_defnedgemul}.

\end{enumerate}
\end{lemma}

\begin{remark}
 We remark that by a discussion similar to \cite[Lemmas 3.2 and 3.3]{bao20171}, we find that (i) and (ii) of Lemma \ref{lem_bao} guarantee that for some small constant $\upsilon>0$, one of the followings must hold true for $z \in \mathcal{S}'$: 
\begin{align}
&\text{Re} \ \Omega_{\alpha}(z)\leq \lambda_-^{\beta}-\upsilon\,,\nonumber \\
&\text{Re} \ \Omega_{\alpha}(z) \geq \lambda_+^{\beta}+\upsilon\,,\label{eq_upsilolololo} \\ 
&\operatorname{Im}\Omega_{\alpha}(z) \geq \upsilon\,. \nonumber
\end{align}
Similar results hold true when we swap $\alpha$ and $\beta$. As a consequence, by recalling the definition
\begin{equation*}
\Omega_{\beta}(z) m_{\mu_{\alpha}}( \Omega_{\beta}(z))+1=\int \frac{x}{x-\Omega_{\beta}(z)} d \mu_{\alpha}(x),
\end{equation*}
for $z \in \mathcal{S}'$, we have
\begin{align}
&|\Omega_{\beta}(z) m_{\mu_{\alpha}}(\Omega_\beta(z))+1| \geq c, \label{bbbb_bound}\\
&|\Omega_{\alpha}(z) m_{\mu_{\beta}}(\Omega_\alpha(z))+1| \geq c\nonumber
\end{align}
for some constant $c>0$.
\end{remark}

The next lemma characterizes the locations of the edges of $\mu_{\alpha} \boxtimes \mu_{\beta}$.  
\begin{lemma}[Propositions 6.13 and 6.15 of \cite{JHC}] Suppose Assumption \ref{assu_product1} holds.  The following equation
\begin{equation*}
\left(\frac{\Omega_{\beta}(z)}{M_{{\alpha}}(\Omega_{\beta}(z))} M_{\alpha}(\Omega'_{\beta}(z)) -1\right) \left( \frac{\Omega_{\alpha}(z)}{M_{{\beta}}(\Omega_{\alpha}(z))} M_{\beta}(\Omega'_{\alpha}(z)) -1\right)=1
\end{equation*}
holds when $z=\lambda_+$ or $z=\lambda_-$, the edges defined in (\ref{eq_defnedgemul}).
\end{lemma}

The following lemma collects the square root properties of the subordination functions at the edge of $\lambda_+$. This square root behavior also hold near $\lambda_-$, but we do not state it since we do not need it. 
\begin{lemma}[Proposition 6.16 of \cite{JHC}]\label{lem_imim} Suppose Assumption \ref{assu_product1} holds. Then there exist positive constants $\gamma^{\alpha}_{+}$ and $\gamma^{\beta}_{+}$ such that
	\begin{equation}
		\Omega_{\alpha}(z)=\Omega_{\alpha}(\lambda_{+})+\gamma^{\alpha}_{+}\sqrt{z-\lambda_{+}}+O(|z-\lambda_{+}|^{3/2})\,,
	\end{equation}
	for $z$ in a neighborhood of $\lambda_+$ with the principal branch of square root with $\sqrt{-1}=\mathrm i$.
The same asymptotic holds with $\alpha$ replaced by $\beta$. 

\end{lemma}

The square root asymptotic of the subordination functions stated in Lemma \ref{lem_imim} implies the square root asymptotic of the density of $\mu_{\alpha}\boxtimes\mu_{\beta}$, which is listed in the following Theorem.

\begin{theorem}[The second part of Theorem 2.6 of \cite{JHC}]\label{lem_bao2} Suppose that  Assumption \ref{assu_product1} holds. If the density of $\mu_{\alpha} \boxtimes \mu_{\beta}$ is denoted by $\rho$, we have the following facts:
\begin{enumerate}
	\item[(i)] $\rho$ is continuous on $\mathbb{R}$ and  $\{x\in(0,\infty):\rho(x)>0\}=(\lambda_{-}, \lambda_{+})$.
	
	\item[(ii)] There exists a constant $C>0$ such that for all $x\in[\lambda_{-},\lambda_{+}]$,
	\begin{equation}
	 C^{-1}\leq \frac{\rho(x)}{\sqrt{(\lambda_{+}-x)(x-\lambda_{-})}} \leq C\,.
	\end{equation} 
\end{enumerate}
\end{theorem}

The following lemma provides the Nevanlinna-Pick representation of the $L$-function defined in \eqref{Definition:LA and LB} and the $\Omega$-functions associated with the subordination functions defined in \eqref{eq_omega}.  

\begin{lemma}[Lemma 6.2 of \cite{JHC}]\label{lem:reprM}
Suppose that $\mu_{\alpha}$ and $\mu_{\beta}$ satisfy Assumption \ref{assu_product1}. Then there exists unique Borel measures $\widehat{\mu}_{\alpha}$  on $\mathbb{R}^+$ such that the followings hold
	\begin{align}
		L_{{\alpha}}(z)&=1+m_{\widehat{\mu}_{\alpha}}(z)\,,\nonumber\\
		\widehat{\mu}_{\alpha}(\mathbb{R}_{+})&=\text{Var}{\mu_{\alpha}}=\int_{\mathbb{R}_{+}} x^{2}\ \mu_{\alpha}(x)-1\,,\label{eq:reprM}\\
		\text{supp} \ \widehat{\mu}_{\alpha}&=\text{supp} \ \mu_{\alpha}\,.	 \nonumber 
	\end{align}
Similar results hold true if we replace $\alpha$ with $\beta$. 
\end{lemma} 

To characterize the behavior of $S_{\alpha \beta}, \mathcal{T}_{\alpha}$ and $\mathcal{T}_{\beta}$, we need the following lemma. Its proof is similar to those as in equations of \cite[Corollary 3.10]{bao20171} and we omit the details here. Given $\eta^*>\tau^{-1},$ we denote the parameter set $\mathcal{S}^*$ as 
\begin{equation} \label{eq_upsilon}
\mathcal{S}^*:=\{ \kappa \leq \upsilon^*, 0 \leq \eta \leq \eta^* \},
\end{equation}
where $\upsilon^*:=\max \{\tau^{-1}, \upsilon \}$ and $\upsilon$ is defined in (\ref{eq_upsilolololo}). 

\begin{lemma}\label{lem:suborder}
Suppose Assumption \ref{assu_product1} holds.  For $z \in \mathcal{S}^*$, we have
	\begin{align}
		&L_{{\alpha}}'(\Omega_{\beta}(z))\sim  M_{{\alpha}}'(\Omega_{\beta}(z)) \sim 1\,, \nonumber\\
		&L_{{\alpha}}''(\Omega_{\beta}(z))\sim  M_{{\alpha}}''(\Omega_{\beta}(z)) \sim 1\,,\label{eq:LMder}\\
		&|\Omega_{\alpha}'(z)| \sim |z-\lambda_{+}|^{-1/2} \,,  \nonumber\\
		&|\Omega_{\alpha}''(z)|\sim |z-\lambda_{+}|^{-3/2}\,.\nonumber 
		\end{align}
		Similar results hold true if we $\alpha$ and $\beta$ are swapped. Moreover, we have
	\begin{align}
	&m_{\mu_{\alpha}\boxtimes\mu_{\beta}}'(z)\sim\frac{1}{|z-\lambda_{+}|^{1/2}}\,,\nonumber \\
	& m_{\mu_{\alpha}\boxtimes\mu_{\beta}}''(z)\sim \frac{1}{|z-\lambda_{+}|^{3/2}}.\label{eq:Stieltjesder}
	\end{align}
\end{lemma}

Finally, we summarize the properties of 
 $S_{\alpha\beta}$, $\mathcal{T}_{\alpha}$, and $\mathcal{T}_{\beta}$ around the edges (and far away from bulk) in the following lemma. The proof is similar to \cite[Corollary 3.11]{bao20171} and we refer readers with interest to the proof therein. 
\begin{lemma}\label{Lemma Omega alpha, S alphabeta Talpha control}
Suppose Assumption \ref{assu_product1} holds. Then we have 
	\begin{itemize}
		\item[(i)] For $z \in \mathcal{S}^*$, we have uniformly that
		\begin{align}
		\operatorname{Im}\Omega_{\alpha}(z)&\,\sim\operatorname{Im}\ \Omega_{\beta}(z)\sim \operatorname{Im}M_{\mu_{\alpha}\boxtimes\mu_{\beta}}(z)\nonumber\\
		&\,\sim \operatorname{Im} 
		\ m_{\mu_{\alpha}\boxtimes\mu_{\beta}}(z)\nonumber\\
			&\, \sim \left\{
			\begin{aligned}
				&\sqrt{\kappa+\eta} &\text{if }&E\in[\lambda_{-},\lambda_{+}],	\\
				&\dfrac{\eta}{\sqrt{\kappa+\eta}} &\text{if }&E\notin[\lambda_{-},\lambda_{+}].
			\end{aligned}
			\right.
		\end{align}

		\item[(ii)] We have 
		\begin{align}
			&S_{\alpha\beta}(z)\sim\sqrt{\kappa+\eta}\,,\\
			& |\mathcal{T}_{\alpha}(z)| \leq C\,,\quad |\mathcal{T}_{\beta}(z)|\leq C\nonumber,
		\end{align}
		uniformly for $z\in \mathcal{S}^*$ for some constant $C>0$. 
		
		\item[(iii)] For $z \in \mathcal{S}^*$, we also have
		\begin{align}\label{eq_talpha}
			\mathcal{T}_{\alpha}(z)\sim 1\,, \quad\mathcal{T}_{\beta}(z)\sim 1.
		\end{align}
	\end{itemize}
\end{lemma}

With the above preparations, we are ready to prove Lemma \ref{lem_addtivemeasurelemma}.

\subsection{Proof of Lemma \ref{lem_addtivemeasurelemma}}\label{sec_prooflema416} 

In the proof of Lemma \ref{lem_addtivemeasurelemma}, we need the following lemma.

\begin{lemma}\label{lem_lasttwo}
Suppose $\mu_{\alpha}, \mu_{\beta}, \mu_{\Ab}$ and $\mu_{\Bb}$ satisfy Assumptions \ref{assu_product1} and \ref{ass_product2}. For a fixed large constant $c_0>10$ and $\eta_0>0$, let $z \in \mathcal{D}$ with 
\begin{equation}\label{spectral_domaind}
\mathcal{D}:=\mathcal{D}_{\text{in}} \cup \mathcal{D}_{\text{out}}\,,
\end{equation}
where
\begin{align*}
& \mathcal{D}_{\text{in}}:= \{z: \lambda_+-\tau \leq E \leq \lambda_++n^{-2/3+c_0 \epsilon}, \ n^{-2/3+c_0 \epsilon} \leq \eta \leq \eta_0 \}, \\
&  \mathcal{D}_{\text{out}}:= \{z: n^{-2/3+c_0 \epsilon}+\lambda_+ \leq E \leq \tau^{-1}, \ 0<\eta \leq \eta_0 \}\,.
\end{align*}
Here $\epsilon$ is defined in (\ref{eq_epsilondefined}).
Then, there exists some constant $C>0$ so that when $n$ is sufficiently large, the following statements hold. 
\begin{enumerate}
\item[(i)] For all $z \in \mathcal{D}$, we have 
\begin{align}
|\Omega_{\Ab}(z)-&\Omega_{\alpha}(z)|+|\Omega_{\Bb}(z)-\Omega_{\beta}(z)|\label{eq_suborclose}\\
& \leq C \frac{n^{-2/3+\epsilon}}{\sqrt{|z-\lambda_+|}} \leq n^{-1/2+(1-c_0/2) \epsilon}\,. \nonumber
\end{align}
\item[(ii)] For all $z \in \mathcal{D}_{\text{out}}$, we have 
\begin{align}
|\operatorname{Im}\Omega_{\Ab}(z)-&\operatorname{Im} \Omega_{\alpha}(z)|+|\operatorname{Im}\Omega_{\Bb}(z)-\operatorname{Im} \Omega_{\beta}(z)|\nonumber\\
& \leq C \frac{n^{-2/3+c_0 \epsilon} \operatorname{Im}(\Omega_{\alpha}(z)+\Omega_{\beta}(z))+\eta}{\sqrt{|z-\lambda_+|}}\,.\nonumber 
\end{align}
\end{enumerate}
\end{lemma}

The proof of this lemma is postponed after completing the proof of Lemma \ref{lem_addtivemeasurelemma}. Note that the domain considered in this lemma is more general than $\mathsf{S}(1/4,\tau)$.  With this lemma, we proceed to prove Lemma \ref{lem_addtivemeasurelemma}.

\begin{proof}[Proof of Lemma \ref{lem_addtivemeasurelemma}] We choose $\eta_0$ such that $\eta_0>\tau^{-1}$ and $\eta_0 \leq \eta^*$; that is, $\mathcal{S}' \subset \mathcal{D} \cap \mathcal{S}^*$. Since $\mathcal{S}' \subset \mathcal{D}$, (i) follows directly from Lemmas \ref{lem_bao} and \ref{lem_lasttwo}. We now prove (ii). We will need the following identities (for instance see \cite[equations (2.15) and (2.16)]{JHC})
\begin{equation}
z m_{\mu_\Ab \boxtimes \mu_\Bb}(z)=\Omega_\Bb(z) m_{\mu_\Ab}(\Omega_\Bb(z)),
\end{equation}
and 
\begin{equation*}
\operatorname{Im} z m_{\mu_\Ab \boxtimes \mu_\Bb}(z)=\operatorname{Im} \Omega_{\Bb}(z)\int\frac{x}{|x-\Omega_{\Bb}(z)|^{2}}d\mu_{\Ab}(x)\,.
\end{equation*}
For the approximation of $\operatorname{Im} m_{\mu_{A}\boxtimes\mu_{B}}(z)$ in (ii), we first see that
	\begin{align}
	&	\operatorname{Im}zm_{\mu_{\Ab}\boxtimes\mu_{\Bb}}(z)\, =\operatorname{Im}\Omega_{\Bb}(z)m_{\mu_{\Ab}}(\Omega_{\Bb}(z))\nonumber\\
		&\,=\operatorname{Im}\Omega_{\Bb}(z)\int\frac{x}{|x-\Omega_{\Bb}(z)|^{2}}d\mu_{\Ab}(x)\sim \operatorname{Im}\Omega_{\Bb}(z)\,,\nonumber
	\end{align}
where we used (i) in the last step. Specifically, we use the facts that  
\begin{equation}
\inf_{z \in \mathbb{C}^+} \text{dist}( \Omega_{\Bb}(z), \supp\mu_{\Ab}) \geq k,
\end{equation}
and $k \leq |\Omega_{\Bb}(z)| \leq K$ for $0<k\leq K$ so that $\int\frac{x}{|x-\Omega_{\Bb}(z)|^{2}}d\mu_{A}(x)$ is of order $1$. By \eqref{eq_suborclose} in  Lemma \ref{lem_lasttwo}, for $z=E+\mathrm{i}\eta$ and $\kappa=|E-\lambda_+|$, we claim that $\operatorname{Im} \Omega_{\Bb}(z)$ enjoys the same asymptotic as $\operatorname{Im} \Omega_{\beta}(z)$ due to the following bound:
	\begin{equation}\label{Control of Omega beta and Omega B}
		|\Omega_{\beta}(z)-\Omega_{\Bb}(z)|\leq C\frac{n^{-2/3+\epsilon}}{\sqrt{\kappa+\eta}} \,.
		\end{equation}
Indeed, note that $|z-\lambda_+|=\sqrt{\kappa^2+\eta^2}$ and $\eta,\,\kappa\geq0$. When $3\kappa \leq \eta$, we have that 
\begin{align}
&|z-\lambda_+| \geq  \sqrt{\kappa^2+\eta^2-2\kappa\eta} =\eta-\kappa\\
=&\,\frac{1}{2}(\eta+\kappa)+\frac{1}{2} \eta-\frac{3}{2} \kappa \geq \frac{1}{2}(\eta+\kappa)\,.\nonumber
\end{align}
Similarly, when $3 \kappa >\eta$, by switching the role of $\kappa$ and $\eta$, we get the desired bound. More discussion of $|\Omega_{\beta}(z)-\Omega_{\Bb}(z)|$ can be found in Remark \ref{Remark of control of Omega beta and Omega B}.
Since
	\begin{align}
	&	\operatorname{Im}(zm_{\mu_{\Ab}\boxtimes\mu_{\Bb}}(z))
		=\eta\int\frac{x}{|x-z|^{2}} d\mu_{\Ab}\boxtimes\mu_{\Bb}(x) \nonumber\\
		&\sim \eta\int\frac{1}{|x-z|^{2}}d\mu_{\Ab}\boxtimes\mu_{\Bb}(x) =\operatorname{Im}m_{\mu_{\Ab}\boxtimes\mu_{\Bb}}(z)\,,
	\end{align}
	where the approximation $\sim$ holds since $x$ is bounded by (i) of Assumption \ref{assu_product1}. We hence conclude the proof of (ii).

For (iii), we prove it via comparing $S_{\Ab \Bb}(z)$ and $S_{\alpha \beta}(z)$:
\begin{align}
&|S_{\alpha \beta}(z)-S_{\Ab \Bb}(z)|\\
=&\,|z^2|\big| L'_{\Ab}(\Omega_{\Bb}(z))[L'_{\beta}(\Omega_{\alpha}(z))-L_{\Bb}'(\Omega_{\Ab}(z))]\nonumber\\
&\quad\qquad+L'_{\beta}(\Omega_{\alpha}(z))[L_{\alpha}'(\Omega_{\beta}(z))-L'_{\Ab}(\Omega_\Bb(z))] \big|\,.\nonumber
\end{align}	
Since $z$ is bounded from above and below, and due to the control of $L'_{\beta}(\Omega_{\alpha}(z))$ in (\ref{eq:LMder}), we only need to control $L'_{\Ab}(\Omega_{\Bb})$, $L_{\Bb}'(\Omega_{\Ab})-L'_{\beta}(\Omega_{\alpha})$ and $L'_{\Ab}(\Omega_\Bb)-L_{\alpha}'(\Omega_{\beta})$.
Note that
	\begin{align}
		&|L'_{{\beta}}(\Omega_{\alpha}(z))-L'_{{\Bb}}(\Omega_{\Ab}(z))|\\
		\leq&\, |L'_{{\beta}}(\Omega_{\alpha}(z))-L'_{{\Bb}}(\Omega_{\alpha}(z)|+|L'_{{\Bb}}(\Omega_{\alpha}(z))-L'_{{\Bb}}(\Omega_{\Ab}(z))|\,,\nonumber
	\end{align}
where we control the right hand side term by term.
Note that by definition, for a probability measure supported on $(0,\infty)$, we have 
\begin{align}
L'_{\mu}(z)&=\Big( \frac{1}{z}-\frac{1}{z(z m_{\mu}(z)+1)} \Big)'\label{eq_uu1}\\
&=-\frac{1}{z^2}+\frac{(z(zm_{\mu}(z)+1))'}{z^2(z m_{\mu}(z)+1)^2},\nonumber
\end{align}
where
\begin{equation}\label{eq_uu2}
(z(zm_{\mu}(z)+1))'=\int \frac{x^2}{(x-z)^2} d \mu(x),
\end{equation}
and
\begin{equation}\label{eq_uu3}
zm_{\mu}(z)+1=\frac{1}{1-M_{\mu}(z)}\,.
\end{equation}
Therefore, the first term $|L'_{{\beta}}(\Omega_{\alpha}(z))-L'_{{\Bb}}(\Omega_{\alpha}(z))|$ becomes
	\begin{align}
		&|L'_{{\beta}}(\Omega_{\alpha}(z))-L'_{{\Bb}}(\Omega_{\alpha}(z))| \nonumber \\
		=&\,\frac{1}{|\Omega_{\alpha}(z)|^{2}}\bigg\vert(\Omega_{\alpha}(z)m_{{\Bb}}(\Omega_{\alpha}(z))+1)^{-2}\int\frac{x^{2}}{(x-\Omega_{\alpha}(z))^2} d\mu_{\Bb}(z)	\nonumber  \\
		&\qquad-(\Omega_{\alpha}(z)m_{\mu_{\beta}}(\Omega_{\alpha}(z))+1)^{-2} \int\frac{x^{2}}{(x-\Omega_{\alpha}(z))^2} d\mu_{\beta}(x)\bigg\vert\nonumber\\
		=&\,\frac{|\mathcal{C}_{\Bb}(z)-\mathcal{C}_{\beta}(z)|}{|\Omega_{\alpha}(z)|^{2}}\,,\nonumber
	\end{align}
	where
	\begin{align}
\mathcal{C}_{\Bb}(z)&:=(\Omega_{\alpha}(z)m_{{\Bb}}(\Omega_{\alpha}(z))+1)^{-2}\int\frac{x^{2}}{(x-\Omega_{\alpha}(z))^2} d\mu_{\Bb}(z)\,,\nonumber \\
\mathcal{C}_{\beta}(z)&:=(\Omega_{\alpha}(z)m_{\mu_{\beta}}(\Omega_{\alpha}(z))+1)^{-2} \int\frac{x^{2}}{(x-\Omega_{\alpha}(z))^2} d\mu_{\beta}(x)\,.\nonumber
\end{align}
	We claim that
	\begin{align}\label{eq_bbbbbbbbdd}
	|\mathcal{C}_{\Bb}(z)-\mathcal{C}_{\beta}(z)| \leq C\big[\mathcal{L}(\mu_{\Ab}, \mu_{\alpha})+\mathcal{L}(\mu_{\Bb}, \mu_{\beta})\big]
	\end{align}
	for some $C>0$.
To simplify the notations, denote 
\begin{align}
\mathbf{d}&:=\mathcal{L}(\mu_{\Ab}, \mu_{\alpha})+\mathcal{L}(\mu_{\Bb}, \mu_{\beta})\,,\\
\mathcal{C}_{\Bb \beta}(z)&:=(\Omega_{\alpha}(z)m_{{\Bb}}(\Omega_{\alpha}(z))+1)^{-2} \int\frac{x^{2}}{(x-\Omega_{\alpha}(z))^2} d\mu_{\beta}(x)\,.\nonumber
\end{align}
By the triangle inequality and (\ref{bbbb_bound}), for some constant $C_1>0$, we have   
\begin{align}
&|\mathcal{C}_{\Bb}(z)-\mathcal{C}_{\beta}(z)| \nonumber\\
\leq\,&   |\mathcal{C}_{\Bb}(z)-\mathcal{C}_{\Bb \beta}(z)|+|\mathcal{C}_{\beta}(z)-\mathcal{C}_{\Bb \beta}(z)|\label{eq_bbbbb} \\
\leq\,&  C \mathbf{d}+C' |m_{\mu_{\beta}}(\Omega_{\alpha}(z))-m_{{\Bb}}(\Omega_{\alpha}(z))|\,.\nonumber
\end{align}
We claim that for some constant $C^{''}>0,$ we have 
\begin{equation}\label{eq_qqq}
|m_{\mu_{\beta}}(\Omega_{\alpha}(z))-m_{{\Bb}}(\Omega_{\alpha}(z))| \leq C'' \mathbf{d}\,.
\end{equation}
Indeed, note that it follows from Lemma \ref{lem:rbound}  by letting $f(x):=\frac{1}{x-\Omega_{\alpha}(z)}.$  Here we also use (iv) of Assumption \ref{ass_product2} and Lemma \ref{lem_bao}. Therefore, we find that $|\mathcal{C}_{\Bb}(z)-\mathcal{C}_{\beta}(z)| \leq C \mathbf{d}$ for some constant $C>0$.   By (i) of Lemma \ref{lem_bao}, we get the claim (\ref{eq_bbbbbbbbdd}).

For the second term, we claim that
\begin{equation}\label{eq_claimtobeproved}
|L'_{{\Bb}}(\Omega_{\Ab}(z))-L'_{{\Bb}}(\Omega_{\alpha}(z))| \leq C|\Omega_{\Ab}(z)-\Omega_{\alpha}(z)|
\end{equation}
for some constant $C>0$.  
By a direct expansion, we have 
\begin{align}
&L_{\Bb}'(\Omega_{\Ab}(z))-L'_{\Bb}(\Omega_{\alpha}(z))\label{eq_originalexpansion}\\
=&\,\int \frac{\mathsf{C}_{\Ab}(z)x^2}{(x-\Omega_{\Ab}(z))^2} - \frac{\mathsf{C}_{\alpha}(z) x^2}{(x-\Omega_{\alpha}(z))^2} d\mu_{\Bb}(x)\nonumber\\
=&\,\int \frac{x^2 \left[ \mathsf{C}_{\Ab}(z) \Pi_1-(\mathsf{C}_{\alpha}(z)-\mathsf{C}_{\Ab}(z)) (x-\Omega_{\Ab}(z))^2  \right] }{(x-\Omega_{\Ab}(z))^2 (x-\Omega_{\alpha}(z))^2} d\mu_{\Bb}(x)\,\nonumber
\end{align}
where
\begin{equation*}
\mathsf{C}_\Ab(z):=\frac{1}{\Omega_\Ab^2(z)} \frac{1}{(\Omega_{\Ab}(z) m_{\mu_\Bb}(\Omega_{\Ab}(z))+1)^2}\,, 
\end{equation*}
$\mathsf{C}_{\alpha}$ is similarly defined by replacing $\Omega_{\Ab}(z)$ with $\Omega_{\alpha}(z)$, and  
\begin{equation}
\Pi_z(z) \equiv \Pi_1(z):=(x-\Omega_{\alpha}(z))^2-(x-\Omega_{\Ab}(z))^2\,.
\end{equation}
By a discussion similar to (\ref{eq_bbbbbbbbdd}) (recall that $\mathcal{C}_{\beta}(z)$ is bounded),  we find that both $\mathsf{C}_{\Ab}(z)$ and $\int \frac{x^2}{(x-\Omega_{\Ab}(z))^2 (x-\Omega_{\alpha}(z))^2} d \mu_{\Bb}(x)$ are bounded. 
By (i) of Proposition \ref{lem_addtivemeasurelemma}, (\ref{eq_suborclose}) and a discussion similar to (\ref{bbbb_bound}), when $n$ is sufficiently large, there exists some constant $C_1>0$ such that $|\mathsf{C}_{\Ab}| \leq C_1$ and $|\mathsf{C}_{\alpha}| \leq C_1.$ 
Moreover, 
\begin{align*}
\mathsf{C}_{\Ab}-\mathsf{C}_{\alpha}
&=\,\mathsf{C}_{\Ab} \mathsf{C}_{\alpha}  \Big( \Omega_{\alpha}^2(z) \Big[ (\Omega_{\alpha}(z) m_{\Bb}(\Omega_{\alpha}(z))+1)^2  -(\Omega_{\Ab}(z) m_{\Bb}(\Omega_{\Ab}(z))+1)^2 \Big] \\
&-[\Omega_{\Ab}^2(z)-\Omega_{\alpha}^2(z)](\Omega_{\Ab}(z) m_{\Bb}(\Omega_{\Ab}(z))+1)^2 \Big)\,, 
\end{align*}
which is controlled by $|\mathsf{C}_{\Ab}-\mathsf{C}_{\alpha}| \leq C_2 |\Omega_{\Ab}(z)-\Omega_{\alpha}(z)|$ for some constant $C_2>0$.
By putting the above together, we have shown (\ref{eq_claimtobeproved}).

By the above two bounds, \eqref{eq_bbbbbbbbdd} and \eqref{eq_claimtobeproved}, we have
	\begin{equation}\label{Control LB'OmegaA and Lbeta'Omegaalpha difference}
|L'_{{\Bb}}(\Omega_{\Ab}(z))-L'_{{\beta}}(\Omega_{\alpha}(z))| \leq C\mathbf{d}+|\Omega_{\Ab}(z)-\Omega_{\alpha}(z)|
\end{equation}
for some $C>0$, which, by \eqref{Control of Omega beta and Omega B} and Assumption \ref{ass_product2}, is sufficiently small when $n$ is sufficiently large. Indeed, note that in $\mathcal{S}'$, $\sqrt{\eta+\kappa}\geq n^{-1/3+\tau/2}$, and hence $|\Omega_{\Ab}(z)-\Omega_{\alpha}(z)|\leq n^{-1/3+\epsilon-\tau/2}\to 0$ when $n\to \infty$.
A direct consequence is that when $n$ is sufficiently large, 
\begin{equation}\label{Control L'A and L'alpha are close}
L'_{\Ab}(\Omega_{\Bb}(z))\sim L'_{{\alpha}}(\Omega_{\beta}(z))\,.
\end{equation}
To finish the proof, note that
\begin{align}
&|S_{\Ab \Bb}(z)-S_{\alpha \beta}(z)|\label{eq_sabaplhabeta}\\
& =|z^2| \big| L'_{\Ab}(\Omega_{\Bb}) [L_{\Bb}'(\Omega_{\Ab})-L'_{\beta}(\Omega_{\alpha})]  \nonumber \\
&+L'_{\beta}(\Omega_{\alpha})[L'_{\Ab}(\Omega_\Bb)-L_{\alpha}'(\Omega_{\beta})] \big|\,.\nonumber
\end{align}	
By the boundedness assumption of $z$, \eqref{Control LB'OmegaA and Lbeta'Omegaalpha difference}, \eqref{Control L'A and L'alpha are close}, and (ii) of Lemma \ref{Lemma Omega alpha, S alphabeta Talpha control}, we conclude that $S_{\Ab\Bb}(z)\sim S_{\alpha\beta}(z)\sim\sqrt{\kappa+\eta}$.

The bounds for $\mathcal{T}_{\Ab}$ and $\mathcal{T}_{\Bb}$ follow the same argument for $S_{\Ab \Bb}$, so we only indicate the key steps and omit the details. In brief, the upper bounds for $\mathcal{T}_{\Ab}$ and $\mathcal{T}_{\Bb}$ follow directly from Lemmas \ref{lem_bao} and \ref{lem_lasttwo} together with their definitions. For the lower bounds, we use the same argument as in (\ref{eq_sabaplhabeta}) to obtain 
\begin{equation*}
|\mathcal{T}_{\alpha}(z)-\mathcal{T}_{\Ab}(z)|=o(1)\,,
\end{equation*} 
so that $\mathcal{T}_{\Ab}(z)\sim 1$ as $\mathcal{T}_{\alpha}(z)\sim 1$ from (\ref{eq_talpha}). 
	
Finally,	we prove (iv). Applying $\frac{d}{d z}$ to the subordination equations \eqref{eq)system} with $\alpha$ and $\beta$ replaced by $\Ab$ and $\Bb$, we have
	\begin{equation*}
	\begin{pmatrix}
		zL'_{{\Bb}}(\Omega_{\Ab}(z)) & -1 \\
		-1 & zL'_{{\Ab}}(\Omega_{\Bb}(z))
	\end{pmatrix}
	\begin{pmatrix}
		\Omega_{\Ab}'(z) \\
		\Omega_{\Bb}'(z) 
	\end{pmatrix}
	=-\frac{1}{z}\begin{pmatrix}
		\Omega_{\Ab}(z) \\
		\Omega_{\Bb}(z)
	\end{pmatrix}.
	\end{equation*}
	Noting that since the determinant of the $(2\times 2)$-matrix on the left-hand side is $S_{\Ab\Bb}(z)$, we have
	\begin{align}
		\begin{pmatrix}
			\Omega_{\Ab}'(z) \\
			\Omega_{\Bb}'(z) 
		\end{pmatrix}
	=&\frac{1}{z S_{\Ab\Bb}(z)}\label{eq_sysysysys}\\
	&\times
	\begin{pmatrix}
		zL'_{{\Ab}}(\Omega_{\Bb}(z)) & 1 \\
		1 & zL'_{{\Bb}}(\Omega_{\Ab}(z))
	\end{pmatrix}
	\begin{pmatrix}
		\Omega_{\Ab}(z) \\
		\Omega_{\Bb}(z)
	\end{pmatrix}.\nonumber
	\end{align}
	We claim that the entries of the matrix and the vector on the right-hand side can be bounded by combining Lemmas \ref{lem_bao} and \ref{lem_lasttwo}. Specifically, by (i) of Lemma \ref{lem_bao} and (i) of Lemma \ref{lem_lasttwo}, we find that $\Omega_{\Ab}$ and $\Omega_{\Bb}$ are bounded from above and below. Recall that by \eqref{Control L'A and L'alpha are close}, $L_{\Ab}'(\Omega_{\Bb}(z))$ is bounded. Further, $z$ is bounded since $z \in \mathcal{S}'$. Therefore, $S_{\Ab \Bb}(z)\sim\sqrt{\kappa+\eta}$ implies 
	\begin{equation*}
	\Omega_{\Ab}'(z),\Omega_{\Bb}'(z)\lesssim (\kappa+\eta)^{-1/2}
	\end{equation*} 
	by (\ref{eq_sysysysys}). Finally, for $S'_{\Ab\Bb}(z)$, by definition, we immediately have
	\begin{align}
	S_{\Ab\Bb}'(z) =&\,\frac{2}{z}(S_{\Ab\Bb}(z)+1)+z^{2}L_{{\Bb}}''(\Omega_{\Ab}(z))L_{{\Ab}}'(\Omega_{\Bb}(z))\Omega_{\Ab}'(z)\nonumber\\
	& +z^{2}L_{{\Bb}}'(\Omega_{\Ab}(z))L_{{\Ab}}''(\Omega_{\Bb}(z))\Omega_{\Bb}'(z)\,.\nonumber
	\end{align}
To control the first term, since $z, \kappa$ and $\eta$ are bounded above, by (iii) of Lemma \ref{lem_addtivemeasurelemma}, we find that
\begin{equation*}
\frac{2}{z}(S_{\Ab \Bb}+1)=O(1)\,. 
\end{equation*}
For the other terms, by a discussion similar to the proof of (ii) and (iii) using the first and second derivatives of $L$ and Lemma \ref{lem:suborder}, we find that both $z^{2}L_{{\Bb}}''(\Omega_{\Ab}(z))L_{{\Ab}}'(\Omega_{\Bb}(z))$ and $L_{{\Bb}}'(\Omega_{\Ab}(z))L_{{\Ab}}''(\Omega_{\Bb}(z))$ are bounded.  Then we conclude the proof using  the first two terms in (iv) of Lemma \ref{lem_addtivemeasurelemma}. 
\end{proof}
\begin{remark}\label{Remark of control of Omega beta and Omega B}
We mention that when $\epsilon_L$ defined in (\ref{eq_levy}) is sufficiently small so that $\epsilon_L=\epsilon<\tau$, \eqref{Control of Omega beta and Omega B} can be further bounded so that the proof of the local law
\cite[equation (2.18)]{bao20171} can be directly carried out; that is, we can bound $\frac{n^{-2/3+\epsilon}}{\sqrt{\kappa+\eta}}$ by
\begin{equation*}
\frac{n^{-2/3+\epsilon}}{\sqrt{\kappa+\eta}} \lesssim \left\{
		\begin{array}{ll}
			\sqrt{\kappa+\eta}			&	\text{ if }E\in \operatorname{supp} \mu_{\Ab} \boxtimes \mu_{\Bb} \,,\\
			\frac{\eta}{\sqrt{\kappa+\eta}}	&	\text{ if }E > \supp \operatorname{supp} \mu_{\Ab} \boxtimes \mu_{\Bb} 
		\end{array}
		\right.
	\end{equation*}
	so that
	\begin{equation*}
		|\Omega_{\beta}(z)-\Omega_{\Bb}(z)| \lesssim \left\{
		\begin{array}{ll}
			\sqrt{\kappa+\eta}			&	\text{ if }E\in \operatorname{supp} \mu_{\Ab} \boxtimes \mu_{\Bb} \,,\\
			\frac{\eta}{\sqrt{\kappa+\eta}}	&	\text{ if }E > \supp \operatorname{supp} \mu_{\Ab} \boxtimes \mu_{\Bb} \,.
		\end{array}
		\right.
	\end{equation*}
Indeed, if $E$ is away from the right-edge of the support of $\mu_{\Ab} \boxtimes \mu_{\Bb}$ or into the bulk of $[\lambda_-, \lambda_+]$, by (\ref{eq_edgedgedge}), we indeed have that $\kappa=O(1)$, since $n^{-2/3+\epsilon} \leq \eta$ due to the assumption that $\epsilon_L<\tau$, the estimates follow. If $E$ is near the edge, i.e. $\kappa$ is of order $n^{-2/3}$ or smaller, since $\eta=O(n^{-2/3})$, it holds true. The other cases hold by the same argument. 
\end{remark}

Finally, we sketch the proof of Lemma \ref{lem_lasttwo}.
\begin{proof}[Sketch proof of Lemma \ref{lem_lasttwo}]

The proof strategy is similar to that of \cite[Lemma 3.12]{bao20171}, which relies on the following two lemmas. We start with some specific $z$ when $\eta$ is large and use a bootstrapping argument to extend to the whole domain $\mathcal{D}$.  
We omit further details but only list the lemmas for the reader's convenience. 

\begin{lemma}\label{lem_llll}
There exists $\eta_0>0$ so that for $z \in \mathbb{C}^+$ with $\text{Re} z \in [\lambda_+-\tau, \tau^{-1}]$ and $\operatorname{Im} z =\eta_0$, we have 
\begin{align}
&|\Omega_{\alpha}(z)-\Omega_{\Ab}(z)| \leq 2 \| r(z) \|\,,\\
&|\Omega_{\beta}(z)-\Omega_{\Bb}(z)| \leq 2 \| r(z) \|\, ,\nonumber
\end{align}
where $r(z):=(r_\Ab(z), r_\Bb(z))$ and $r_\Ab(z)$ and $r_\Bb(z)$ are defined as 
\begin{align}
r_\Ab(z)&:=\frac{M_{\alpha}(\Omega_{\Bb}(z))-M_{\Ab}(\Omega_{\Bb}(z))}{\Omega_\Bb(z)}\,,\\
r_\Bb(z)&:=\frac{M_{\beta}(\Omega_{\Ab}(z))-M_{\Bb}(\Omega_{\Ab}(z))}{\Omega_\Ab(z)},\nonumber
\end{align}
respectively.
\end{lemma}  
The proof of the first lemma is similar to \cite[Lemma 4.2]{BAO2016} and we omit the details here.
The second lemma provides the bootstrapping part. Its proof is similar to \cite[Lemma 3.13]{bao20171} and we omit the details here. 
\begin{lemma}
Consider $z_0=E+\mathrm{i} \eta_0 \in \mathcal{D}$, where $\eta_0$ is described in Lemma \ref{lem_llll}. Suppose there exists a function $q=q(z)>0$ so that 
\begin{equation}
|\Omega_{\Ab}(z_0)-\Omega_{\alpha}(z_0)| \leq q,  \  |\Omega_{\Bb}(z_0)-\Omega_{\beta}(z_0)| \leq q,
\end{equation}
with $q=o(1)$ and $S_{\alpha \beta}(z_0) q=o(1)$ when $n\to \infty$ uniformly in $z\in \mathcal{D}$. Then, we have 
\begin{equation*}
|\Omega_{\Ab}(z_0)-\Omega_{\alpha}(z_0)|+|\Omega_{\Bb}(z_0)-\Omega_{\beta}(z_0)| \leq \frac{C\| r(z_0) \|}{|S_{\alpha \beta}(z_0)|},
\end{equation*}
with some constant $C>0$ independent of $n$ and $z_0$.
\end{lemma} 

 \end{proof}

\section{Proof of Lemma \ref{lem_proofrigidityest}}\label{OS:Proof:Lemma5.9}

In this section, we prove the rigidity estimate Lemma \ref{lem_proofrigidityest}. We will closely follow the proof strategy of \cite[Lemma 3.14]{bao20171}. The key ingredient is to make use of Proposition \ref{pro_key} to translate the closeness of Stieltjes transforms of two measures into closeness of their quantiles. Note that we are only interested in proving the closeness near the edge and hence the proof is easier than that of \cite[Lemma 3.14]{bao20171}.
We validate in the end of the proof of Lemma \ref{lem_2.4.3} that Assumptions \ref{assu_product1} and \ref{ass_product2} are satisfied for the setup of Lemma \ref{lem_proofrigidityest}.

\begin{proof}[Proof of Lemma \ref{lem_proofrigidityest}] 

Near the edge, by using the square root behavior of 
$\nu_{xy}=\nu_{x} \boxtimes \nu_{y}$ near $\lambda_+$ in Theorem \ref{lem_bao2}, we know that there are at most { $n^{3 \epsilon/2}$} $\gamma_{\nu_{xy}}(j)$-quantiles in an $n^{-2/3+\epsilon}$ vicinity of $\lambda_+$, where $\epsilon=\epsilon_L$ is defined in (\ref{eq_levy}). 
By a discussion similar to \cite[equation (3.84)]{bao20171}, we find that 
\begin{equation}\label{eq_specupper}
\sup( \text{supp} \mu_{\Sigma_x} \boxtimes \mu_{\Sigma_y}) \leq \lambda_++n^{-2/3+10\epsilon}.
\end{equation} 
Again, by the square root behavior of $\nu_{xy}$ and the definition of typical locations in Definition \ref{Definition typical locations}, when $ j \leq L$, we have that 
\begin{equation}\label{eq_bbbblowlow}
|\gamma_{\nu_{xy}}(j) - \lambda_+| \leq c_0n^{-2/3},
\end{equation} 
for some constant $c_0>0$. Indeed,
for $\gamma_{\nu_{xy}}(j)$, when $n$ is sufficiently large, we have
\begin{align}
\frac{j-1/2}{n} =\int_{\gamma_{\nu_{xy}}(j)}^{\infty} d \mu_{\alpha} \boxtimes \mu_{\beta}=\int_{\gamma_{\nu_{xy}}(j)}^{\lambda_+} d\mu_{\alpha} \boxtimes \mu_{\beta} \sim (\lambda_+-\gamma_{\nu_{xy}}(j))^{3/2},  \nonumber 
\end{align}
where we use a discussion similar to (\ref{proof theorem lambdai right most edge gap}).

Since we have  
\begin{align}
|\gamma_{\nu_{xy}}(j)-\gamma_{\Sigma_x \boxtimes \Sigma_y}(j)| \leq  |\gamma_{\nu_{xy}}(j)-\lambda_+|+|\gamma_{\Sigma_x \boxtimes \Sigma_y}(j)-\lambda_+| \,,\label{eq_twotwo}
\end{align}
it suffices to control $| \gamma_{\Sigma_x \boxtimes \Sigma_y}(j)-\lambda_+|$ in order to finish the proof. On one hand, when $\gamma_{\Sigma_x \boxtimes \Sigma_y}(j) \geq \lambda_+$, we have the conclusion directly by (\ref{eq_specupper}). On the other hand, when $\gamma_{\Sigma_x \boxtimes \Sigma_y}(j)< \lambda_+$, we prove the result using the assumption on the Levy distance.  
By the (vi) of Lemma \ref{lem_collecttion} and \eqref{eq_levy} in Assumption \ref{ass_product2}, we get 
\begin{equation*}
\mathcal{L}(\mu_{\Sigma_x} \boxtimes \mu_{\Sigma_y}, \nu_x \boxtimes \nu_y) \leq \mathcal{L}(\mu_{\Sigma_x}, \nu_x)+\mathcal{L}(\mu_{\Sigma_y}, \nu_y) \leq  n^{-2/3+\epsilon}\,,
\end{equation*} 
which, by a discussion similar to \cite[equation (3.104)]{bao20171}, leads to
\begin{equation}\label{Proof:Lemma5.9:eq_boundfinal}
|\mu_{\Sigma_x} \boxtimes \mu_{\Sigma_y} ((x, \infty))-\nu_{xy}((x, \infty))| \leq C n^{-2/3+\epsilon} \,,
\end{equation}
for $x<\lambda_+${, where $C>0$ is a constant}. Under the condition $\gamma_{\Sigma_x \boxtimes \Sigma_y}(j)<\lambda_+$,  by (\ref{eq_specupper}) and the definition of $\gamma_{\Sigma_x \boxtimes \Sigma_y}(j)$, we have  
\begin{equation*}
\mu_{\Sigma_x} \boxtimes \mu_{\Sigma_y} ((\gamma_{\Sigma_x \boxtimes \Sigma_y}(j), \infty))=\frac{j-1/2}{n}\,.
\end{equation*}
Similarly, by the definition of $\lambda_+$, we have 
\begin{equation*}
\nu_{xy}((\gamma_{\Sigma_x \boxtimes \Sigma_y}(j), \lambda_++n^{-2/3+10\epsilon}))=\nu_{xy}((\gamma_{\Sigma_x \boxtimes \Sigma_y}(j),\infty))\,. 
\end{equation*}
As a consequence, when $j \leq L$, by \eqref{Proof:Lemma5.9:eq_boundfinal}, we obtain
\begin{equation*}
\nu_{xy}((\gamma_{\Sigma_x \boxtimes \Sigma_y}(j), \lambda_++n^{-2/3+10\epsilon})) \leq C n^{-2/3+\epsilon}\,,
\end{equation*}
By the second fact of Lemma \ref{lem_bao2}, i.e. the square root behavior of $\nu_{xy}$, we find that
\begin{align}
\nu_{xy}((\gamma_{\Sigma_x \boxtimes \Sigma_y}(j), \lambda_++n^{-2/3+10\epsilon}))&=\nu_{xy}((\gamma_{\Sigma_x \boxtimes \Sigma_y}(j), \lambda_+)) \nonumber\\
&\sim  (\lambda_+-\gamma_{\Sigma_x \boxtimes \Sigma_y}(j))^{3/2}\,,\nonumber
\end{align}
which leads to $|\gamma_{\Sigma_x \boxtimes \Sigma_y}(j)-\lambda_+|\leq Cn^{-4/9+2\epsilon/3}$.
This concludes our proof.

\end{proof}

\section{Verification of Remarks \ref{remk_rates} and \ref{eq_firstthree}}\label{Section validation remark eq_firstthree}
{
\subsection{Proof of Remark \ref{remk_rates}}
We prove (\ref{eq_rates}) by contradiction. Suppose that there exists an $i$ such that   
\begin{equation}\label{eq_lowerboundone}
\left| \lambda_{i+1}(\Wb_x)-\gamma_{\nu_x}(i) \right| \leq \frac{C \log n}{n}.
\end{equation}
Let $z=\lambda_{i+1}(\Wb_x)+\mathrm{i}\eta_0,$ where $\eta_0=n^{-1+\delta}$ and $\delta>0$ is an arbitrary small constant. By Lemma \ref{lem_mplaw}, it is easy to see that 
\begin{equation}\label{eq_imim}
\operatorname{Im} m_{\nu_x}(z) \sim \sqrt{|\lambda_{i+1}(\Wb_x)-\gamma_{\nu_x}(1)|+{\eta_0}}.  
\end{equation}
First of all, we consider the case $i=O(n^{1-\theta})$, where $\theta>0.$ Since $\gamma_{\nu_x}(i)-\gamma_{\nu_x}(i+1) \sim n^{-1},$ by Corollary \ref{coro_lesspertub} and (\ref{eq_imim}), we find that 
\begin{equation*}
\operatorname{Im} m_{\nu_x}(z) \leq n^{-\theta/2}. 
\end{equation*}
On the other hand, by the definition of Stieltjes's transform, we deduce that 
\begin{equation*}
\operatorname{Im} m_{\nu_x}(z) \geq \frac{1}{n \eta_0}=n^{-\delta}.
\end{equation*}
Due to the arbitrariness of $\delta,$ it is a contradiction if we let $\delta=\frac{\theta}{4}.$ 

Second, we consider the case when $\theta=0$ and {$i$ is of order $n$}. Since { $\lambda_j(\Wb_x)$, $j=2,3,\cdots, n,$} are bounded with high probability, we {claim that for any integer $i_0=o(n)$, $\lambda_{j}(\Wb_x)-\lambda_{j+1}(\Wb_x) \leq C n^{-1}\sqrt{\log n}$ for $ i_0 \leq j \leq i_0+C_1n,$ for some $C_1 \leq 1.$ Indeed, note that if $\lambda_{j}(\Wb_x)-\lambda_{j+1}(\Wb_x) \asymp C n^{-1}\sqrt{\log n}$, we have $\lambda_{i_0}(\Wb_x)-\lambda_{i_0+C_1 n}(\Wb_x) \asymp \sqrt{\log n},$ which is a contradiction to the boundedness of $\lambda_j(\Wb_x)$.} 
 { Consequently, (\ref{eq_lowerboundone}) will hold for $i_0 \leq j \leq i_0+C_1 n$ as $\gamma_{\nu_x}(j)-\gamma_{\nu_x}(j+1) \sim n^{-1}.$ } 
  Let $z_0=\lambda_{2}(\Wb_x)+\mathrm{i} \eta_0,$ where $\eta_0$ is to be chosen later.  One one hand, by Lemma \ref{lem_mplaw} and Corollary \ref{coro_lesspertub}, we find that
\begin{equation}\label{eq_bbbbbbbbbbbdddddd}
\operatorname{Im} m_{\nu_x}(z_0) \geq C_2 \sqrt{n^{-1/9+2 \vartheta}+\eta_0},
\end{equation}
for some constant $C_2>0.$ On the other hand, let $m_1(z)=\frac{1}{n}\sum_{k=1}^{n-1} \frac{1}{\lambda_{k+1}(\Wb_x)-z},$ { by the definition of Stieltjes transform and the assumption $i_0=o(n),$}
we have 
{
\begin{equation}
\left| m_1(z_0)-m_{\nu_x}(z_0) \right| \leq \left|  \frac{C}{n} \sum_{k=i_0}^{i_0+C_1n} \frac{|\lambda_k(\Wb_x)-\gamma_{\nu_x}(k)|}{\eta_0^2} \right| \leq C \frac{\log n}{n \eta_0^2}.
\end{equation}
}
We now choose { $\eta_0>1$} to be a sufficiently large constant and $|m_1(z_0)-m_{\nu_x}(z_0)| \leq C n^{-1}\log n \eta_0^{-1}.$ 
Consequently, { by (\ref{eq_bbbbbbbbbbbdddddd}),} we shall have that $\operatorname{Im} m_{1}(z_0) \geq C_2 \sqrt{\eta_0}.$ 
 However, by (\ref{eq_stitrivial}), we find that 
\begin{equation*}
\operatorname{Im} m_1(z) \leq  \frac{1}{\eta_0}. 
\end{equation*}
This leads to a contradiction { as $\eta_0>1$ is a large constant}. 

}

\subsection{Proof of Remark \ref{eq_firstthree}}
First of all, we get from the definition of $\mathbf{S}_{xy}$ that
\begin{align}
\mathbf{S}_{xy}=&\,n^{-2} (n^{-1} \Db_x)^{-1} \mathrm{Sh}_0^2 (n^{-1} \Db_y)^{-1}\nonumber\\
&+n^{-2}(n^{-1} \Db_x)^{-1} (\mathbf{W}_x-\mathrm{Sh}_0)(\mathbf{W}_y-\mathrm{Sh}_0)  (n^{-1} \Db_y)^{-1}\nonumber \\
& + n^{-2}(n^{-1} \Db_x)^{-1} \mathrm{Sh}_0(\mathbf{W}_y-\mathrm{Sh}_0)  (n^{-1} \Db_y)^{-1} \nonumber\\
& +n^{-2}(n^{-1} \Db_x)^{-1} (\mathbf{W}_x-\mathrm{Sh}_0)\mathrm{Sh}_0  (n^{-1} \Db_y)^{-1}\,.\nonumber
\end{align}
By Weyl's inequality and a discussion similar to (\ref{eq_weylkeybound}), we conclude from Theorem \ref{lem_rigiditydistance} that the second part of the above equation satisfies
\begin{align}
&\norm{n^{-2}(n^{-1} \Db_x)^{-1} (\mathbf{W}_x-\mathrm{Sh}_0)(\mathbf{W}_y-\mathrm{Sh}_0)  (n^{-1} \Db_y)^{-1}} =O_{\prec}( n^{-4})\,,\nonumber\\
&\norm{n^{-2}(n^{-1} \Db_x)^{-1} \mathrm{Sh}_0(\mathbf{W}_y-\mathrm{Sh}_0)  (n^{-1} \Db_y)^{-1}}=O_{\prec}(n^{-3/2})\,,\nonumber\\
&\norm{n^{-2}(n^{-1} \Db_x)^{-1} (\mathbf{W}_x-\mathrm{Sh}_0)\mathrm{Sh}_0  (n^{-1} \Db_y)^{-1}}=O_{\prec}(n^{-3/2})\,,\nonumber
\end{align}
where we use the fact that the top eigenvalue of $\mathrm{Sh}_0$ is $f(2)n$.  We also conclude from Lemma \ref{lem_diagonal} that
\begin{equation*}
\lambda_1(n^{-2}  (n^{-1} \Db_x)^{-1} \mathrm{Sh}_0^2 (n^{-1} \Db_y)^{-1}) \rightarrow 1
\end{equation*}
in probability. As a consequence, we  conclude from Weyl's inequality that
\begin{equation*}
\lambda_1( \mathbf{S}_{xy}) \rightarrow 1\,.
\end{equation*}

\end{document}